\documentclass[reqno,10pt,centertags]{amsart}
\usepackage{amsmath,amsthm,amscd,amssymb,latexsym,upref,mathrsfs}
\usepackage{color}

\usepackage{hyperref}
\newcommand*{\mailto}[1]{\href{mailto:#1}{\nolinkurl{#1}}}
\newcommand{\arxiv}[1]{\href{http://arxiv.org/abs/#1}{arXiv:#1}}

\newcommand{\bbN}{{\mathbb{N}}}
\newcommand{\bbR}{{\mathbb{R}}}

\newcommand{\bbC}{{\mathbb{C}}}

\newcommand{\dC}{{\mathbb{C}}}
\newcommand{\dR}{{\mathbb{R}}}
\newcommand{\dN}{{\mathbb{N}}}

\newcommand{\cA}{{\mathcal A}}
\newcommand{\cB}{{\mathcal B}}
\newcommand{\cC}{{\mathcal C}}
\newcommand{\cD}{{\mathcal D}}
\newcommand{\cE}{{\mathcal E}}

\newcommand{\cG}{{\mathcal G}}
\newcommand{\cH}{{\mathcal H}}

\newcommand{\cK}{{\mathcal K}}
\newcommand{\cL}{{\mathcal L}}
\newcommand{\cM}{{\mathcal M}}
\newcommand{\cN}{{\mathcal N}}
\newcommand{\cO}{{\mathcal O}}

\newcommand{\cS}{{\mathcal S}}

\newcommand{\sH}{{\mathfrak H}}
\newcommand{\sS}{{\mathfrak S}}
\newcommand{\sN}{{\mathfrak N}}

\newcommand{\no}{\notag}
\newcommand{\lb}{\label}

\newcommand{\ol}{\overline}

\newcommand{\f}{\frac}

\def\senki{{\lbrack\negthinspace [\bot ]\negthinspace\rbrack}}
\def\senki+{{\lbrack\negthinspace [+] \negthinspace\rbrack}}

\renewcommand{\Re}{\mathop\mathrm{Re}}
\renewcommand{\Im}{\mathop\mathrm{Im}}
\renewcommand{\ge}{\geqslant}

\DeclareMathOperator{\dom}{dom}
\DeclareMathOperator{\ran}{ran}
\DeclareMathOperator{\tr}{tr}

\allowdisplaybreaks \numberwithin{equation}{section}

\newtheorem{theorem}{Theorem}[section]
\newtheorem{proposition}[theorem]{Proposition}
\newtheorem{lemma}[theorem]{Lemma}
\newtheorem{corollary}[theorem]{Corollary}
\newtheorem{definition}[theorem]{Definition}
\newtheorem{hypothesis}[theorem]{Hypothesis}
\newtheorem{example}[theorem]{Example}
\theoremstyle{remark}
\newtheorem{remark}[theorem]{Remark}

\begin{document}

\title[Spectral shift functions and Dirichlet-to-Neumann maps]{Spectral shift functions and Dirichlet-to-Neumann maps}

\author[J.\ Behrndt]{Jussi Behrndt}
\address{Institut f\"ur Numerische Mathematik, Technische Universit\"at
Graz, Steyrergasse 30, 8010 Graz, Austria}
\email{\mailto{behrndt@tugraz.at}}
\urladdr{\url{http://www.math.tugraz.at/~behrndt/}}

\author[F.\ Gesztesy]{Fritz Gesztesy} 
\address{Department of Mathematics
Baylor University, One Bear Place \#97328,
Waco, TX 76798-7328, USA}
\email{\mailto{Fritz\_Gesztesy@baylor.edu}}
\urladdr{\url{http://www.baylor.edu/math/index.php?id=935340}}

\author[S.\ Nakamura]{Shu Nakamura}
\address{Graduate School of Mathematical Sciences, University of Tokyo
3-8-1, Komaba, Meguro-ku, Tokyo, Japan 153-8914}
\email{\mailto{shu@ms.u-tokyo.ac.jp}}
\urladdr{\url{http://www.ms.u-tokyo.ac.jp/~shu/}}


\date{\today}



\begin{abstract}
The spectral shift function of a pair of self-adjoint operators is expressed via an abstract operator valued Titchmarsh--Weyl $m$-function. 
This general result is applied to different self-adjoint realizations of 
second-order elliptic partial differential operators on smooth domains with compact boundaries, 
Schr\"{o}dinger operators with compactly supported potentials, and finally, Schr\"{o}dinger operators with singular potentials supported on hypersurfaces. In these applications
the spectral shift function is determined in an explicit form with the help of (energy parameter 
dependent) Dirichlet-to-Neumann maps. 
\end{abstract}

\maketitle

{\scriptsize{\tableofcontents}}

\section{Introduction}

Let $A$ and $B$ be self-adjoint operators in a separable Hilbert space $\sH$ and assume that the $m$-th powers of their resolvents differ by a trace class operator,
\begin{equation}\label{resabmmm}
\big[(B- z I_{\sH})^{-m}-(A- z I_{\sH})^{-m}\big]\in\sS_1(\sH),\quad z\in\rho(A)\cap\rho(B),
\end{equation}
for some odd integer $m\in\dN$. It is known that in this case there exists a real-valued function $\xi\in L^1_{\rm loc}(\dR)$ such that 
$\int_\dR \vert\xi(\lambda)\vert(1+\vert \lambda \vert)^{-(m+1)}d\lambda < \infty$ 
and the trace formula
\begin{equation}\label{trab}
 \tr_{\sH}\bigl(\varphi(B)-\varphi(A)\bigr)=\int_\dR\varphi^\prime(\lambda) \, \xi(\lambda)\,d\lambda
\end{equation}
holds for all suitable smooth functions $\varphi:\dR\rightarrow\dC$ such that 
$[\varphi(B)-\varphi(A)] \in\sS_1(\sH)$. 
The function $\xi$ in \eqref{trab} is called a {\it spectral shift function} of the pair $\{A,B\}$. Note that for 
$\varphi(\lambda)=(\lambda - z)^{-m}$ one has $[\varphi(B)-\varphi(A)]\in\sS_1(\sH)$ according to \eqref{resabmmm} and the trace formula \eqref{trab} takes the special form
\begin{equation*}
\tr_{\sH}\bigl((B- z I_{\sH})^{-m}-(A- z I_{\sH})^{-m}\bigr) 
= -m \int_\dR \frac{\xi(\lambda)\,d\lambda}{(\lambda - z)^{m+1}}.
\end{equation*}

Historically the trace formula \eqref{trab} was first proposed and verified on a formal level by I.\,M. Lifshitz for the case that $[B-A]$ is a finite-rank operator in \cite{L52},
and shortly afterwards in \cite{K53} M.\,G.~Krein proved \eqref{trab} rigorously in the more general case $[B-A]\in\sS_1(\sH)$ 
for all $C^1$-functions $\varphi$ with derivatives in the Wiener class. Furthermore, in \cite{K53} it was shown how 
the spectral shift function $\xi$ can be computed with the help of the perturbation determinant corresponding to the pair $\{A,B\}$. For pairs of unitary operators and thus via Cayley transforms for the case $m=1$ in
\eqref{resabmmm} the spectral shift function and the trace formula were obtained later by M.\,G.~Krein in \cite{K62}. 
Afterwards in \cite{K71} the more general case $m>0$ 
in \eqref{resabmmm} for self-adjoint operators $A$ and $B$ with $\rho(A)\cap\rho(B)\cap\dR\not=\emptyset$ was discussed by L.\,S.~Koplienko, and for odd integers $m$ in \eqref{resabmmm} 
and arbitrary self-adjoint operators $A$ and $B$ see \cite{Y05} by D.\,R.~Yafaev or \cite[Chapter 8, $\S$11]{Y92} and \cite[Chapter 0, Theorem 9.4]{Y10}.
For more details on the history, development and multifaceted applications
of the spectral shift function in mathematical analysis we refer the reader to the survey papers 
\cite{BP98,BY92,BY92-1}, the standard monographs \cite{Y92,Y10}, and, for instance, to 
\cite{BS75,C76,GHS95,GMN99,GM00,GS96,K87,KY81,P97,S98} and the more recent contributions \cite{ACDS09,GLMST11,GN12,HM10,HKNSV06,KKN13,KT09,MN14,MN15,PSS13,PSS14,PS12,P08,Y07}.

The main objective of the present paper is to prove a representation formula for the spectral shift function in terms of an abstract Titchmarsh--Weyl $m$-function 
of two self-adjoint operators satisfying the condition \eqref{resabmmm}, and to apply this result to different self-adjoint realizations of
second-order elliptic PDEs and Schr\"{o}dinger operators with compactly supported potentials and singular $\delta$-type potentials supported on compact hypersurfaces. In these applications 
the abstract Titchmarsh--Weyl $m$-function will turn out to be the energy dependent Neumann-to-Dirichlet map or Dirichlet-to-Neumann 
map associated to the elliptic differential expression
and the Schr\"{o}dinger operators on the interior and exterior domains, respectively.

More precisely, assume that $A$ and $B$ are self-adjoint operators in a separable Hilbert space $\sH$ and consider the underlying
closed symmetric operator 
\begin{equation*}
 Sf:=Af=Bf,\quad\dom(S):=\bigl\{f\in\dom(A)\cap\dom(B) \, \big| \, Af=Bf\bigr\},
\end{equation*}
which for convenience we assume is densely defined (our extension theory approach would easily generalize 
to the non-densely defined setting and
even to the case
that the domain of $S$ is trivial; however, in order to avoid adjoints of non-densely defined operators we restrict ourselves to the densely defined case here). 
We emphasize that neither $A$ nor $B$ needs to be semibounded in our approach. However, we first impose an implicit sign condition on the perturbation by assuming
\begin{equation}\label{signi}
 (A-\mu_0 I_{\sH})^{-1}\geq (B-\mu_0 I_{\sH})^{-1}
\end{equation}
for some $\mu_0 \in\rho(A)\cap\rho(B)\cap\dR$; in the semibounded case the condition \eqref{signi} is equivalent to $A\leq B$ 
interpreted in the sense of the corresponding quadratic forms. 
We then make use of the
concept of quasi boundary triples in extension theory of symmetric operators from \cite{BL07,BL12} and construct an operator $T$ such that $\overline T=S^*$ and
two boundary mappings $\Gamma_0,\Gamma_1:\dom(T)\rightarrow\cG$, where $\cG$ is an auxiliary Hilbert space, such that 
\begin{equation}
 A=T\upharpoonright\ker(\Gamma_0)\, \text{ and } \, B=T\upharpoonright\ker(\Gamma_1);
\end{equation}
see Proposition~\ref{haha} and Section~\ref{section2} for more details.
To such a quasi boundary triple $\{\cG,\Gamma_0,\Gamma_1\}$ one associates the $\gamma$-field and Weyl function (or abstract Titchmarsh--Weyl $m$-function) $M$ which are defined by 
\begin{equation*}
\gamma(z)\Gamma_0 f_z = f_z \, \text{ and }\, M(z)\Gamma_0 f_z=\Gamma_1 f_z, 
\quad f_z \in \ker(T- z I_{\sH}), \; z\in\rho(A),
\end{equation*}
respectively.
Very roughly speaking the values $M(z)$, $z \in \rho(A)$, of the function $M$ map abstract Dirichlet boundary values 
to abstract Neumann boundary values, or vice versa, 
and hence the Weyl function $M$ associated to a quasi boundary triple can be viewed as an abstract analog of the (energy parameter dependent) Dirichlet-to-Neumann map. 
The resolvents of $A$ and $B$ are related with the $\gamma$-field and Weyl function
via the useful Krein-type formula
\begin{equation*}
 (B- z I_{\sH})^{-1}-(A- z I_{\sH})^{-1} = - \gamma(z) M(z)^{-1} \gamma({\ol z})^*, 
 \quad z\in\rho(A)\cap\rho(B).
\end{equation*}
In our main result, Theorem~\ref{mainssf2},  
in the abstract part of this paper we provide sufficient $\sS_p$-type conditions on the $\gamma$-field and Weyl function of the quasi boundary triple $\{\cG,\Gamma_0,\Gamma_1\}$ such that 
\eqref{resabmmm} is satisfied with $m=2k+1$ and conclude that for any orthonormal basis $(\varphi_j)_{j \in J}$ in $\cG$ (with $J \subseteq \bbN$ an appropriate index set), the function
\begin{equation}\label{ssfi}
 \xi(\lambda)=\sum_{j \in J} \lim_{\varepsilon\downarrow 0}\pi^{-1} 
 \bigl(\Im \big(\log \big(\overline{M(\lambda +i\varepsilon)}\big)\big)\varphi_j,\varphi_j\bigr)_\cG 
 \, \text{ for a.e.~$\lambda\in\dR$,} 
  \end{equation}
is a spectral shift function for the pair $\{A,B\}$ such that $\xi(\lambda)=0$ in an open neighborhood 
of $\mu_0$. In particular, the trace formula
\begin{equation*}
 \tr_{\sH}\bigl( (B- z I_{\sH})^{-(2k+1)} - (A- z I_{\sH})^{-(2k+1)}\bigr) 
 = - (2k+1) \int_\dR \frac{\xi(\lambda)\,d\lambda}{(\lambda - z)^{2k+2}}
\end{equation*}
is valid for all $z \in\rho(A)\cap\rho(B)$. Furthermore, if \eqref{resabmmm} is satisfied with $m=1$ then according to Corollary~\ref{mainthmcorchen} the 
imaginary part of the logarithm of $ z \mapsto\overline{M(z)}$ is a trace class valued Nevanlinna (or Riesz--Herglotz) function 
on the open upper half-plane $\dC_+$ (and hence admits nontangential limits for a.e.~$\lambda \in\dR$ 
from $\dC_+$ in the trace class topology), and the spectral shift function in \eqref{ssfi} has the form
\begin{equation}\label{ssfi2}
 \xi(\lambda)=\lim_{\varepsilon\downarrow 0}\pi^{-1} \tr_{\cG}
 \bigl(\Im \big(\log \big(\overline{M(\lambda +i\varepsilon)}\big)\big)\bigr) 
 \, \text{ for a.e.~$\lambda\in\dR$.} 
  \end{equation} 
 Since $z \mapsto\log \big(\overline{M(z)}\big)$ is a Nevanlinna function it follows that the values of the spectral shift function $\xi$ in \eqref{ssfi} and \eqref{ssfi2}
are nonnegative for a.e.~$\lambda\in\dR$; this is rooted in the sign condition \eqref{signi}. In a second step we weaken the sign condition \eqref{signi} and extend
our representation of the spectral shift function to more general perturbations in the end of Section~\ref{ssfsec} (cf.\ \eqref{ssfab}). 
We point out that the key difficulty in the proof of \eqref{ssfi} and \eqref{ssfi2} is to ensure the existence of the limits on the right hand side of \eqref{ssfi}
and the 
trace class property of the function $\Im\big(\log\big(\overline M\big)\big)$ in the case $k=0$, respectively, which
are indispensable for \eqref{ssfi} and \eqref{ssfi2}. These problems are investigated
separately in Section \ref{logsec} on the logarithm of operator-valued Nevanlinna functions, where special attention is paid to the analytic continuation by reflection 
with respect to open subsets of the real line (cf.\ Theorem~\ref{nthm} and Proposition~\ref{nprop}),
which can be viewed as extensions of some results in \cite[Section 2]{GMN99}. 
We also mention that for the special case where \eqref{resabmmm} is a rank one or 
finite-rank operator and $m=1$, our representation for 
the spectral shift function coincides with one in \cite{BMN08,LSY01}. Furthermore, for $m=1$ in \eqref{resabmmm} a formula for the spectral shift 
function via a perturbation determinant involving boundary parameters and the Weyl function in the context of ordinary boundary triples was shown recently in 
\cite{MN15} (see also \cite{MN14}). We remark that our abstract result can also be formulated and remains valid in the special situation that the 
quasi boundary triple $\{\cG,\Gamma_0,\Gamma_1\}$ is a generalized or ordinary boundary triple in the sense of \cite{BGP08,DHMS06,DM91,DM95,GG91}.

Our main reason to provide the general result in Section~\ref{ssfsec} for the spectral shift function in terms of the abstract notion of quasi boundary triples and their 
Weyl functions is its convenient applicability to various PDE situations, see also \cite{BL07,BL12,BLLLP10,BLL13-AHP,BLL13-IEOT,BLL13-3,BLL-Exner,BMN15}
for other related applications of quasi boundary triples in PDE problems. In Section~\ref{ap1sec} we consider 
a formally
symmetric uniformly elliptic second-order partial differential expression $\cL$ with smooth coefficients on a bounded or unbounded domain in $\dR^n$, $n\geq 2$, with compact boundary,
and 
two self-adjoint realizations 
$A_{\beta_0}$ and $A_{\beta_1}$ of $\cL$ subject to
Robin boundary conditions $\beta_p\gamma_D f=\gamma_N f$, where $\gamma_D$ and $\gamma_N$ denote the Dirichlet and Neumann trace operators, and 
$\beta_p\in C^1(\partial\Omega)$, $p=0,1$, are real-valued functions. 
It then turns out that the Robin realizations $A_{\beta_0}$ and $A_{\beta_1}$ satisfy
\begin{equation}\label{dfggfd}
\big[(A_{\beta_1}- z I_{L^2(\partial\Omega)})^{-(2k+1)} 
 - (A_{\beta_0}- z I_{L^2(\partial\Omega)})^{-(2k+1)}\big] \in\sS_1\bigl(L^2(\Omega)\bigr)
\end{equation}
for all $k \in \bbN_0$, $k\geq (n-3)/4$, and $z \in\rho(A_{\beta_0})\cap\rho(A_{\beta_1})$, and for any orthonormal basis $(\varphi_j)_{j \in J}$ in 
$L^2(\partial\Omega)$, the function 
\begin{align}\label{xihi}
   \xi(\lambda)=\sum_{j \in J} \lim_{\varepsilon\downarrow 0}\pi^{-1} \Bigl(\bigl(\Im\bigl(\log (\cM_1(\lambda+i\varepsilon))-\log (\cM_0(t+i\varepsilon))\bigr)\bigr)\varphi_j,
   \varphi_j\Bigr)_{L^2(\partial\Omega)}& \\
   \, \text{for a.e.~$\lambda \in\dR$,}& \nonumber
\end{align}
 is a spectral shift function for the pair $\{A_{\beta_0},A_{\beta_1}\}$, where 
\begin{equation*}
 \cM_p(z)= (\beta-\beta_p)^{-1}\bigl(\beta_p\overline{\cN(z)}-I_{L^2(\partial\Omega)}\bigr)\bigl(\beta\overline{\cN(z)}-I_{L^2(\partial\Omega)}\bigr)^{-1},
 \quad z \in\dC\backslash\dR,
\end{equation*}
$\beta\in\dR$ is such that $\beta_p(x)<\beta$ for all $x\in\partial\Omega$, and
$\cN(z)$ denotes the ($z$-dependent) Neumann-to-Dirichlet map that assigns Neumann boundary values of solutions $f_z \in H^2(\Omega)$ of $\cL f_z = z f_z$, $z \in \dC\backslash\dR$, onto 
their Dirichlet boundary values. 
We note that the trace class property \eqref{dfggfd} was shown in \cite{BLLLP10,G11} for the case $k=0$ and in \cite{BLL13-3}
for $k\geq 1$. Moreover, in the case $k=0$, that is, $n=2$ or $n=3$, it follows from \eqref{ssfi2} that the spectral shift function in \eqref{xihi} has the form
\begin{equation*}
   \xi(\lambda)=\lim_{\varepsilon\downarrow 0}\pi^{-1} \tr_{L^2(\partial\Omega)}\left(\Im\bigl(\log (\cM_1(\lambda+i\varepsilon))-\log (\cM_0(t+i\varepsilon))\bigr)\right)
   \, \text{   for a.e.~$\lambda \in\dR$.}
  \end{equation*}

In our second example, presented in Section \ref{ap11sec}, we consider a Schr\"{o}dinger operator $B = - \Delta+V$  with a compactly supported potential $V\in L^\infty(\dR^n)$. Here we split the Euclidean space
$\dR^n$ and the Schr\"{o}dinger operator
via a multi-dimensional Glazman decomposition and consider the orthogonal sum $B_D=B_+\oplus C$ of the Dirichlet realizations of $-\Delta+V$ in $L^2(\cB_+)$
and $L^2(\cB_-)$, where $\cB_+$ is a sufficiently large ball which contains $\text{\rm supp}\,(V)$ and $\cB_-:=\dR^n\backslash\overline\cB_+$.
Similarly, the unperturbed operator $A=-\Delta$ is decoupled and compared with the orthogonal sum $A_D=A_+\oplus C$ of the Dirichlet realizations of $-\Delta$
in $L^2(\cB_+)$ and $L^2(\cB_-)$. Our abstract result applies to the pairs $\{B,B_D\}$ and $\{A,A_D\}$, whenever $k> (n-2)/4$, $n \in \bbN$, $n \geq 2$, and yields an 
explicit formula for their spectral shift functions $\xi_B$ and $\xi_A$ 
in terms of the ($z$-dependent) Dirichlet-to-Neumann maps associated to $-\Delta$ and $-\Delta+V$ on 
$\cB_+$ and $\cB_-$.
Since the spectra of the Dirichlet realizations $A_+=-\Delta$ and $B_+=-\Delta+V$ on the bounded domain $\cB_+$ are both discrete and bounded from below, the difference of their eigenvalue
counting functions is a spectral shift function $\xi_+$ for the pair $\{A_+,B_+\}$, and hence also for the pair $\{A_D,B_D\}$. Then it follows that
the function
\begin{equation*}
 \xi(\lambda)=\xi_A(\lambda)-\xi_B(\lambda)+\xi_+(\lambda)\, \text{ for a.e.~$\lambda\in\dR$,}
\end{equation*}
is a spectral shift function for the original pair $\{A,B\}$ (cf.\ Theorem~\ref{thmvv}). Some considerations in the first part of the proof of Theorem~\ref{thmvv} are related to \cite[Section 5.2]{BMN15},
where the scattering matrix of the pair $\{B,B_D\}$ in $\dR^2$ was expressed in terms of 
($z$-dependent) Dirichlet-to-Neumann maps. We also mention that the trace class property of the resolvent differences of $A$ and $A_D$,
and $B$ and $B_D$ goes back to M.\,Sh.~Birman \cite{B62} and G.~Grubb \cite{G84}, and that similar decoupling methods are often used in scattering theory,
see, for instance, \cite{DS76} or \cite{SZ91} for a slighty more abstract and general framework. 

As a third application, presented in Section \ref{ap2sec}, we consider the pair $\{H,H_{\delta,\alpha}\}$, where $H=-\Delta$ is the usual self-adjoint realization of the Laplacian
in $L^2(\dR^n)$, and $H_{\delta,\alpha}=-\Delta-\alpha\delta_\cC$ is a singular perturbation of $H$ by a $\delta$-potential of variable real-valued 
strength $\alpha\in C^1(\cC)$ supported on some
compact hypersurface $\cC$ that splits $\dR^n$, $n\geq 2$, into a bounded interior and an unbounded exterior domain. Schr\"{o}dinger operators with $\delta$-interactions are often used as idealized models 
of physical systems with short-range potentials; in the simplest case point interactions are considered, but in the last decades also 
interactions supported on curves and hypersurfaces have attracted a lot of attention, see the monographs \cite{AGHH05,AK99,EK15}, the review \cite{E08}, and, for instance, 
\cite{AKMN13,AGS87,BLL13-AHP,BLL-Exner,BEKS94,EI01,EK03,EK05,EY02} for a small 
selection of papers in this area. It will be shown in Theorem~\ref{dddthm2} that the trace class condition 
\begin{equation}\label{dfggfd2}
 \big[(H_{\delta,\alpha}- z I_{L^2(\bbR^n)}))^{-(2k+1)}-(H- z I_{L^2(\bbR^n)})^{-(2k+1)}\big] 
 \in\sS_1\bigl(L^2(\dR^n)\bigr)
\end{equation}
is satisfied for for all  $k \in \bbN_0$, $k\geq (n-3)/4$, and $z \in \dC\backslash [0,\infty)$, and in the special case $\alpha<0$ the function 
\begin{align*}
\xi(\lambda)=\sum_{j \in J}\lim_{\varepsilon\downarrow 0}\pi^{-1} \Bigl(\big(\Im \big(\log \big(\bigl(\overline{\cD_{\rm i}(\lambda+i\varepsilon)+\cD_{\rm e}(\lambda+i\varepsilon)}\bigr)^{-1} 
- \alpha^{-1}\big)\big)\big)\varphi_j,\varphi_j\Bigr)_{L^2(\cC)}& \\   
\text{ for a.e.~$\lambda \in\dR$,}&  
\end{align*}
is a spectral shift function for the pair $\{H,H_{\delta,\alpha}\}$ such that $\xi(\lambda)=0$ for $\lambda < 0$; here $\cD_{\rm i}(z)$ and $\cD_{\rm e}(z)$
denote the ($z$-dependent) Dirichlet-to-Neumann maps associated to $-\Delta$ on the interior and exterior domain, respectively, and $(\varphi_j)_{j \in J}$ 
is an orthonormal basis in $L^2(\cC)$.
For the case that no sign condition on $\alpha$ is assumed, a slightly more involved formula in the spirit of \eqref{xihi} for the spectral shift function is provided 
in Theorem~\ref{dddthm} and in Corollary~\ref{mainthmcorchen3} for the cases $n=2$ and $n=3$.
We mention that the trace class property \eqref{dfggfd2}
and the existence of the wave operators was already established in \cite{BLL13-AHP}, see also \cite{BLL-Exner}. 

The applications in Sections~\ref{ap1sec}, \ref{ap11sec}, and \ref{ap2sec} serve as typical examples for the abstract formalism and results in Section~\ref{ssfsec}. In this context
we mention that one may compare in a similar form as in Section~\ref{ap1sec} the Dirichlet realization with the Neumann, or other self-adjoint Robin realizations of
an elliptic partial differential expression, and that in principle also higher-order differential expressions with smooth coefficients could be considered. Similarly,
instead of the singularly perturbed Schr\"{o}dinger operator $H_{\delta,\alpha}$ in Section~\ref{ap2sec} one may compare the free Laplacian $H$ with orthogonal
couplings of Dirichlet, Neumann or Robin realizations on the interior and exterior domain. 
We refer the reader to \cite{GM08,GM11,G11a,G11b,G12,M10,MPS15,MPS16,P15} for some recent related contributions in this area. 

Finally, we briefly summarize the basic notation used in this paper: Let $\cG$, 
$\cH$, $\sH$, etc., be separable complex Hilbert spaces, $(\cdot,\cdot)_{\cH}$ the
scalar product in $\cH$ (linear in the first factor), and $I_{\cH}$ the identity operator
in $\cH$. If $T$ is a linear operator mapping (a subspace of\,) a
Hilbert space into another, $\dom(T)$ denotes the domain and $\ran(T)$ is the range 
of $T$. The closure
of a closable operator $S$ is denoted by $\ol S$. The spectrum and
resolvent set of a closed linear operator in $\cH$ will be denoted by
$\sigma(\cdot)$  and $\rho(\cdot)$, respectively.
The Banach spaces of bounded linear operators in $\cH$ are
denoted by $\cL(\cH)$; in the context of two
Hilbert spaces, $\cH_j$, $j=1,2$, we use the analogous abbreviation
$\cL(\cH_1, \cH_2)$. The $p$-th Schatten-von Neumann ideal consists of compact operators with singular values in $l^p$, $p>0$, and is denoted by $\sS_p(\cH)$ and $\sS_p(\cH_1,\cH_2)$.

For $\Omega \subseteq \bbR^n$ nonempty, $n \in \bbN$,  we suppress the 
$n$-dimensional Lebesgue measure $d^n x$ and use the shorthand notation 
$L^2(\Omega) := L^2(\Omega; d^n x)$; similarly, if $\partial \Omega$ is sufficiently regular we write $L^2(\partial \Omega) := L^2(\partial \Omega; d^{n-1} \sigma)$, with $d^{n-1} \sigma$ 
the surface measure on $\partial \Omega$. 
We also abbreviate $\bbC_{\pm} := \{z \in \bbC \, | \, \Im(z) \gtrless 0\}$ and 
$\bbN_0 = \bbN \cup \{0\}$.

\section{Quasi boundary triples and their Weyl functions}\label{section2}

In this preliminary section we recall the concept of quasi boundary triples and their Weyl functions from
extension theory of symmetric operators. We shall make use of these notions in Section~\ref{ssfsec} and formulate our main
abstract result Theorem~\ref{mainssf2} in terms of the Weyl function of a quasi boundary triple. 
In Sections~\ref{ap1sec}, \ref{ap11sec}, and \ref{ap2sec} quasi boundary triples and their Weyl functions are used to parametrize 
self-adjoint Schr\"{o}dinger operators and self-adjoint elliptic differential operators with suitable boundary conditions. 
We refer to \cite{BL07,BL12} for more details on quasi boundary triples and to \cite{BLLLP10,BLL13-AHP,BLL13-IEOT,BLL13-3,BLL-Exner,BMN15} for some
applications. 

Throughout this section let $\sH$ be a separable Hilbert space and let $S$ be a densely defined closed symmetric operator in $\sH$. 

\begin{definition}\label{qbtdefinition}
Let $T\subset S^*$ be a linear operator in $\sH$ such that $\overline T=S^*$.
A triple $\{\cG,\Gamma_0,\Gamma_1\}$ is said to be a {\em quasi boundary triple}
for $T\subset S^*$ if $\cG$ is a Hilbert space and $\Gamma_0,\Gamma_1:\dom (T)\rightarrow\cG$
are linear mappings such that the following conditions $(i)$--$(iii)$ are satisfied:
\begin{itemize}
  \item [$(i)$] The abstract Green's identity
    \begin{equation*}
      (Tf,g)_\sH-(f,Tg)_\sH=(\Gamma_1 f,\Gamma_0 g)_\cG-(\Gamma_0 f,\Gamma_1 g)_\cG
    \end{equation*}
    holds for all $f,g\in\dom (T)$. 
  \item [$(ii)$] The range of the map $(\Gamma_0,\Gamma_1)^\top:\dom(T)\rightarrow\cG\times\cG$ is dense. 
  \item [$(iii)$] The operator $A_0:=T\upharpoonright\ker(\Gamma_0)$ is self-adjoint in $\sH$.
\end{itemize}
\end{definition}

The notion of quasi boundary triples is a slight extension of the concepts of generalized and ordinary 
boundary triples (see \cite{BGP08,DHMS06,DM91,DM95,GG91,S12} for more details). 
We recall from \cite{BL07,BL12} that
a quasi boundary triple with the additional property that the map $(\Gamma_0,\Gamma_1)^\top:\dom(T)\rightarrow\cG\times\cG$ in Definition~\ref{qbtdefinition}\,$(ii)$ is onto, 
or equivalently, $T=S^*$, is automatically an {\em ordinary boundary triple}, that is,
$\Gamma_0,\Gamma_1:\dom (S^*)\rightarrow\cG$ are linear mappings such that the abstract Green's identity in Definition~\ref{qbtdefinition}\,$(i)$ holds 
for all $f,g \in \dom(S^*)$ and the map $(\Gamma_0,\Gamma_1)^\top$ in 
Definition~\ref{qbtdefinition}\,$(ii)$
is surjective. For an ordinary boundary triple item $(iii)$ in Definition~\ref{qbtdefinition} is automatically satisfied. 
In the proof of Proposition~\ref{haha} we shall also make use of ordinary boundary triples and some of their properties. 

The next theorem from \cite{BL07,BL12} is very useful in the applications in Sections~\ref{ap1sec}, \ref{ap11sec}, and \ref{ap2sec}; it contains a sufficient condition for a
triple $\{\cG,\Gamma_0,\Gamma_1\}$ to be a quasi boundary triple.

\begin{theorem}\label{ratemal}
Let $\sH$ and $\cG$ be separable Hilbert spaces and let $T$ be a linear operator in $\sH$.
Assume that $\Gamma_0,\Gamma_1: \dom (T)\rightarrow\cG$ are linear mappings such
that the following conditions $(i)$--$(iii)$ hold:
\begin{itemize}
\item[$(i)$] The abstract Green's identity
\begin{equation*}
 (Tf,g)_\sH-(f,Tg)_\sH=(\Gamma_1 f,\Gamma_0 g)_\cG-(\Gamma_0 f,\Gamma_1 g)_\cG
\end{equation*}
holds for all $f,g\in\dom(T)$. 
\item [$(ii)$]
The range of 
$(\Gamma_0,\Gamma_1)^\top: \dom (T)\rightarrow\cG\times\cG$
is dense and $\ker(\Gamma_0)\cap\ker(\Gamma_1)$ is dense in $\sH$. 
\item[$(iii)$] $T\upharpoonright \ker(\Gamma_0)$ is an extension of a self-adjoint
operator $A_0$.
\end{itemize}
Then $$S:= T\upharpoonright\bigl(\ker(\Gamma_0)\cap\ker(\Gamma_1)\bigr)$$ 
is a densely defined closed
symmetric operator in $\sH$ such that $\overline T= S^*$ holds and the triple $\{\cG,\Gamma_0,\Gamma_1\}$ is a
quasi boundary triple for $S^*$ with $A_0=T\upharpoonright \ker(\Gamma_0)$.
\end{theorem}

Next, we recall the definition of the $\gamma$-field $\gamma$ and Weyl function $M$ associated to a quasi boundary triple, which
is formally the same as in \cite{DM91,DM95} for the case of ordinary or generalized boundary triples. 
For this let $\{\cG,\Gamma_0,\Gamma_1\}$ be a quasi boundary triple for $T\subset S^*$ with $A_0=T\upharpoonright\ker(\Gamma_0)$.
We note that the direct sum decomposition
\begin{equation}\label{jkjk}
  \dom (T) = \dom (A_0)\,\dot +\,\ker(T - z I_{\sH})
  = \ker(\Gamma_0)\,\dot +\,\ker(T- z I_{\sH})
\end{equation}
of $\dom(T)$ holds for all $z \in \rho(A_0)$, and hence the mapping 
$\Gamma_0\upharpoonright \ker(T - z I_{\sH})$
is injective for all $z \in \rho(A_0)$ and its range coincides with $\ran(\Gamma_0)$.

\begin{definition}\label{gwdeffi}
Let $T\subset S^*$ be a linear operator in $\sH$ such that $\overline T=S^*$
and let $\{\cG,\Gamma_0,\Gamma_1\}$ be a quasi boundary triple for $T\subset S^*$ with $A_0=T\upharpoonright\ker(\Gamma_0)$.
The {\em $\gamma$-field} $\gamma$ and the {\em Weyl function} $M$ corresponding to $\{\cG,\Gamma_0,\Gamma_1\}$ are operator-valued functions on $\rho(A_0)$ which are defined by
\begin{equation*}
   z \mapsto\gamma(z):=\bigl(\Gamma_0\upharpoonright\ker(T - z I_{\sH})\bigr)^{-1} \, \text{ and } \,  z \mapsto  
  M(z) := \Gamma_1\bigl(\Gamma_0\upharpoonright\ker(T - z I_{\sH})\bigr)^{-1}.
\end{equation*}
\end{definition}

Various useful properties of the $\gamma$-field and Weyl function associated to a quasi boundary triple were provided in \cite{BL07,BL12}, see also \cite{BGP08,DHMS06,DM91,DM95,S12} for the special cases of ordinary and 
generalized boundary triples. In the following we briefly review some items which are  important for our purposes. We first note that by Definition~\ref{gwdeffi} 
the values $\gamma(z)$, $z \in \rho(A_0)$, 
of the $\gamma$-field are operators defined on 
the dense subspace $\ran(\Gamma_0)\subset\cG$ which map onto
$\ker(T - z I_{\sH})\subset\sH$. The operators $\gamma(z)$, $z \in \rho(A_0)$, are bounded and admit continuous extensions $\overline{\gamma(z)}\in\cL(\cG,\sH)$. 
For the adjoint operators $\gamma(z)^*\in\cL(\sH,\cG)$, $z \in \rho(A_0)$, it follows from the abstract Green's identity in Definition~\ref{qbtdefinition}\,$(i)$ that  
\begin{equation}\label{gstar}
 \gamma(z)^*=\Gamma_1(A_0-{\ol z} I_{\sH})^{-1},\quad z \in \rho(A_0),
\end{equation}
and, in particular, $\ran(\gamma(z)^*)=\ran(\Gamma_1\upharpoonright\dom (A_0))$ does not depend on $z \in \rho(A_0)$.
It is also important to note that $(\ran(\gamma(z)^*))^\bot=\ker(\gamma(z))=\{0\}$ and hence
\begin{equation}\label{gamden}
 \overline{\ran(\gamma(z)^*)}=\cG,\quad z \in \rho(A_0).
\end{equation}
In the same way as for ordinary boundary triples one verifies 
\begin{equation}\label{xxa}
 \gamma(z)\varphi=\bigl(I_\sH+(z - z_0)(A_0 - z I_{\sH})^{-1}\bigr)\gamma(z_0)\varphi, 
 \quad z, z_0 \in\rho(A_0),\,\,\,\varphi\in\ran(\Gamma_0),
\end{equation}
and therefore $ z \mapsto \gamma(z)\varphi$ is holomorphic on $\rho(A_0)$
for all $\varphi\in\ran(\Gamma_0)$. The relation \eqref{xxa} extends by continuity to
\begin{equation}\label{xxaa}
 \overline{\gamma(z)}=\bigl(I_\sH+(z - z_0)(A_0 - z I_{\sH})^{-1}\bigr)\overline{\gamma(z_0)} 
 \in\cL(\cG,\sH),\quad z, z_0 \in \rho(A_0),
\end{equation}
and it follows that $ z \mapsto\overline{\gamma(z)}$ is a holomorphic $\cL(\cG,\sH)$-valued operator function. According to \cite[Lemma 2.4]{BLL13-3}
the identities
\begin{equation}\label{gammad1}
 \frac{d^k}{dz^k}\overline{\gamma(z)}=k! \, (A_0 - z I_{\sH})^{-k}\overline{\gamma(z)}
\end{equation}
and 
\begin{equation}\label{gammad2}
 \frac{d^k}{dz^k}\gamma({\ol z})^*=k! \, \gamma({\ol z})^*(A_0 - z I_{\sH})^{-k}
\end{equation}
hold for all $k \in \bbN_0$ and $z \in \rho(A_0)$.

The values $M(z)$, $z \in \rho(A_0)$, of the Weyl
function $M$ associated to a quasi boundary triple are operators in $\cG$ and it follows from Definition~\ref{gwdeffi} that
\begin{equation*}
\dom (M(z)) = \ran(\Gamma_0) \quad \mbox{and} \quad \ran (M(z)) \subset \ran (\Gamma_1)
\end{equation*}
hold for all $z \in \rho(A_0)$. In particular, the operators $M(z)$, $z \in \rho(A_0)$, are densely defined in $\cG$. With the help of the abstract Green's identity 
in Definition~\ref{qbtdefinition}\,$(i)$ one concludes that
for $z, z_0 \in\rho(A_0)$ and $\varphi,\psi\in\ran(\Gamma_0)$ the Weyl function and the $\gamma$-field satisfy
\begin{equation}\label{olko}
 (M(z)\varphi,\psi)_\cG-(\varphi,M(z_0)\psi)_\cG=(z - \ol{z_0})\bigl(\gamma(z)\varphi,\gamma(z_0)\psi\bigr)_\cG
\end{equation}
and it follows that $M(z)\subset M({\ol z})^*$ and hence the operators $M(z)$ are closable for all 
$z\in\rho(A_0)$. 
It is important to note that the operators $M(z)$, $z \in \rho(A_0)$, and their closures are unbounded in general; the situation is different when the quasi boundary triple 
$\{\cG,\Gamma_0,\Gamma_1\}$ is a generalized or ordinary boundary triple.
From \eqref{olko} it also follows that
the Weyl function and the $\gamma$-field are connected via
\begin{equation}\label{gutgut}
 M(z)\varphi- M(z_0)^*\varphi =(z - \ol{z_0})\gamma(z_0)^*\gamma(z)\varphi, 
 \quad z, z_0 \in \rho(A_0), \; \varphi\in\ran(\Gamma_0).
\end{equation}
From \eqref{gutgut} and \eqref{xxa} one also obtains  
\begin{equation}\label{imm}
\Im(M(z))\varphi = \Im(z)\,\gamma(z)^*\gamma(z)\varphi,\quad z \in\rho(A_0), \; 
\varphi\in\ran(\Gamma_0),
\end{equation}
and
\begin{equation}\label{mform}
\begin{split}
 M(z)\varphi &=\Re(M(z_0))\varphi\\
 &\quad +\gamma(z_0)^*\bigl((z - \Re(z_0))+(z - z_0)(z - \ol{z_0})(A_0 - z I_{\sH})^{-1}\bigr)\gamma(z_0)\varphi
\end{split}
 \end{equation}
for all $z,z_0\in\rho(A_0)$ and $\varphi\in\ran(\Gamma_0)$. 
One observes that $ z \mapsto M(z)\varphi$ is holomorphic on $\rho(A_0)$ for all $\varphi\in\ran(\Gamma_0)$ and by \eqref{imm} the imaginary part of $M(z)$ is a bounded
operator in $\cG$ which admits a bounded continuation to
\begin{equation}\label{trf}
\overline{\Im (M(z))} = \Im(z)\,\gamma(z)^*\overline{\gamma(z)}\in\cL(\cG).
\end{equation}
Furthermore, the derivatives $\frac{d^k}{dz^k} M(z)$, $k \in \bbN$, of the Weyl function are densely defined bounded operators in $\cG$ and according to 
\cite[Lemma 2.4]{BLL13-3} their closures are given by
\begin{equation*}
 \overline{\frac{d^k}{dz^k} M(z)}=k! \, \gamma({\ol z})^*(A_0 - z I_{\sH})^{-(k-1)}\overline{\gamma(z)}, 
 \quad k \in \bbN, \; z \in \rho(A_0). 
\end{equation*}
In the situation where the values $M(z)$ are densely defined bounded operators for some,
and hence for all $z \in \rho(A_0)$ one has 
\begin{equation}\label{gammad3}
 \frac{d^k}{dz^k}\overline{M(z)}=k! \, \gamma({\ol z})^*(A_0 - z I_{\sH})^{-(k-1)}\overline{\gamma(z)}, 
  \quad k \in \bbN, \; z \in \rho(A_0). 
\end{equation}

The next result will be used in the formulation and proof of our abstract representation formula for the spectral shift function
in Section~\ref{ssfsec}. The statement on the existence of a quasi boundary triple follows also from \cite[Proposition 2.9\,$(i)$]{BMN15} and 
the Krein-type resolvent formula in \eqref{kreinab} is a special case of \cite[Corollary~6.17]{BL12} or \cite[Corollary~3.9]{BLL13-IEOT}. 
For convenience of the reader we provide a simple direct proof.

\begin{proposition}\label{haha}
 Let $A$ and $B$ be self-adjoint operators in $\sH$ and assume that the closed symmetric operator $S=A\cap B$ is densely defined. 
 Then the closure of the operator 
 \begin{equation}\label{trtrz}
 T=S^*\upharpoonright\bigl(\dom (A)+\dom (B)\bigr)
 \end{equation}
 coincides with $S^*$ and there exists a quasi boundary triple 
 $\{\cG,\Gamma_0,\Gamma_1\}$ for $T\subset S^*$ such that 
\begin{equation}\label{hohodfg}
 A=T\upharpoonright\ker(\Gamma_0)\, \text{ and } \, B=T\upharpoonright\ker(\Gamma_1).
\end{equation}
Furthermore, if $\gamma$ and $M$ are the corresponding $\gamma$-field and Weyl function then
\begin{equation}\label{kreinab}
 (B - z I_{\sH})^{-1}-(A - z I_{\sH})^{-1}=-\gamma(z) M(z)^{-1}\gamma({\ol z})^*, \quad 
z \in \rho(A)\cap\rho(B).
\end{equation}
\end{proposition}
\begin{proof}
Since $A$ and $B$ are self-adjoint extensions of the closed symmetric operator $S=A\cap B$, that is,
\begin{equation*}
 Sf=Af=Bf,\quad\dom(S)=\bigl\{f\in\dom(A)\cap\dom(B) \, \big| \, Af=Bf\bigr\},
\end{equation*}
there exists an ordinary boundary triple $\{\cG,\Lambda_0,\Lambda_1\}$ for $S^*$ and a self-adjoint operator $\Theta$
in $\cG$ such that 
\begin{equation}\label{abjaha}
A=S^*\upharpoonright\ker(\Lambda_0)\, \text{ and } \, B=S^*\upharpoonright\ker(\Lambda_1-\Theta\Lambda_0).
\end{equation}
We note that in the present situation the self-adjoint parameter $\Theta$ in $\cG$ is an operator (and not a linear relation) since 
$S=A\cap B$, that is, $A$ and $B$ are disjoint self-adjoint extensions of $S$ (cf.\  \cite{BGP08,DM91,DM95,GG91}).
Now consider the restriction $T$ of $S^*$ onto the subspace $\dom (A)+\dom (B)$ in \eqref{trtrz}. 
Since $A$ and $B$ are disjoint self-adjoint extensions of $S$ it follows that $\overline T=S^*$ 
(see, e.g., \cite[Proposition 2.9]{BMN15}).
We claim that $\{\cG,\Gamma_0,\Gamma_1\}$, where
\begin{equation*}
 \Gamma_0 f:=\Lambda_0 f\, \text{ and } \, \Gamma_1 f:=\Lambda_1 f-\Theta\Lambda_0 f,\quad f\in\dom (T),
\end{equation*}
 is a quasi boundary triple for $T\subset S^*$ such that \eqref{hohodfg} holds. In fact, \eqref{hohodfg} is clear from \eqref{abjaha} and the definition of $\Gamma_0$
 and $\Gamma_1$, and hence it remains to check items $(i)$--$(iii)$ in Definition~\ref{qbtdefinition}.
 For $f,g\in\dom(T)=\dom(A)+\dom(B)$ one computes 
\begin{equation*}
 \begin{split}
 (\Gamma_1 f,\Gamma_0 g)_\cG-(\Gamma_0 f,\Gamma_1 g)_\cG&=\bigl(\Lambda_1f-\Theta\Lambda_0f,\Lambda_0g\bigr)_\cG-\bigl(\Lambda_0f,\Lambda_1g-\Theta\Lambda_0g\bigr)_\cG \\
 &= (\Lambda_1f,\Lambda_0g)_\cG-(\Lambda_0f,\Lambda_1g)_\cG \\
 &= (S^*f,g)_\sH-(f,S^*g)_\sH\\
 &= (Tf,g)_\sH-(f,Tg)_\sH, 
 \end{split}
\end{equation*}
and hence the abstract Green's identity in Definition~\ref{qbtdefinition}\,$(i)$ is valid. Next, assume that
\begin{equation*}
 0=\left(\begin{pmatrix}\varphi \\ \psi\end{pmatrix},\begin{pmatrix} \Gamma_0 f\\ \Gamma_1 f\end{pmatrix}\right)_{\cG\times\cG}
=(\varphi,\Lambda_0 f)_\cG+\bigl(\psi,\Lambda_1 f-\Theta\Lambda_0 f)_\cG
\end{equation*}
holds for some $\varphi,\psi\in\cG$ and all $f\in\dom(T)$. Since $\{\cG,\Lambda_0,\Lambda_1\}$ is an ordinary boundary triple the map 
$(\Lambda_0,\Lambda_1)^\top:\dom (S^*)\rightarrow\cG\times\cG$
is surjective. It follows that $\Lambda_1\upharpoonright\ker(\Lambda_0)$ maps onto $\cG$ and hence for $f\in\dom(A)=\ker(\Lambda_0)$
one has $0=(\psi,\Lambda_1 f)$, and therefore, $\psi=0$. Now $(\varphi,\Lambda_0 f)_\cG=0$ for $f\in\dom(T)$, and the fact that
the range of the restriction of $\Lambda_0$ onto $\dom(T)$ is dense in $\cG$ (this follows since $\Lambda_0:\dom(S^*)\rightarrow\cG$ 
is surjective, continuous with respect to the graph on $\dom(S^*)$ and $\dom(T)$ is dense in $\dom(S^*)$ with respect to the graph norm), 
yield $\varphi=0$. Therefore, the range of the mapping $(\Gamma_0,\Gamma_1)^\top:\dom(T)\rightarrow\cG\times\cG$ is dense and
hence condition $(ii)$ in Definition~\ref{qbtdefinition}  holds. Condition $(iii)$ is clear from \eqref{hohodfg}. Thus, we have shown that
$\{\cG,\Gamma_0,\Gamma_1\}$ is a quasi boundary triple for $\overline T=S^*$.

Next, we verify the Krein-type resolvent formula \eqref{kreinab}. Fix $z \in \rho(A)\cap\rho(B)$ and note  that $\ker (M(z))=\{0\}$. In fact, if $M(z)\varphi=0$ for some $z \in \rho(A)\cap\rho(B)$ and 
$\varphi\in\cG$, there exists $f_z \in\ker(T - z I_{\sH})$
such that $\Gamma_0 f_z =\varphi$. Since $M(z)\Gamma_0 f_z =\Gamma_1 f_z $ by 
Definition~\ref{gwdeffi}, it follows that $\Gamma_1 f_z =0$,
that is, $f_z \in\dom(B)\cap\ker(T - z I_{\sH})$ and hence $f_z \in\ker(B - z I_{\sH})$. From $z \in \rho(B)$ one concludes that $f_z =0$ and hence $\varphi=\Gamma_0 f_z =0$,
that is, $\ker (M(z))=\{0\}$. Similarly, it follows from the decomposition \eqref{jkjk} with $A_0$ and $\Gamma_0$ replaced by $B$ and $\Gamma_1$ that
$\ran (M(z))=\ran(\Gamma_1)\supset\ran(\gamma({\ol z})^*)$ holds for $z \in \rho(A)\cap\rho(B)$; see \eqref{gstar} for the last inclusion.
Let $g\in\sH$ and define
\begin{equation}\label{fg}
 f:=(A- z I_{\sH})^{-1}g-\gamma(z) M(z)^{-1}\gamma({\ol z})^*g.
\end{equation}
Then \eqref{gstar} and Definition~\ref{gwdeffi} yield
\begin{equation*}
\begin{split} 
\Gamma_1 f&=\Gamma_1(A - z I_{\sH})^{-1}g-\Gamma_1\gamma(z)M(z)^{-1}\gamma({\ol z})^*g\\
          &=\gamma({\ol z})^*g - M(z)M(z)^{-1}\gamma({\ol z})^*g\\
          &=0
\end{split}
\end{equation*}
and hence $f\in\ker(\Gamma_1)=\dom(B)$. From 
\begin{equation*}
(B - z I_{\sH})f=(T - z I_{\sH})\bigl((A - z I_{\sH})^{-1}g-\gamma(z)M(z)^{-1}\gamma({\ol z})^*g\bigr) =g
\end{equation*}
and \eqref{fg}
one infers \eqref{kreinab}.
\end{proof}

\section{Logarithms of operator-valued Nevanlinna functions}\label{logsec}

This section is closely connected to and inspired by the considerations in \cite[Section 2]{GMN99} 
on the logarithm of operator-valued Nevanlinna (or Nevanlinna--Herglotz, resp., Riesz--Herglotz) functions. Here we shall recall some of the results formulated 
in \cite{GMN99} which go back to \cite{BE67,N87,N89,N90}, and slightly extend and reformulate these in a form convenient for our subsequent purposes.

We first recall the integral representation of the logarithm that corresponds to the cut along the negative imaginary axis,
\begin{equation}\label{logz}
 \log (z)=-i\int_0^\infty \left(\frac{1}{z+i \lambda}-\frac{1}{1+i \lambda}\right) d\lambda, 
 \quad z\in\dC,\; z\not =-i \lambda,\; \lambda \geq 0.
\end{equation}
Next, let $\cG$ be a separable Hilbert space and let $K\in\cL(\cG)$ be a bounded operator such that 
$\Im(K) \geq 0$ and $0\subset\rho(K)$. Then we use
\begin{equation}\label{logt}
 \log (K):=-i\int_0^\infty \bigl[(K+ i \lambda I_{\cG})^{-1}-(1+ i \lambda)^{-1}I_\cG\bigr] \, d\lambda
\end{equation}
as the definition of the logarithm of the operator $K$.
Then $\log (K)\in\cL(\cG)$ by \cite[Lemma 2.6]{GMN99} and in the special case that $K\in\cL(\cG)$
is self-adjoint and $0\in\rho(K)$, it follows from \cite[Lemma 2.7]{GMN99} that
\begin{equation}\label{t-}
 \Im(\log(K))=\pi E_K((-\infty,0)),
\end{equation}
where $E_K(\cdot)$ is the spectral measure of $K$. In particular, if $K\in\cL(\cG)$ is self-adjoint and $0\in\rho(K)$ then 
$\sigma(K)\subset (0,\infty)$ if and only if $\log(K)$ is a self-adjoint operator.

In the next lemma we show that besides $\log(K)$ also  $\log(K^*)$ is well-defined via \eqref{logt} when $K$ is a dissipative operator with spectrum 
off the imaginary axis (cf.\ \cite[Lemma 2.6 and Lemma 2.7]{GMN99}).
\begin{lemma}\label{loglem}
 Let $K\in\cL(\cG)$ be a dissipative operator such that $i \lambda \in \rho(K)$ for all $\lambda \geq 0$, and define 
 \begin{equation}\label{logns}
  \log(K^*):=-i\int_0^\infty \bigl[(K^*+ i \lambda I_{\cG})^{-1}-(1+ i \lambda)^{-1}I_\cG\bigr] \, d\lambda.
 \end{equation}
 Then $\log(K^*)\in\cL(\cG)$.
\end{lemma}
\begin{proof}
 From $\sigma(K^*)=\{z\in\dC \, | \, {\ol z}\in\sigma(K)\}$ and the assumption $i\lambda \in \rho(K)$ 
 for $\lambda \geq 0$ it is clear that
     $- i \lambda\in\rho(K^*)$ for $\lambda \geq 0$. Since $K$ is dissipative it follows that $K^*$ is accretive, that is, $\Im(K^*) \leq 0$. For $\delta>0$ one estimates  
\begin{align}
\begin{split}
 \Vert \log(K^*)\Vert_{\cL(\cG)} &\leq 
 \int_0^\delta \bigl[\big\Vert (K^*+ i \lambda I_{\cG})^{-1}\big\Vert_{\cL(\cG)} + 1\bigr] \, d\lambda \\
 & \quad + \int_\delta^\infty \big\Vert (K^*+ i \lambda I_{\cG})^{-1}\big\Vert_{\cL(\cG)} 
 \bigl(\Vert K\Vert_{\cL(\cG)} + 1\bigr) 
 \lambda^{-1}\, d\lambda.     \lb{3.4a} 
 \end{split} 
\end{align}
For $0< \lambda < \big\Vert (K^*)^{-1}\big\Vert^{-1}_{\cL(\cG)}$ one has  
\begin{equation*}
\big\Vert (K^*+ i \lambda I_{\cG})^{-1}\big\Vert_{\cL(\cG)} \leq \frac{\big\Vert(K^*)^{-1}\big\Vert_{\cL(\cG)}}{1-\lambda \big\Vert (K^*)^{-1}\big\Vert_{\cL(\cG)}},    
\end{equation*}
and with the choice $\delta = \big(2 \big\Vert (K^*)^{-1}\big\Vert_{\cL(\cG)}\big)^{-1}$ it follows that the first integral in \eqref{3.4a} is bounded.
In order to show that the second integral in \eqref{3.4a} is also bounded it suffices to show that 
\begin{equation}\label{malsehen}
\big\Vert (K^*+ i \lambda I_{\cG})^{-1}\big\Vert_{\cL(\cG)} \leq \frac{1}{\lambda - \Vert K^*\Vert_{\cL(\cG)}}, 
\quad \lambda > \Vert K^*\Vert_{\cL(\cG)}.
\end{equation}
In fact, since $\Im (K^*+i\Vert K^*\Vert_{\cL(\cG)} I_{\cG})\geq 0$ one estimates for 
$\lambda > \Vert K^*\Vert_{\cL(\cG)}$, 
\begin{equation}
 \begin{split}
 0 &\leq (\lambda - \Vert K^*\Vert_{\cL(\cG)})\Vert f\Vert^2_{\cG}   \\ 
&  =\Im \bigl(i(\lambda - \Vert K^*\Vert_{\cL(\cG)}) f,f\bigr)_{\cG}  \\
 & \leq \Im \bigl((i \lambda - i\Vert K^*\Vert_{\cL(\cG)}) f,f\bigr)_{\cG} 
 +\Im \bigl((K^*+i\Vert K^*\Vert_{\cL(\cG)} I_{\cG}) f,f\bigr)_{\cG}  \\
 & =\Im\bigl((K^*+ i \lambda I_{\cG})f,f\bigr)_{\cG}  \\
 & \leq \Vert (K^*+ i \lambda I_{\cG})f\Vert_{\cG} \Vert f\Vert_{\cG}
 \end{split}
\end{equation}
and for $f\not=0$ this yields 
\begin{equation}\label{fastfertig}
 0\leq (\lambda - \Vert K^*\Vert_{\cL(\cG)})\Vert f\Vert_{\cG} 
 \leq \Vert (K^*+ i \lambda I_{\cG})f\Vert_{\cG}.
\end{equation}
Since $- i \lambda\in\rho(K^*)$ there exists $g\in\cG$ such that $f=(K^*+ i \lambda I_{\cG})^{-1}g$ and then \eqref{fastfertig} has the form
\begin{equation*}
 \big\Vert (K^*+ i \lambda I_{\cG})^{-1} g \big\Vert_{\cG} \leq 
 \frac{1}{\lambda - \Vert K^*\Vert_{\cL(\cG)}} \Vert g\Vert_{\cG}, 
 \quad \lambda > \Vert K^*\Vert_{\cL(\cG)}.
\end{equation*}
This implies \eqref{malsehen}, and hence the second integral in the estimate \eqref{3.4a} is 
finite. Thus, $\log(K^*)$ in \eqref{logns} is a bounded operator in $\cG$.
\end{proof}

We recall that a function 
$N:\dC_+\rightarrow\cL(\cG)$ is an operator-valued {\it Nevanlinna} (or {\it Riesz--Herglotz}) {\it function} if $N$ is holomorphic and $\Im(N(z))\geq 0$
holds for all $z\in\dC_+$. An $\cL(\cG)$-valued Nevanlinna function is extended onto $\dC_-$ by setting
\begin{equation}\label{n-}
 N(z):=N({\ol z})^*,\quad z\in\dC_-.
\end{equation}
We shall say that a Nevanlinna function $N$ admits an {\it analytic continuation by reflection} with respect to some open subset $I\subset\dR$ if $N$ can be continued analytically 
from $\dC_+$ onto an open set $\cO\subset\dC$ which contains $I$ such that the values of the continuation in $\cO\cap\dC_-$ coincide with the values of 
$N$ in \eqref{n-} there.

\begin{example}
If $\sqrt{z}$ is fixed for $\dC\backslash[0,\infty)$ by $\Im(\sqrt{z}) > 0$ and by $\sqrt{z}\geq 0$ for $z\in[0,\infty)$ then $\dC_+\ni z\mapsto\sqrt{z}$ is a $($scalar\,$)$ Nevanlinna function
which admits an analytic continuation by reflection with respect to $(-\infty,0)$, but it does not admit an analytic continuation by reflection with respect to any open subinterval of $[0,\infty)$.
\end{example}

An operator-valued Nevanlinna function admits a minimal operator representation via the resolvent of a self-adjoint operator or relation in an auxiliary or larger Hilbert space (see, e.g., 
\cite{BE67,HSW98,KL77,LT77,N87}).
More precisely, if $N:\dC_+\rightarrow \cL(\cG)$ is a Nevanlinna function and $z_0\in\dC_+$ is fixed then there exists a Hilbert space $\cK$, 
a self-adjoint operator or self-adjoint relation $L$ in $\cK$ and an operator $R \in \cL(\cG,\cK)$ (depending on the choice of $z_0$) such that
\begin{equation}\label{logmrep7}
N(z) = \Re(N(z_0)) + (z - \Re(z_0)) R^* R + (z - z_0)(z - {\ol z}_0)R^*(L - z I_{\cK})^{-1}R
\end{equation}
holds for $z \in \dC_+$. If $N$ satisfies the condition
\begin{equation}\label{condin}
 \lim_{y\uparrow +\infty} y^{-1}(N(iy)h,h)_{\cG}=0\, \text{ for all } \, h\in\cG,
\end{equation}
then $L$ in \eqref{logmrep7} is a self-adjoint  operator in $\cK$ (and not a relation, cf.\  \cite[Corollary~2.5]{LT77}).
The representation \eqref{logmrep7} also holds for $z\in\dC_-$ when $N$ is extended onto $\dC_-$ via \eqref{n-}.  For us it is important that the model can be chosen 
{\it minimal},
that is, the minimality condition
\begin{equation*}
 \cK=\text{clsp}\,\bigl\{(I_\cK+(z - z_0)(L - z I_{\cK})^{-1})R h \, \big| \, z\in\dC\backslash\dR, 
 \, h\in\cG\bigr\}
\end{equation*}
is satisfied, in which case the resolvent set $\rho(L)$ of $L$ coincides with the maximal domain of analyticity of the function $N$. In particular, in this case 
$N$ admits an
analytic continuation by reflection with respect to an open subset $I\subset\dR$ if and only if $I\subset\rho(L)$, and the open subset $\rho(L)\cap\dR$ is maximal with this property.

Next, assume that $N$ is an $\cL(\cG)$-valued Nevanlinna function and suppose that $N(z)^{-1}\in\cL(\cG)$ for some, and hence (by \cite[Lemma 2.3]{GMN99}) for all $z\in\dC\backslash\dR$.
Then we define for $z\in\dC_+$ the logarithm $\log(N(z))$ in accordance with \eqref{logt} by
\begin{equation}\label{logn}
 \log (N(z)):=-i\int_0^\infty \bigl[(N(z)+ i \lambda I_{\cG})^{-1}-(1+ i \lambda)^{-1}I_\cG\bigr] \, d\lambda,
\end{equation}
and extend the function $\log (N)$ onto $\dC_-$ by reflection,
\begin{equation}\label{logn2}
 \log(N(z)):=\bigl(\log (N({\ol z}))\bigr)^*,\quad z\in\dC_- 
\end{equation}
(cf.\ \eqref{n-}).
By \cite[Lemma 2.8]{GMN99} the function $ z \mapsto\log(N(z))$ is also an $\cL(\cG)$-valued Nevanlinna function
and satisfies
\begin{equation}\label{estilogi}
 0\leq\Im(\log(N(z)))\leq\pi I_\cG,\quad z\in\dC_+.
\end{equation}

The following theorem is a variant and slight extension of \cite[Theorem 2.10]{GMN99}, 
the new and important feature here is that we provide a sufficient condition in terms of the function $N$ such that $\log (N)$ admits an analytic continuation by reflection with respect 
to some real interval and a corresponding integral representation there.

\begin{theorem}\label{nthm}
Let $N:\dC\backslash\dR\rightarrow\cL(\cG)$ be a Nevanlinna function and assume that $N(z)^{-1}\in\cL(\cG)$ for some, and hence for all $z\in\dC\backslash\dR$. Then there exists a weakly Lebesgue measurable operator-valued function 
$\lambda \mapsto \Xi(\lambda)\in\cL(\cG)$ on $\dR$ such that 
\begin{equation}\label{bitte}
\Xi(\lambda)=\Xi(\lambda)^*\, \text{ and } \, 
 0\leq\Xi(\lambda)\leq I_\cG\, \text{ for a.e.~$\lambda \in \dR$,}
\end{equation}
and the Nevanlinna function $\log (N):\dC\backslash\dR\rightarrow\cL(\cG)$ in \eqref{logn}--\eqref{logn2} admits an integral representation of the form
\begin{equation}\label{bittesehr}
 \log(N(z))= C +\int_\dR\left(\frac{1}{\lambda - z}-\frac{\lambda}{1+\lambda^2}\right) 
 \Xi(\lambda)\, d\lambda,
\end{equation}
where $C=\Re(\log(N(i)))\in\cL(\cG)$ is a self-adjoint operator and the integral is understood in the weak sense.

If, in addition, $N$ admits an analytic continuation by reflection with respect to an open interval $I\subset\dR$ such that $\sigma(N(z))\subset (\varepsilon,\infty)$ for some 
$\varepsilon>0$ and  all $z \in I$, 
then also $\log (N)$ admits an analytic continuation by reflection with respect to $I$, $\Xi(\lambda)=0$ for a.e.~$\lambda \in  I$, and \eqref{bittesehr} remains valid for $z \in I$.
\end{theorem}
\begin{proof}
We make use of the representation \eqref{logmrep7} applied to the Nevanlinna function $\log (N)$ 
with $z_0=i$. Then there exists a Hilbert space $\cK$ and $R\in\cL(\cG,\cK)$
such that
\begin{equation}\label{logmrep8}
\log (N(z)) = C+ z R^* R + (1 + z^2)R^*(L - z I_{\cK})^{-1}R, \quad z \in \bbC \backslash \bbR, 
\end{equation}
where $C=\Re(\log (N(i))) \in \cL(\cG)$ is a self-adjoint operator. For $h\in\cG$ it follows from \eqref{logmrep8} that
\begin{equation*}
 \lim_{y\rightarrow +\infty}\frac{1}{y}\Re \bigl(\log((N(iy))h,h)_{\cG})\bigr) =0
\end{equation*}
and \eqref{estilogi} implies
\begin{equation*}
 \lim_{y\rightarrow +\infty} \frac{1}{y}\Im \bigl(\log((N(iy))h,h)_{\cG})\bigr) =0,
\end{equation*}
so that \eqref{condin} holds for the function $\log (N)$. Hence $L$ in \eqref{logmrep8} is a self-adjoint operator in $\cK$ (cf.\ \cite[Lemma 2.9]{GMN99}). 
We can assume that the model
is chosen minimal and hence $\rho(L)$ coincides with the maximal domain of analyticity of the Nevanlinna function $\log (N)$.

In order to prove \eqref{bitte} and \eqref{bittesehr} one can argue in the same way as in the proof of \cite[Theorem 2.10]{GMN99}. 
Let $\lambda \mapsto E_L(\lambda)$ be the spectral function of $L$ such that 
$\lim_{\lambda \downarrow -\infty} (E_L(\lambda)h,h)_{\cG}=0$. Then \eqref{logmrep8} yields
\begin{equation*}
 \bigl(\log (N(z))h,h\bigr)_{\cG} = (Ch,h)_{\cG}+\int_\dR \left(\frac{1}{\lambda - z} 
 - \frac{\lambda}{1 + \lambda^2}\right)(1 + \lambda^2)\, d \bigl(R^*E_L(\lambda)Rh,h\bigr)_{\cG}
\end{equation*}
for $h\in\cG$, $z \in \bbC \backslash \bbR$, and \eqref{estilogi} and the Stieltjes inversion formula implies that the measures
\begin{equation}\label{omegah}
 d \omega_h(\cdot)=(1 + \lambda^2)d \bigl(R^*E_L(\cdot)Rh,h\bigr)_{\cG}
\end{equation}
are absolutely continuous with respect to the Lebesgue measure $d\lambda$ and there exist measurable functions 
$\xi_h$ with $0\leq \xi_h(\lambda)\leq\Vert h\Vert^2_{\cG}$ for a.e.~$\lambda \in \dR$ 
such that $d\omega_h(\lambda)=\xi_h(\lambda)\, d\lambda$. Hence there exists a weakly Lebesgue measurable function $\lambda \mapsto\Xi(\lambda)$ such that
\begin{equation*}
 \xi_h(\lambda)=(\Xi(\lambda)h,h)_{\cG}\, \text{ and } \, 0\leq\Xi(\lambda)\leq I_\cG,
\end{equation*}
proving \eqref{bitte} and \eqref{bittesehr}.

Next, assume that $N$ admits an analytic continuation by reflection with respect to an open interval $I\subset\dR$ such that $\sigma(N(z))\subset (\varepsilon,\infty)$ for some 
$\varepsilon>0$ and  all $z \in I$. Fix some $z_0 \in I$ and an open ball $\cB_{z_0}\subset\dC$ centered at $z_0$ such that $N$ is analytic on $\cB_{z_0}$.
Since $\sigma(N(z_0))\subset(\varepsilon,\infty)$ we can assume that $\cB_{z_0}$ was chosen such that 
\begin{equation*}
 \sigma(N(z))\subset \{z\in\dC \, | \, \varepsilon/2 < \Re(z), \,  
0 \leq \Im (z) < \varepsilon\},   \quad  z \in\cB_{z_0}\cap\dC_+,
\end{equation*}
and hence the operators $N(z)$, $z \in \cB_{z_0}\cap\dC_+$, satisfy the assumptions in Lemma~\ref{loglem}. 
Therefore, the operators
\begin{equation*}
 \log(N(z)^*):=-i\int_0^\infty \bigl[(N(z)^*+ i \lambda I_{\cG})^{-1}-(1+ i \lambda)^{-1}I_\cG\bigr] 
 \, d\lambda, 
 \quad z \in\cB_{z_0}\cap\dC_+,
\end{equation*}
are well-defined, and since $N({\ol z})=N(z)^*$, it follows that 
\begin{equation*}
 \log(N(z))=-i\int_0^\infty \bigl[(N(z)+ i \lambda I_{\cG})^{-1}-(1+ i \lambda)^{-1}I_\cG\bigr] \, d\lambda, 
 \quad z \in \cB_{z_0}\cap\dC_-,
\end{equation*}
are well-defined, bounded operators in $\cG$.
Furthermore, Lemma~\ref{loglem} also ensures that for $z \in \cB_{z_0}\cap\dR$ the operators 
\begin{equation*}
 \log(N(z)):=-i\int_0^\infty \bigl[(N(z)+ i \lambda I_{\cG})^{-1}-(1+ i \lambda)^{-1}I_\cG\bigr] \, d\lambda, 
 \quad z \in \cB_{z_0}\cap\dR,
\end{equation*}
are well-defined, bounded operators in $\cG$. Thus for all $z \in \cB_{z_0}$, the operators 
$\log(N(z))$ are well-defined
via \eqref{logn}. It then follows from \eqref{logn} that the function $ z \mapsto\log(N(z))$ is 
analytic on $\cB_{z_0}$ (cf.\ \cite[Proof of Lemma 2.8]{GMN99}).

We shall now also make use of the logarithm 
\begin{equation}
 \ln(z)=\int_{-\infty}^0\left(\frac{1}{\lambda - z}-\frac{\lambda}{1 + \lambda^2}\right) d\lambda, 
 \quad z\in\dC\backslash (-\infty,0],
\end{equation}
which corresponds to the cut along the negative real axis. Since 
\begin{equation*}
 \sigma(N(z))\subset \{ z\in\dC \, | \, \varepsilon/2 < \Re(z), 
 \, -\varepsilon < \Im(z) < \varepsilon\},\quad z \in \cB_{z_0},
\end{equation*}
it follows that  
\begin{equation}\label{lnn}
 \ln(N(z))=\int_{-\infty}^0 \bigl[(\lambda I_{\cG}-N(z))^{-1} - \lambda(1 + \lambda^2)^{-1}I_\cG\bigr] 
 \, d\lambda,  \quad z \in \cB_{z_0},
\end{equation}
are well-defined operators and the function $ z \mapsto\ln(N(z))$ is analytic on $\cB_{z_0}$. 
In addition, \eqref{lnn} yields
\begin{equation}\label{lnns}
 \bigl(\ln(N(z))\bigr)^*=\ln(N(z)^*),\quad z \in\cB_{z_0}.
\end{equation}
As $\log(z)=\ln(z)$ (see \eqref{logz}) for all $z>0$ and $N(z)$ is self-adjoint for $z \in I$
it follows from the spectral theorem that
\begin{equation*}
 \log(N(z))=\ln(N(z)),\quad z \in I,
\end{equation*}
and hence $\log(N(z))=\ln(N(z))$, $z \in\cB_{z_0}$, by analyticity. Therefore, \eqref{lnns} and $N(z)^*=N({\ol z})$ yield
\begin{equation*}
 \bigl(\log(N(z))\bigr)^*=\bigl(\ln(N(z))\bigr)^*=\ln(N(z)^*)=\ln(N({\ol z}))=\log(N({\ol z})), 
 \quad z \in \cB_{z_0}.
\end{equation*}
It follows that $ z \mapsto \log(N(z))$ is analytic on $\cB_{z_0}$ and the continuation of $\log(N)$
onto $\cB_{z_0} \cap \dC_-$ coincides with the extension of $\log (N)$ onto $\dC_-$ defined by
\begin{equation*}
\log (N(z))=\bigl(\log N({\ol z})\bigr)^*,\quad z\in\dC_- 
\end{equation*}
(cf.\ \eqref{n-}). This reasoning applies to all $\nu\in I$ and hence we have shown that $\log (N)$ admits an analytic continuation by reflection with respect to $I$.

Since we have chosen a minimal operator model for $\log (N)$ the interval $I$ belongs to $\rho(L)$ and the representation \eqref{logmrep8} remains valid for $z \in \rho(L)$.
It follows in this situation that the measures $d\omega_h(\cdot)$, $h\in\cG$, in \eqref{omegah} have no support in $I$ and hence their Radon--Nikodym deriatives satisfy $\xi_h(\lambda)=0$ for
a.e.~$\lambda \in  I$. It follows that $(\Xi(\lambda)h,h)_{\cG}=0$ for a.e.~$\lambda \in  I$ and 
all $h\in\cG$. Since $\Xi(\lambda)\geq 0$ we conclude $\Xi(\lambda)=0$ for a.e.~$\lambda \in  I$.
\end{proof}

In the next proposition we provide a sufficient condition such that the values of the function $\Xi$ are trace class operators and we express the traces of $\Xi(\lambda)$
in terms of certain weak limits of the imaginary part of $\log(N)$.

\begin{proposition}\label{nprop}
Let $N:\dC\backslash\dR\rightarrow\cL(\cG)$ be a Nevanlinna function such that $N(z)^{-1}\in\cL(\cG)$ for some, and hence 
for all $z\in\dC\backslash\dR$, and assume that $N$ admits an analytic continuation by reflection with respect to an open interval 
$I\subset\dR$ such that $\sigma(N(\zeta))\subset (\varepsilon,\infty)$ 
for some 
$\varepsilon>0$ and  all $\zeta \in I$. Consider the integral representation  
\begin{equation}\label{bittesehr2}
 \log(N(z))= C +\int_\dR\left(\frac{1}{\lambda - z}-\frac{\lambda}{1 + \lambda^2}\right)\Xi(\lambda)\, d\lambda,
\end{equation}
for $z\in (\dC\backslash\dR)\cup I$ with $\Xi(\lambda)=\Xi(\lambda)^*$ and $0\leq\Xi(\lambda)\leq I_\cG$ for a.e.~$\lambda \in \dR$ as in \eqref{bitte},
and assume, in addition, that for some $k \in \bbN_0$ and some $\zeta \in I$, 
\begin{equation}\label{2k}
\frac{d^{2k+1}}{d\zeta^{2k+1}}\log(N(\zeta)) \in\sS_1(\cG). 
\end{equation}
Then 
\begin{equation*} 
0\leq \Xi(\lambda)\in\sS_1(\cG) \, \text{ for a.e.~$\lambda \in \dR$,} 
\end{equation*} 
and
\begin{equation}\label{trtrxi}
 \tr_{\cG}(\Xi(\lambda))=\sum_{j \in J} \lim_{\varepsilon\downarrow 0}\frac{1}{\pi}\bigl(\Im(\log(N(\lambda+i\varepsilon)))\varphi_j,\varphi_j\bigr)_\cG
\end{equation}
holds for any orthonormal basis $(\varphi_j)_{j \in J}$ in $\cG$ $($$J \subseteq \bbN$ 
an appropriate index set\,$)$ and for a.e. $\lambda\in\dR$. 
Furthermore, if \eqref{2k} holds for some $\zeta\in I$ and $k=0$, that is,
\begin{equation}\label{hahagut}
\frac{d}{d\zeta}\log(N(\zeta)) \in\sS_1(\cG),
\end{equation}
then $\Im(\log(N(z)))\in\sS_1(\cG)$ for all $z\in\dC\backslash\dR$, the limit 
$$\Im\bigl(\log(N(\lambda+i0))\bigr):=\lim_{\varepsilon\downarrow 0}\Im\big(\log(N(\lambda+i\varepsilon))\big)\in\sS_1(\cG)$$ 
exists for a.e. 
$\lambda\in\dR$ in the norm of $\sS_1(\cG)$, and 
\begin{equation}\label{immerhin}
 \tr_{\cG}(\Xi(\lambda))=\frac{1}{\pi}\tr_\cG\bigl(\Im(\log(N(\lambda+i0)))\bigr) \, \text{ for a.e.~$\lambda \in \dR$.}
\end{equation}
\end{proposition}
\begin{proof}
 The assumption \eqref{2k} together with the integral representation \eqref{bittesehr2} yields
 \begin{equation}\label{jajagut}
  \frac{d^{2k+1}}{d\zeta^{2k+1}}\log (N(\zeta)) 
  = (2k+1)!\int_\dR\frac{1}{(\lambda - \zeta)^{2k+2}}\,\Xi(\lambda)\, d\lambda \in\sS_1(\cG), 
  \quad k \in \bbN_0, \; \zeta \in I. 
 \end{equation}
Since $\Xi(\lambda)\geq 0$ by \eqref{bitte} and $(\lambda - \zeta)^{-2k-2}\geq 0$ for all 
$\lambda \in \dR$, $\zeta \in I$, 
it follows together with the assumption \eqref{2k} that the integral in \eqref{jajagut} is a nonnegative trace class operator. Similarly, as in \cite[Proof of Theorem 2.10]{GMN99}, the monotone convergence theorem yields 
$\Xi(\lambda)\in\sS_1(\cG)$ for a.e.~$\lambda \in \dR$. For $\varepsilon>0$ it follows from the integral representation \eqref{bittesehr2} that
\begin{equation}
 \bigl(\Im(\log(N(\lambda+i\varepsilon)))h,h\bigr)_\cG = \int_\dR \frac{\varepsilon}{\vert \lambda' - \lambda\vert^2+\varepsilon^2}\,(\Xi(\lambda')h,h)_\cG\,d\lambda'
\end{equation}
holds for all $h\in\cG$ and all $\lambda\in\dR$, and therefore the Stietljes inversion formula yields
\begin{equation}
 \lim_{\varepsilon\downarrow 0}\frac{1}{\pi}\bigl(\Im(\log(N(\lambda+i\varepsilon)))h,h\bigr)_\cG = (\Xi(\lambda)h,h)_\cG\, \text{ for a.e.~$\lambda \in \dR$.}
\end{equation}
Let $(\varphi_j)_{j \in J}$ be an orthonormal basis in $\cG$. Then 
\begin{equation}\label{kkl}
 \lim_{\varepsilon\downarrow 0}\frac{1}{\pi}\bigl(\Im(\log(N(\lambda+i\varepsilon)))\varphi_j,\varphi_j\bigr)_\cG = (\Xi(\lambda)\varphi_j,\varphi_j)_\cG
\end{equation}
holds for all $\lambda\in\dR\backslash\cA_j$, where $\cA_j\subset\dR$, $j \in J$, is a set of Lebesgue measure zero. The countable union $\cA:=\cup_{j \in J}\cA_j$ is also
a set of Lebesgue measure zero and for all $\lambda\in\dR\backslash\cA$ and all $\varphi_j$ one has \eqref{kkl}. Taking into acount that $0\leq \Xi(\lambda)\in\sS_1(\cG)$ 
for a.e. $\lambda\in\dR$ this implies
\begin{equation*}
 \sum_{j \in J} \lim_{\varepsilon\downarrow 0}\frac{1}{\pi}\bigl(\Im(\log(N(\lambda+i\varepsilon)))\varphi_j,\varphi_j\bigr)_\cG 
 = \sum_{j \in J} (\Xi(\lambda)\varphi_j,\varphi_j)_\cG=\tr_{\cG}(\Xi(\lambda))
\end{equation*}
for a.e. $\lambda\in\dR$, that is, \eqref{trtrxi} holds.

In the special case that \eqref{2k} holds with $k=0$ the formula \eqref{jajagut} has the form
\begin{equation*}
  \frac{d}{d\zeta}\log (N(\zeta)) 
  = \int_\dR\frac{1}{(\lambda - \zeta)^{2}}\,\Xi(\lambda)\, d\lambda \in\sS_1(\cG),
  \quad  \zeta \in I. 
 \end{equation*}
 Since $0\leq \Xi(\lambda)\in\sS_1(\cG)$ for a.e. $\lambda\in\dR$ we conclude
\begin{equation}\label{jkljkl}
 \Im(\log(N(z)))= \int_\dR \frac{\Im (z)}{\vert \lambda- z\vert^2}\,\Xi(\lambda)\,d\lambda\in\sS_1(\cG)
\end{equation} 
for all $z\in\dC\backslash\dR$. The last assertion on the existence of the limit $\Im(\log(N(\lambda+i0)))$ for a.e. $\lambda\in\dR$ in $\sS_1(\cG)$ is an immediate consequence of 
\eqref{jkljkl} and well-known results in \cite{BE67,N87,N89} (cf.\ \cite[Theorem 2.2\,$(iii)$]{GMN99}).
\end{proof}

The following lemma will be useful in the proof of our main result, Theorem \ref{mainssf2}, in the next section; it also provides a sufficient condition for the assumption \eqref{2k} in Proposition~\ref{nprop}.

\begin{lemma}\label{prooflater2}
Let $N:\dC\backslash\dR\rightarrow\cL(\cG)$ be a Nevanlinna function such that $N(z)^{-1}\in\cL(\cG)$ for some, and hence for all $z\in\dC\backslash\dR$. Let $\ell\in\dN$ and assume that
 \begin{equation}\label{nassi}
   \frac{d^j}{dz^j} N (z)\in\sS_\frac{l}{j}(\cG),\quad j=1,\dots,\ell,
 \end{equation}
holds for all $z\in\dC\backslash\dR$. Then 
\begin{equation}\label{ddss}
 \frac{d^{\ell}}{dz^{\ell}}\log (N(z))\in\sS_1(\cG)\, \text{ and } \, \frac{d^{\ell-1}}{dz^{\ell-1}}\left(N(z)^{-1}\frac{d}{dz} N (z)\right)\in\sS_1(\cG)
\end{equation}
and
\begin{equation}\label{tracesoso}
 \tr_{\cG}\left(\frac{d^{\ell-1}}{dz^{\ell-1}}\left(N(z)^{-1}\frac{d}{dz} N(z)\right)\right)=\tr_{\cG}\left(\frac{d^{\ell}}{dz^{\ell}}\log (N(z))\right)
\end{equation}
hold for all $z\in\dC\backslash\dR$. 
 
Furthermore, if  $N$ admits an analytic continuation by reflection with respect to an open interval $I\subset\dR$ 
such that $\sigma(N(z))\subset (\varepsilon,\infty)$ for some 
$\varepsilon>0$ and  all $z \in I$, and \eqref{nassi} is satisfied for $z \in I$, then also the assertions \eqref{ddss} and \eqref{tracesoso} are valid for all $z \in I$. 
\end{lemma}
\begin{proof}
We first proof Lemma~\ref{prooflater2} for the case $\ell=1$. Assume that 
\begin{equation}\label{dfv}
   \frac{d}{dz} N (z)\in\sS_1(\cG)
 \end{equation}
holds for $z\in\dC_+$ (the proof works also for $z \in I$ if  $N$ admits an analytic continuation by reflection with respect to $I$ 
and $\sigma(N(z))\subset (\varepsilon,\infty)$ holds for some 
$\varepsilon>0$ and  all $z \in I$). One notes that $N(z)^{-1}\in\cL(\cG)$ implies the second 
assertion in \eqref{ddss} for $\ell=1$. In addition, one observes that
$\log(N(z))$ is well-defined and analytic for $z\in\dC_+$ according to \eqref{logn} and 
Theorem~\ref{nthm}. Since $0\in\rho(N(z))$ and 
\begin{equation}\label{1t}
\big\Vert (N(z)+ i \lambda I_{\cG})^{-1}\big\Vert_{\cL(\cG)} \leq \lambda^{-1},\quad \lambda > 0 
\end{equation}
(cf.\ the proof of Lemma~\ref{loglem} and \cite[Proof of Lemma~2.6\,$(i)$]{GMN99}), it follows by  the dominated convergence theorem that
\begin{align*}
& \frac{d}{dz}
 \bigl(\log (N(z))\varphi,\psi\bigr)_{\cG}   \\
 & \quad =i\int_0^\infty \left((N(z)+ i \lambda I_{\cG})^{-1} \left(\frac{d}{dz} N(z)\right)
 (N(z)+ i \lambda I_{\cG})^{-1}\varphi,\psi\right)_{\cG} d\lambda 
\end{align*}
holds for all $\varphi,\psi\in\cG$ and all $z\in\dC_+$, and hence 
\begin{equation}\label{ddl}
 \frac{d}{dz}\log (N(z))=i\int_0^\infty (N(z)+ i \lambda I_{\cG})^{-1} \left(\frac{d}{dz} N(z)\right) (N(z)+ i \lambda I_{\cG})^{-1}\, d\lambda, \quad z\in\dC_+.
\end{equation}

The assumption \eqref{dfv} yields  
\begin{equation*}
 (N(z)+ i \lambda I_{\cG})^{-1} \left(\frac{d}{dz} N(z)\right)(N(z)+ i \lambda I_{\cG})^{-1}\in\sS_1(\cG), 
 \quad \lambda \geq 0. 
\end{equation*}
From \eqref{1t} and the properties of the trace class norm 
$\Vert\cdot\Vert_{\sS_1(\cG)}$ one gets 
\begin{align*}
 \left\Vert(N(z)+ i \lambda I_{\cG})^{-1} \left(\frac{d}{dz} N(z)\right)(N(z)+ i \lambda I_{\cG})^{-1}\right\Vert_{\sS_1(\cG)}\leq \frac{1}{\lambda^2}\left\Vert \frac{d}{dz} N(z)\right\Vert_{\sS_1(\cG)},& \\ 
\lambda > 0,&
\end{align*}
and hence the integral in \eqref{ddl} exists in trace class norm, that is, the first assertion in \eqref{ddss} holds for $\ell=1$. 
In order to prove \eqref{tracesoso} for $\ell=1$ we use \eqref{ddl} and cyclicity of the trace 
(i.e., $\tr_{\cG}(CD) = \tr_{\cG}(DC)$ whenever $C,D\in\cL(\cG)$ such that $CD,DC\in\sS_1(\cG)$) and obtain
\begin{equation*}
\begin{split}
& \tr_{\cG}\left( \frac{d}{dz}\log (N(z))\right)  \\
& \quad =\tr_{\cG}\left(i\int_0^\infty (N(z)+ i \lambda I_{\cG})^{-1} \left(\frac{d}{dz} N(z)\right)(N(z)+ i \lambda I_{\cG})^{-1}\, d\lambda\right)\\
 & \quad =i\int_0^\infty \tr_{\cG}\left((N(z)+ i \lambda I_{\cG})^{-1} \left(\frac{d}{dz} N(z)\right)(N(z)+ i \lambda I_{\cG})^{-1}\right) \, d\lambda\\
 & \quad =i\int_0^\infty \tr_{\cG}\left((N(z)+ i \lambda I_{\cG})^{-2} \frac{d}{dz} N(z)\right)\, d\lambda\\
 & \quad =\int_0^\infty \tr_{\cG}\left(-\frac{d}{d\lambda}(N(z)+ i \lambda I_{\cG})^{-1} \frac{d}{dz} N(z)\right)\, d\lambda\\
 & \quad =-\int_0^\infty \frac{d}{d\lambda}\tr_{\cG}\left((N(z)+ i \lambda I_{\cG})^{-1} \frac{d}{dz} N(z)\right)\, d\lambda\\
 &\quad =\tr_{\cG}\left(N(z)^{-1} \frac{d}{dz} N(z)\right). 
 \end{split}
 \end{equation*}
Here we have used 
\begin{equation}\label{echt?}
 \lim_{\lambda \rightarrow +\infty} \tr_{\cG}\left((N(z)+ i \lambda I_{\cG})^{-1} 
 \frac{d}{dz} N(z)\right)= 0
\end{equation}
in the last step, which follows from 
\begin{equation*}
 \left\Vert (N(z)+ i \lambda I_{\cG})^{-1} \frac{d}{dz} N(z)\right\Vert_{\sS_1(\cG)}
 \leq \frac{1}{\lambda}\left\Vert \frac{d}{dz} N(z)\right\Vert_{\sS_1(\cG)}, \quad 
 \lambda > 0. 
\end{equation*}
Thus, we have shown Lemma~\ref{prooflater2} for the case $\ell=1$.

Next we provide the proof of Lemma~\ref{prooflater2} for the case $\ell=2$, sketch the case $\ell=3$, and leave the more technical proof by induction to the reader. For brevity, 
the derivative with respect to $z$ will subsequently often be denoted by $'$.
In the case $\ell=2$ one computes 
\begin{equation}\label{2er}
 \frac{d}{dz}\left(N(z)^{-1} N^\prime(z)\right)=N(z)^{-1}N^{\prime\prime}(z)-N(z)^{-1} 
 N^\prime(z)N(z)^{-1}N^\prime(z), 
\end{equation}
and since $N^\prime(z)\in\sS_2(\cG)$ and $N^{\prime\prime}(z)\in\sS_1(\cG)$ by \eqref{nassi}, 
it follows that the operator in \eqref{2er} is trace class.
Furthermore, making use of 
\eqref{ddl} we find with the same arguments as in the proof of \eqref{ddl} that
\begin{align}\label{jkm}
 &\frac{d^2}{dz^2}\log (N(z))=i\int_0^\infty (N(z)+ i \lambda I_{\cG})^{-1} N^{\prime\prime}(z) 
 (N(z)+ i \lambda I_{\cG})^{-1}\, d\lambda  \\
 &\quad -i\int_0^\infty 2(N(z)+ i \lambda I_{\cG})^{-1} N^\prime(z)(N(z)+ i \lambda I_{\cG})^{-1} 
 N^\prime(z)(N(z)+ i \lambda I_{\cG})^{-1} \, d\lambda.  \no
\end{align}
Again it follows that both integrands are trace class operators for
all $\lambda \geq 0$ and as above one verifies that \eqref{jkm} is a trace class operator.  Using cyclicity of the trace again, one obtains in a similar way as above that
\begin{align*}
\begin{split}
& \tr_{\cG}\left(\frac{d^2}{dz^2}\log (N(z))\right) 
=i\int_0^\infty\tr_{\cG}\left((N(z)+ i \lambda I_{\cG})^{-2} N^{\prime\prime}(z) \right) d\lambda \\
& \qquad - i\int_0^\infty\tr_{\cG}\left((N(z)+ i \lambda I_{\cG})^{-2} N^\prime(z)(N(z)+ i \lambda I_{\cG})^{-1} N^\prime(z)\right) d\lambda \\
& \qquad - i\int_0^\infty\tr_{\cG}\left((N(z)+ i \lambda I_{\cG})^{-1} N^\prime(z)(N(z)+ i \lambda I_{\cG})^{-2} N^\prime(z)\right) d\lambda\\
& \quad = -\int_0^\infty \tr_{\cG}\bigg(\frac{d}{d\lambda}\left((N(z)+ i \lambda I_{\cG})^{-1} N^{\prime\prime}(z)\right)\bigg)  d\lambda  \\
& \qquad + \int_0^\infty \tr_{\cG}\bigg(\frac{d}{d\lambda}\left((N(z)+ i \lambda I_{\cG})^{-1} N^\prime(z)(N(z)+ i \lambda I_{\cG})^{-1} N^\prime(z)\right)\bigg) d\lambda\\
& \quad =\tr_{\cG}\bigl(N(z)^{-1}N^{\prime\prime}(z)\bigr)-\tr_\cG\bigl(N(z)^{-1} N^\prime(z)N(z)^{-1} N^\prime(z)\bigr)\\
& \quad =\tr_{\cG}\left(\frac{d}{dz}\left(N(z)^{-1}\frac{d}{dz} N(z)\right)\right).
\end{split} 
\end{align*}

For the case $\ell=3$ one notes that 
\begin{equation*}
\begin{split}
 &\frac{d^2}{dz^2}\left(N(z)^{-1} N^\prime(z)\right)\\
 &\quad=N(z)^{-1}N^{\prime\prime\prime}(z)+2N(z)^{-1}N^\prime(z)N(z)^{-1}N^\prime(z) 
 N(z)^{-1}N^\prime(z)\\
 &\quad\quad -N(z)^{-1}N^{\prime\prime}(z)N(z)^{-1}N^\prime(z) 
 - 2N(z)^{-1}N^\prime(z)N(z)^{-1}N^{\prime\prime}(z) 
\end{split}
 \end{equation*}
is a trace class operator since it is assumed that $N^\prime(z)\in\sS_3(\cG)$, $N^{\prime\prime}(z) \in\sS_{3/2}(\cG)$, and $N^{\prime\prime\prime}(z)\in\sS_1(\cG)$.
Similarly,
\begin{equation*}
\begin{split}
 &\frac{d^3}{dz^3}\log (N(z))=i\int_0^\infty (N(z)+ i \lambda I_{\cG})^{-1} 
 N^{\prime\prime\prime}(z) 
 (N(z)+ i \lambda I_{\cG})^{-1}\, d\lambda\\
 &\quad +i\int_0^\infty 6\bigl((N(z)+ i \lambda I_{\cG})^{-1} N^\prime(z)\bigr)^3 (N(z)+ i \lambda I_{\cG})^{-1} d\lambda\\
 &\quad -i\int_0^\infty 3(N(z)+ i \lambda I_{\cG})^{-1} N^{\prime\prime}(z)(N(z)+ i \lambda I_{\cG})^{-1} 
 N^\prime(z)(N(z)+ i \lambda I_{\cG})^{-1} d\lambda\\
 &\quad -i\int_0^\infty 3(N(z)+ i \lambda I_{\cG})^{-1} N^{\prime}(z)(N(z)+ i \lambda I_{\cG})^{-1} N^{\prime\prime}(z)(N(z)+ i \lambda I_{\cG})^{-1} d\lambda
\end{split}
\end{equation*}
and each integrand is a trace class operator for
all $\lambda \geq 0$. Using once more cyclicity of the trace one obtains in a similar way as above that
\begin{equation*}
\tr_{\cG}\left(\frac{d^3}{dz^3}\log (N(z))\right)=
 \tr_{\cG}\left(\frac{d^2}{dz^2}\left(N(z)^{-1}\frac{d}{dz} N(z)\right)\right).
\end{equation*}
\end{proof}

\section{A representation of the spectral shift function in terms of the Weyl function}\label{ssfsec}

Let $A$ and $B$ be self-adjoint operators in a separable Hilbert space $\sH$ and assume that the closed symmetric operator $S=A\cap B$, that is,
\begin{equation}\label{jass}
 Sf=Af=Bf,\quad\dom(S) = \bigl\{f\in\dom(A)\cap\dom(B) \, | \, Af=Bf\bigr\},
\end{equation}
is densely defined.  
According to Proposition~\ref{haha} we can choose a quasi boundary triple $\{\cG,\Gamma_0,\Gamma_1\}$ 
with $\gamma$-field $\gamma$ and Weyl function $M$ such that
\begin{equation}\label{hoho2}
 A=T\upharpoonright\ker(\Gamma_0)\, \text{ and } \, B=T\upharpoonright\ker(\Gamma_1),
\end{equation}
and
\begin{equation}\label{resab}
 (B - z I_{\sH})^{-1}-(A - z I_{\sH})^{-1}=-\gamma(z) M(z)^{-1}\gamma({\ol z})^*,\quad z \in \rho(A)\cap\rho(B).
\end{equation}

In the next theorem we find an explicit expression for a spectral shift function of the pair $\{A,B\}$ in terms of the
Weyl function $M$ (cf.\ \cite[Theorem 1]{LSY01}) for the case that the difference of (the first powers of) the resolvents $A$ and $B$ is a rank one operator, \cite[Theorem 4.1]{BMN08} 
for the finite-rank case, and \cite[Theorem 3.4 and Remark 3.5]{MN15} for a different representation via a perturbation determinant involving the Weyl function and 
boundary parameters
of an ordinary boundary triple. In the present situation of infinite dimensional perturbations and differences of higher powers of resolvents  
a much more careful analysis is necessary, in particular, the properties of the logarithm of operator-valued Nevanlinna functions 
discussed in Section~\ref{logsec} will play an essential role.
In Theorem~\ref{mainssf2} an implicit sign condition on the perturbation is imposed via the resolvents which leads to a nonnegative spectral shift function; this condition
will be weakend afterwards (cf.\ \eqref{sisi} and \eqref{ssfab}). In the special case that $A$ and $B$ are semibounded operators the sign condition 
\eqref{sign333} is equivalent to the inequality
$\mathfrak t_A\leq \mathfrak t_B$ of the semibounded closed quadratic forms $\mathfrak t_A$
and $\mathfrak t_B$ corresponding to $A$ and $B$. In order to ensure that for some $k \in \bbN_0$ the difference of the 
$2k+1$th-powers of the resolvents of $A$ and $B$ is 
a trace class operator a set of $\sS_p$-conditions on the $\gamma$-field and the Weyl function are imposed. In the applications to scattering problems for different self-adjoint realizations of 
elliptic PDEs, and Schr\"{o}dinger operators with compactly supported potentials and singular 
potentials in Sections~\ref{ap1sec}, \ref{ap11sec}, and \ref{ap2sec}, these conditions are satisfied.

\begin{theorem}\label{mainssf2}
Let $A$ and $B$ be self-adjoint operators in a separable Hilbert space $\sH$ and assume that for some $\zeta_0 \in\rho(A)\cap\rho(B)\cap\dR$ the sign condition
 \begin{equation}\label{sign333}
  (A-\zeta_0 I_{\sH})^{-1}\geq (B- \zeta_0 I_{\sH})^{-1}
 \end{equation}
 holds. Let the closed symmetric operator $S=A\cap B$ in \eqref{jass} be densely defined and 
 let $\{\cG,\Gamma_0,\Gamma_1\}$ be a quasi boundary triple with $\gamma$-field $\gamma$ and Weyl function $M$ 
such that
\eqref{hoho2}, and hence also \eqref{resab} holds. Assume that $M(z_1)$, $M(z_2)^{-1}$ are 
bounded $($not necessarily everywhere defined\,$)$ 
operators in $\cG$ for some $z_1, z_2 \in \rho(A)\cap\rho(B)$
and that for some $k \in \bbN_0$, all $p,q \in \bbN_0$ and all $z \in \rho(A)\cap\rho(B)$, 
\begin{equation}\label{ass1}
 \left(\frac{d^p}{d z^p}\overline{\gamma(z)}\right)\frac{d^q}{dz^q}\bigl( M(z)^{-1} \gamma({\ol z})^*\bigr)\in\sS_1(\sH),\quad p+q=2k,
\end{equation}
\begin{equation}\label{ass2}
 \left(\frac{d^q}{dz^q}\bigl( M(z)^{-1} \gamma({\ol z})^*\bigr)\right)\frac{d^p}{dz^p}\overline{\gamma(z)}\in\sS_1(\cG),\quad p+q=2k,
\end{equation}
and 
 \begin{equation}\label{ass3}
   \frac{d^j}{dz^j} \overline{M (z)}\in\sS_{(2k+1)/j}(\cG),\quad j=1,\dots,2k+1.
 \end{equation}
Then the following assertions $(i)$ and $(ii)$ hold:
 \begin{itemize}
  \item [$(i)$] The difference of the $2k+1$th-powers of the resolvents of $A$ and $B$ is 
a trace class operator, that is,
\begin{equation}\label{tracehop}
 \big[(B - z I_{\sH})^{-(2k+1)}-(A - z I_{\sH})^{-(2k+1)}\big] \in \sS_1(\sH)
\end{equation}
holds for all $z\in\rho(A)\cap\rho(B)$. 
 \item [$(ii)$] For any orthonormal basis $(\varphi_j)_{j \in J}$ in $\cG$ the function 
  \begin{equation}\label{ssfssf}
   \xi(\lambda)=
   \sum_{j \in J} \lim_{\varepsilon\downarrow 0}\pi^{-1}\bigl(\Im\big(\log\big(\overline{M(\lambda+i\varepsilon)}\big)\big)\varphi_j,\varphi_j\bigr)_\cG
   \, \text{ for a.e.~$\lambda \in \dR$}, 
  \end{equation}
is a spectral shift function for the pair $\{A,B\}$ such that $\xi(\lambda)=0$ in an open neighborhood 
of $\zeta_0$; the function $\xi$ does not depend on the choice of the orthonormal basis $(\varphi_j)_{j \in J}$. In particular, the trace formula
\begin{equation*}
 \tr_{\sH}\bigl( (B - z I_{\sH})^{-(2k+1)}-(A - z I_{\sH})^{-(2k+1)}\bigr) 
 = - (2k+1) \int_\dR \frac{\xi(\lambda)\, d\lambda}{(\lambda - z)^{2k+2}} 
\end{equation*}
is valid for all $z \in \rho(A)\cap\rho(B)$.
\end{itemize}
\end{theorem}
\begin{proof}
{\it Step 1.}
In this step we show that the Nevanlinna function $ z \mapsto\overline{M(z)}$ satisfies the assumptions of Theorem~\ref{nthm} and admits an analytic continuation by reflection with respect to an open interval $I_{\zeta_0}\subset\dR$, such that $\sigma \big(\overline{M(z)}\big)\subset (\varepsilon,\infty)$ for some 
$\varepsilon>0$ and  all $z \in I_{\zeta_0}$, where $I_{\zeta_0}\subset\rho(A)\cap\rho(B)$ is a suitable small open interval in $\dR$ with $\zeta_0 \in I_{\zeta_0}$.
Hence by Theorem~\ref{nthm} there exists a weakly Lebesgue measurable operator function $\lambda \mapsto \Xi(\lambda)\in\cL(\cG)$ on $\dR$ such that 
\begin{equation}\label{bitte4}
\Xi(\lambda)=\Xi(\lambda)^*\, \text{ and } \, 
 0\leq\Xi(\lambda)\leq I_\cG\, \text{ for a.e.~$\lambda \in \dR$,}
\end{equation}
and the Nevanlinna function $\log \big(\overline M\big)$ admits an integral representation of the form
\begin{equation}\label{bittesehr4}
 \log\big(\overline{M(z)}\big)  
 = \Re\big(\log(\overline {M(i)}\big)\big) +\int_\dR\left(\frac{1}{\lambda - z}-\frac{\lambda}{1 + \lambda^2}\right)\Xi(\lambda)\, d\lambda,
\end{equation}
valid for all $z\in(\dC\backslash\dR)\cup I_{\zeta_0}$, and  $\Xi(\lambda)=0$ for 
a.e.~$\lambda \in  I_{\zeta_0}$.

First, it follows from \eqref{gutgut} and the assumption that $M(z_1)$ is bounded for some 
$z_1\in\rho(A)$ that $M(z)$ is bounded for all $z \in \rho(A)$
and hence the closures are bounded operators defined on $\cG$, that is, 
\begin{equation}\label{mb}
\overline{M(z)} \in \cL(\cG), \quad z \in \rho(A).
\end{equation}
Moreover, by \eqref{trf} 
\begin{equation}\label{imi}
\Im \big(\overline{M(z)}\big) \geq 0,\quad z\in\dC_+.
\end{equation}
Since $-M^{-1}$ is the Weyl function corresponding to the quasi boundary triple 
$\{\cG,\Gamma_1,-\Gamma_0\}$, where $B=\ker (\Gamma_1)$ is self-adjoint
according to \eqref{hoho2}, it follows from the assumption that $M(z_2)^{-1}$ is bounded for some 
$z_2\in\rho(B)$ that $M(z)^{-1}$ is bounded for all $z \in \rho(A)\cap\rho(B)$, that is,
\begin{equation}\label{m1b}
\overline{M(z)^{-1}}\in\cL(\cG)\, \text{ for all } \, z \in \rho(A)\cap\rho(B).
\end{equation} 
Therefore, taking into account \eqref{mb}, \eqref{imi}, and \eqref{m1b}, it follows 
that the logarithm 
\begin{equation*}
  z \mapsto  \log\big(\overline{M(z)}\big)\in \cL(\cG)
\end{equation*}
is well-defined by
\begin{equation}\label{logim}
 \log \big(\overline{M(z)}\big):=-i\int_0^\infty \bigl[\big(\overline{M(z)} 
 + i \lambda I_{\cG}\big)^{-1}-(1+ i \lambda)^{-1}I_\cG\bigr] \, d\lambda,\quad z\in\dC_+,
\end{equation}
and
\begin{equation}\label{logim2}
 \log\big(\overline{M(z)}\big):=\bigl(\log\big(\overline{M({\ol z})}\big)\bigr)^*,\quad z\in\dC_-;
\end{equation}
see \eqref{logn}--\eqref{logn2} in Section~\ref{logsec} and \cite[Lemma 2.6]{GMN99}.
We claim that the function $\overline{M}$ has the property 
\begin{equation}\label{scond}
 \sigma \big(\overline{M(z)}\big) \subset (\varepsilon,\infty)
\end{equation}
for some $\varepsilon>0$ and all $z \in I_{\zeta_0}$, where $I_{\zeta_0}$ is a suitable small open interval in $\dR$ with $\zeta_0 \in I_{\zeta_0}$. In fact, due to \eqref{resab} and the sign condition
\eqref{sign333}, one has 
\begin{equation*}
0\leq \bigl((A - \zeta_0 I_{\sH})^{-1}f-(B - \zeta_0 I_{\sH})^{-1}f,f\bigr)_{\sH}=\bigl(M(\zeta_0)^{-1}\gamma(\zeta_0)^*f,\gamma(\zeta_0)^*f\bigr)_{\cG},\quad f\in\sH,
\end{equation*}
and since $\ran(\gamma(\zeta_0)^*)$ is dense in $\cG$ (see \eqref{gamden}), it follows that the bounded operator $M(\zeta_0)^{-1}$ 
is nonnegative. The same is true for $M(\zeta_0)$ and the closure
$\overline{M(\zeta_0)}$, and from \eqref{m1b} one concludes $\sigma\big(\overline{M(\zeta_0)}\big)\subset (\varepsilon,\infty)$ for some $\varepsilon>0$. Since $\zeta_0 \in\rho(A)\cap\rho(B)$
the Nevanlinna function $\overline{M}$ admits an analytic continuation by reflection with respect to 
a real neighborhood of $\zeta_0$, and it follows that \eqref{scond} holds for all $\lambda$ in a sufficiently small
interval $I_{\zeta_0} \subset \rho(A)\cap\rho(B)\cap\dR$ with $\zeta_0 \in I_{\zeta_0}$.

\vskip 0.2cm\noindent
{\it Step 2.} In this step we show that for $z\in(\dC\backslash\dR)\cup I_{\zeta_0}$, the trace class property \eqref{tracehop} holds, and that
\begin{equation}\label{nimmdie}
 \tr_{\sH}\bigl((B - z I_{\sH})^{-(2k+1)}-(A - z I_{\sH})^{-(2k+1)}\bigr) 
 = \tr_{\cG}\left(\frac{-1}{(2k)!}\frac{d^{2k+1}}{dz^{2k+1}} \log \big(\overline{M(z)}\big) \right).
\end{equation}
In fact, for $z\in(\dC\backslash\dR)\cup I_{\zeta_0}$ one computes 
\begin{equation*}
\begin{split}
 &(B - z I_{\sH})^{-(2k+1)}-(A - z I_{\sH})^{-(2k+1)}\\
 &\quad=\frac{1}{(2k)!}\frac{d^{2k}}{dz^{2k}}\bigl((B - z I_{\sH})^{-1}
 - (A - z I_{\sH})^{-1}\bigr)\\
 &\quad=\frac{-1}{(2k)!}\frac{d^{2k}}{dz^{2k}} \bigl(\gamma(z) M(z)^{-1}\gamma({\ol z})^*\bigr)\\
 &\quad=\frac{-1}{(2k)!}\frac{d^{2k}}{dz^{2k}} \bigl(\overline{\gamma(z)} M(z)^{-1}\gamma({\ol z})^*\bigr)\\
 &\quad=\frac{-1}{(2k)!}\sum_{\substack{p+q=2k \\[0.2ex] p,q\ge0}} \begin{pmatrix} 2k \\ p \end{pmatrix}\left(\frac{d^p}{d z^p}\,\overline{\gamma(z)}\right) 
 \frac{d^q}{d z^q}\bigl( M(z)^{-1} \gamma({\ol z})^*\bigr),
 \end{split}
\end{equation*}
and by assumption \eqref{ass1} each summand is a trace class operator; in the last step the product rule for 
holomorphic operator functions was applied, see, e.g. \cite[(2.6)]{BLL13-3}. This proves \eqref{tracehop}.  Furthermore, 
making use of both assumptions \eqref{ass1} and \eqref{ass2}, the 
cyclicity of the trace (see, e.g., \cite[Theorem 7.11 (b)]{W80}), and
\begin{equation}\label{hehe}
 \frac{d}{dz} \overline{M(z)}=\gamma({\ol z})^*\overline{\gamma(z)},
 \quad z \in \rho(A),
\end{equation}
one obtains
\begin{equation*}
 \begin{split}
  &\tr_{\sH}\bigl((B - z I_{\sH})^{-(2k+1)}-(A - z I_{\sH})^{-(2k+1)}\bigr)\\
  &\quad=\tr_{\sH}\left(\frac{-1}{(2k)!}\sum_{\substack{p+q=2k \\[0.2ex] p,q\ge0}} \begin{pmatrix} 2k \\ p \end{pmatrix}\left(\frac{d^p}{d z^p}\,\overline{\gamma(z)}\right)
  \frac{d^q}{d z^q}\bigl( M(z)^{-1} \gamma({\ol z})^*\bigr) \right)\\
  &\quad=\frac{-1}{(2k)!}\sum_{\substack{p+q=2k \\[0.2ex] p,q\ge0}} \begin{pmatrix} 2k \\ p \end{pmatrix}\tr_{\sH}\left(\left(\frac{d^p}{d z^p}\,\overline{\gamma(z)}\right)
  \frac{d^q}{d z^q}\bigl( M(z)^{-1} \gamma({\ol z})^*\bigr)\right)\\
  &\quad=\frac{-1}{(2k)!}\sum_{\substack{p+q=2k \\[0.2ex] p,q\ge0}} \begin{pmatrix} 2k \\ p \end{pmatrix}\tr_{\cG}\left(
  \left(\frac{d^q}{d z^q}\bigl( M(z)^{-1} \gamma({\ol z})^*\bigr)\right)\frac{d^p}{d z^p}\,\overline{\gamma(z)}\right)\\
  &\quad=\tr_{\cG}\left(\frac{-1}{(2k)!}\sum_{\substack{p+q=2k \\[0.2ex] p,q\ge0}} \begin{pmatrix} 2k \\ q \end{pmatrix}
  \left(\frac{d^q}{d z^q}\bigl( M(z)^{-1} \gamma({\ol z})^*\bigr)\right)\frac{d^p}{d z^p}\,\overline{\gamma(z)}\right)\\
  &\quad=\tr_{\cG}\left(\frac{-1}{(2k)!}\frac{d^{2k}}{dz^{2k}} \bigl(M(z)^{-1}\gamma({\ol z})^*\overline{\gamma(z)}\bigr)\right)\\
  &\quad=\tr_{\cG}\left(\frac{-1}{(2k)!}\frac{d^{2k}}{dz^{2k}} \left(\overline{M(z)^{-1}}\frac{d}{dz}\overline{M(z)}\right)\right).
 \end{split}
\end{equation*}
Noting that assumption \eqref{ass3} and Lemma~\ref{prooflater2} with $\ell=2k+1$ imply
\begin{equation}\label{hurra2}
 \frac{d^{2k+1}}{dz^{2k+1}} \log \big(\overline{M(z)}\big)\in\sS_1(\cG),
\end{equation}
and that 
\begin{equation}\label{kmj}
 \tr_{\cG}\left(\frac{d^{2k}}{dz^{2k}}\left(\overline{M(z)^{-1}}\frac{d}{dz} \overline{M(z)}\right)\right)
 =\tr_{\cG}\left(\frac{d^{2k+1}}{dz^{2k+1}}\log \big(\overline{M(z)}\big)\right), 
\end{equation}
one concludes the trace formula \eqref{nimmdie}.
 
\vskip 0.2cm\noindent
{\it Step 3.} Now we complete the proof of Theorem~\ref{mainssf2}. Since \eqref{hurra2} is 
valid for all $z \in I_{\zeta_0}$ the assumption \eqref{2k} in
Proposition~\ref{nprop} is satisfied. 
It then follows from 
Proposition~\ref{nprop} that $0 \leq \Xi(\lambda)\in\sS_1(\cG)$ for a.e.~$\lambda \in \dR$ and
\begin{equation}\label{trtrxi2}
 \tr_{\cG}(\Xi(\lambda))=\sum_{j \in J} \lim_{\varepsilon\downarrow 0}\pi^{-1}\bigl(\Im\big(\log\big(\overline{M(\lambda+i\varepsilon)}\big)\big)\varphi_j,\varphi_j\bigr)_\cG
\end{equation}
holds for any orthonormal basis $(\varphi_j)_{j \in J}$ in $\cG$ and for a.e. 
$\lambda\in\dR$. Furthermore, from \eqref{bittesehr4} one obtains 
\begin{equation*}
 \frac{d^{2k+1}}{dz^{2k+1}}\log\big(\overline{M(z)}\big) 
 = (2k+1)! \int_\dR \frac{1}{(\lambda - z)^{2k+2}}\, \Xi(\lambda)\, d\lambda, 
 \quad z\in(\dC\backslash\dR)\cup I_{\zeta_0},
\end{equation*}
and hence 
\begin{equation}\label{fastgeschafft2}
 \tr_{\sH}\bigl((B - z I_{\sH})^{-(2k+1)}-(A - z I_{\sH})^{-(2k+1)}\bigr) 
 = - (2k+1) \int_\dR \frac{\tr_{\cG}(\Xi(\lambda))\, d\lambda}{(\lambda- z)^{2k+2}}
\end{equation}
for all $z\in(\dC\backslash\dR)\cup I_{\zeta_0}$
by \eqref{nimmdie}. It also follows from \eqref{fastgeschafft2}
that 
\begin{equation}\label{inti}
 \int_\dR \frac{\tr_{\cG}(\Xi(\lambda)) \, d\lambda}{(1+\vert \lambda \vert)^{2k+2}} < \infty
\end{equation}
holds and together with \eqref{trtrxi2}, \eqref{fastgeschafft2}, and \eqref{inti} we conclude that the function 
\begin{equation*}
 \xi(\lambda):=\tr_{\cG}(\Xi(\lambda))=\sum_{j \in J} \lim_{\varepsilon\downarrow 0}\pi^{-1}\bigl(\Im\big(\log\big(\overline{M(\lambda+i\varepsilon)}\big)\big)\varphi_j,\varphi_j\bigr)_\cG
 \, \text{ for a.e.~$\lambda \in \dR$}
\end{equation*}
in \eqref{ssfssf}
is a spectral shift function for the pair $\{A,B\}$. Next, since 
$\tr_{\cG}(\Xi(\lambda))= \sum_{j \in J}(\Xi(\lambda)\varphi_j,\varphi_j)_\cG$ does not depend on the choice of the orthonormal basis $(\varphi_j)_{j \in J}$, it follows that the function $\xi$ does not depend on the choice of the orthonormal basis (cf.\ 
Proposition~\ref{nprop}).
Finally, since $\Xi(\lambda)=0$ for a.e. $\lambda\in I_{\zeta_0}$ by Theorem~\ref{nthm} 
it follows that $\xi(\lambda)=0$ for a.e.  $\lambda\in I_{\zeta_0}$.
\end{proof}

In the special case $k=0$ Theorem~\ref{mainssf2} can be reformulated and slightly improved. Here the essential feature is that
Proposition~\ref{nprop} can be applied under the assumption \eqref{hahagut}, so that the limit $\Im(\log(\overline {M(\lambda+i0)}))$ exists in $\sS_1(\cG)$ 
for a.e. $\lambda\in\dR$.

\begin{corollary}\label{mainthmcorchen}
Let $A$ and $B$ be self-adjoint operators in a separable Hilbert space $\sH$ and assume that for some $\zeta_0 \in \rho(A)\cap\rho(B)\cap\dR$ the sign condition
 \begin{equation*}
  (A-\zeta_0 I_{\sH})^{-1}\geq (B-\zeta_0 I_{\sH})^{-1}
 \end{equation*}
 holds. Assume that the closed symmetric operator $S=A\cap B$ in \eqref{jass} is densely defined and 
 let $\{\cG,\Gamma_0,\Gamma_1\}$ be a quasi boundary triple with $\gamma$-field $\gamma$ and Weyl function $M$ such that \eqref{hoho2}, and hence also \eqref{resab}, hold. Assume that $M(z_1)$, 
$M(z_2)^{-1}$ are bounded $($not necessarily everywhere defined\,$)$ 
operators in $\cG$ for some $z_1,z_2\in\rho(A)$ and that $\overline{\gamma(z_0)}\in\sS_2(\cG,\sH)$ for some $z_0\in\rho(A)$. 
Then the following assertions $(i)$--$(iii)$ hold: 
\begin{itemize}
  \item [$(i)$]
The difference of the resolvents of $A$ and $B$ is 
a trace class operator, that is,
\begin{equation*}
 \big[(B - z I_{\sH})^{-1}-(A - z I_{\sH})^{-1}\big] \in \sS_1(\sH)
\end{equation*}
holds for all $z\in\rho(A)\cap\rho(B)$. 
  \item [$(ii)$] $\Im\big(\log \big(\overline{M(z)}\big)\big)\in\sS_1(\cG)$ for all $z\in\dC\backslash\dR$ and the limit 
  $$\Im\big(\log\big(\overline{M(\lambda+i 0)}\big)\big):=\lim_{\varepsilon\downarrow 0}\Im\big(\log\big(\overline{M(\lambda+i\varepsilon)}\big)\big)$$ 
  exists for a.e.~$\lambda \in \dR$ in $\sS_1(\cG)$. 
  \item [$(iii)$] The function
  \begin{equation}
   \xi(\lambda)=\pi^{-1} \tr_{\cG}\bigl(\Im\big(\log\big(\overline{M(\lambda + i0)}\big)\big)\bigr) \, \text{ for a.e.~$\lambda \in \dR$}, 
  \end{equation}
is a spectral shift function for the pair $\{A,B\}$ such that $\xi(\lambda)=0$ in an open neighborhood 
of $\zeta_0$ and the trace formula
\begin{equation*}
 \tr_{\sH}\bigl( (B - z I_{\sH})^{-1}-(A - z I_{\sH})^{-1}\bigr) 
 = -  \int_\dR \frac{\xi(\lambda)\, d\lambda}{(\lambda - z)^{2}} 
\end{equation*}
is valid for all $z \in \rho(A)\cap\rho(B)$.
\end{itemize}
\end{corollary}
\begin{proof}
 The assumption $\overline{\gamma(z_0)}\in\sS_2(\cG,\sH)$ for some $z_0\in\rho(A)$ implies $\overline{\gamma(z)}\in\sS_2(\cG,\sH)$ for all
 $z \in \rho(A)$ by \eqref{xxaa} and hence also $\gamma(z)^*\in\sS_2(\sH,\cG)$ for all
 $z \in \rho(A)$. Since $M(z)^{-1}$ is bounded for all $z \in \rho(A)\cap\rho(B)$ (see \eqref{gutgut}),  
 conditions \eqref{ass1}--\eqref{ass2} in Theorem~\ref{mainssf2} are satisfied for $k=0$ and all $z\in\rho(A)\cap\rho(B)$. 
 Furthermore,  
 \begin{equation*}
  \frac{d}{dz}\overline{M(z)}=\gamma(z)^*\overline{\gamma(z)}\in\sS_1(\cG)
 \end{equation*}
by \eqref{gammad3} and hence also condition \eqref{ass3} 
in Theorem~\ref{mainssf2} is satisfied for $k=0$. In particular, by Lemma~\ref{prooflater2} we have 
\begin{equation*}
 \frac{d}{dz} \log \big(\overline{M(z)}\big)\in\sS_1(\cG).
\end{equation*}
Furthermore, in Step 3 of the proof of Theorem~\ref{mainssf2} we can now apply Proposition~\ref{nprop}
under the assumption \eqref{hahagut}, so that \eqref{immerhin} holds with $N(\lambda+i0)$ replaced by $\overline{M(\lambda+i0)}$. Now the assertions $(i)$--$(iii)$  
in Corollary~\ref{mainthmcorchen} follow from Theorem~\ref{mainssf2} and Proposition~\ref{nprop}.
\end{proof}

In the next step we replace the sign condition \eqref{sign333} in the assumptions in Theorem~\ref{mainssf2} by some weaker comparability condition,
which is satisfied in our applications in the next sections.
Again, let $A$ and $B$ be self-adjoint operators in a separable Hilbert space $\sH$ and assume that there exists a self-adjoint operator $C$ in $\sH$ such that
\begin{equation}\label{sisi}
 (C-\zeta_A I_{\sH})^{-1}\geq (A-\zeta_A I_{\sH})^{-1}\, \text{ and } \,  
 (C-\zeta_B I_{\sH})^{-1}\geq (B-\zeta_B I_{\sH})^{-1}
\end{equation}
for some $\zeta_A\in\rho(A)\cap\rho(C)\cap\dR$ and some $\zeta_B\in\rho(B)\cap\rho(C)\cap\dR$, respectively. Assume that the closed symmetric
operators $S_A=A\cap C$ and $S_B=B\cap C$ are both densely defined and choose quasi boundary triples $\{\cG_A,\Gamma_0^A,\Gamma_1^A\}$ and  
$\{\cG_B,\Gamma_0^B,\Gamma_1^B\}$ with $\gamma$-fields $\gamma_A,\gamma_B$ and Weyl functions $M_A$, $M_B$ for
\begin{equation*}
 T_A=S_A^*\upharpoonright\bigl(\dom (A)+\dom(C)\bigr)\, \text{ and } \, T_B=S_B^*\upharpoonright\bigl(\dom (B)+\dom(C)\bigr)
\end{equation*}
such
that
\begin{equation}\label{hoho0}
 C=T_A\upharpoonright\ker(\Gamma_0^A)=T_B\upharpoonright\ker(\Gamma_0^B),
\end{equation}
and
\begin{equation}\label{hoho1}
  A=T_A\upharpoonright\ker(\Gamma_1^A)\, \text{ and } \, B=T_B\upharpoonright\ker(\Gamma_1^B),
\end{equation}
(cf.\ Proposition~\ref{haha}). Next, assume that for some $k \in \bbN_0$, 
the conditions in Theorem~\ref{mainssf2} are satisfied for the $\gamma$-fields $\gamma_A,\gamma_B$ and the Weyl functions $M_A$, $M_B$.
Then the difference of the $2k+1$-th powers of the resolvents of $A$ and $C$, and the difference of the $2k+1$-th powers of the resolvents of $B$ and $C$ are trace class operators,
and for orthonormal bases $(\varphi_j)_{j \in J}$ in $\cG_A$ and $(\psi_{\ell})_{\ell \in L}$ in $\cG_B$ ($J, L \subseteq \bbN$ appropriate index sets), 
\begin{equation}
   \xi_A(\lambda)=\sum_{j \in J}\lim_{\varepsilon\downarrow 0}
   \pi^{-1}\bigl(\Im\big(\log\big(\overline{M_A(\lambda+i\varepsilon)}\big)\big)\varphi_j,\varphi_j\bigr)_{\cG_A}\, \text{ for a.e.~$\lambda \in \dR$,}
  \end{equation}
and 
\begin{equation}
   \xi_B(\lambda)=\sum_{\ell \in L}\lim_{\varepsilon\downarrow 0}
   \pi^{-1}\bigl(\Im\big(\log\big(\overline{M_B(\lambda+i\varepsilon)}\big)\big)
   \psi_{\ell},\psi_{\ell}\bigr)_{\cG_B}\, \text{ for a.e.~$\lambda \in \dR$,}
  \end{equation}
are spectral shift functions for the pairs $\{C,A\}$ and $\{C,B\}$, respectively. It follows for $z \in \rho(A)\cap\rho(B)\cap\rho(C)$ 
that
\begin{equation*}
\begin{split}
 &\tr_{\sH}\bigl( (B - z I_{\sH})^{-(2k+1)}-(A - z I_{\sH})^{-(2k+1)}\bigr)\\
 &\quad= \tr_{\sH}\bigl( (B - z I_{\sH})^{-(2k+1)}-(C - z I_{\sH})^{-(2k+1)}\bigr)  
 \\ 
 & \qquad -\tr_{\sH}\bigl( (A - z I_{\sH})^{-(2k+1)}-(C - z I_{\sH})^{-(2k+1)}\bigr) \\
 &\quad = - (2k+1) \int_\dR \frac{[\xi_B(\lambda) - \xi_A(\lambda)] \, d\lambda }{(\lambda - z)^{2k+2}}
 \end{split}
 \end{equation*}
 and
 \begin{equation*}
  \int_\dR \frac{\vert \xi_B(\lambda)-\xi_A(\lambda)\vert \, d\lambda}{(1+\vert \lambda \vert)^{2m+2}} < \infty.
 \end{equation*}
  Therefore, 
\begin{equation} 
    \xi(\lambda) = \xi_B(\lambda)-\xi_A(\lambda)  \, \text{ for a.e.~$\lambda \in \dR$,}   \label{ssfab}
\end{equation}
 is a spectral shift function for the pair $\{A,B\}$, and 
  in the special case where $\cG_A = \cG_B := \cG$ and $(\varphi_j)_{j \in J}$ is an orthonormal basis in $\cG$, one infers that 
  \begin{align}\label{ssfabc}
   \xi(\lambda) 
   =\sum_{j \in J}\lim_{\varepsilon\downarrow 0}\pi^{-1} \Bigl(\bigl(\Im\bigl( \log \big(\overline{M_B(\lambda+i\varepsilon)}\big) - \log \big(\overline{M_A(\lambda + i \varepsilon)}\big)\bigr)
   \varphi_j,\varphi_j\Bigr)_\cG &\\
   \text{for a.e. $\lambda\in\dR$.} &\nonumber
  \end{align}
We emphasize that in contrast to the spectral shift function in Theorem~\ref{mainssf2}, the spectral shift 
function $\xi$ in \eqref{ssfab} and \eqref{ssfabc} is not necessarily nonnegative.

\section{Elliptic differential operators with Robin boundary conditions}\label{ap1sec}

In this section we consider a uniformly elliptic formally symmetric second-order differential expression $\cL$ on a bounded or unbounded domain in $\bbR^n$ with compact boundary, and we
determine a spectral shift function for a pair $\{A_{\beta_0},A_{\beta_1}\}$ consisting of two self-adjoint Robin-realizations 
of $\cL$. We shall assume throughout this section that the following hypothesis holds.

\begin{hypothesis}\label{hypo5}
Let $n \in \bbN$, $n \geq 2$, and $\Omega \subseteq \bbR^n$  be nonempty and open such that its 
boundary $\partial\Omega$ is nonempty, $C^\infty$-smooth, and compact. 
Consider the differential expression
\begin{equation}\label{elli}
  \cL = -\sum_{j,k=1}^n \left(
  \frac{\partial}{\partial x_j} \Bigl(a_{jk} \frac{\partial }{\partial
  x_k}\Bigr)\right)+ a
\end{equation}
on $\Omega$, where the real-valued coefficients $a_{jk}\in C^\infty(\overline\Omega)$ satisfy 
$a_{jk}(x)=a_{kj}(x)$ for all $x\in\overline\Omega$ and $j,k=1,\dots,n$, their first partial derivatives are bounded in $\overline\Omega$, 
and $a\in C^\infty(\overline\Omega)$ is a real-valued, bounded, measurable function.
Furthermore, it is assumed that $\cL$ is uniformly elliptic on $\overline\Omega$, that is, 
for some $C>0$, 
\begin{equation}\label{ellijaha}
  \sum_{j,k=1}^n a_{jk}(x)\xi_j\xi_k\geq C\sum_{k=1}^n\xi_k^2
\end{equation}
holds for all $\xi=(\xi_1,\dots,\xi_n)^\top\in\dR^n$ and $x\in\overline\Omega$. 
\end{hypothesis}

We briefly recall the definition and some mapping properties of the Dirichlet and (oblique) 
Neumann trace maps associated with the differential expression $\cL$. 
For a function $f\in C^\infty(\overline\Omega)$ we denote its trace by
$\gamma_D f=f\vert_{\partial\Omega}$ and we set
\begin{equation*}
  \gamma_\nu f=
  \sum_{j,k=1}^n a_{jk} \mathfrak n_j \frac{\partial f}{\partial x_k}
  \Bigl|_{\partial\Omega},\quad f\in C^\infty(\overline\Omega),
\end{equation*}
where $\mathfrak n(x)=(\mathfrak n_1(x),\dots, \mathfrak n_n(x))^\top$ is the unit normal vector at 
$x\in\partial\Omega$ pointing out of the domain $\Omega$.
Let $C_0^\infty(\overline\Omega):=\{h\vert_{\overline\Omega} \, | \,h\in C_0^\infty(\dR^n)\}$ and
recall that the mapping $C_0^\infty(\overline\Omega)\ni
f\mapsto\{\gamma_D f,\gamma_\nu f\}$
can be extended to a continuous surjective mapping
\begin{equation}\label{tracemapelli}
  H^2(\Omega)\ni f\mapsto \left\{\gamma_D f,\gamma_\nu f\right\}
  \in H^{3/2}(\partial\Omega)\times H^{1/2}(\partial\Omega),
\end{equation}
and that Green's second identity
\begin{equation}\label{g2}
 (\cL f,g)_{L^2(\Omega)}-(f,\cL g)_{L^2(\Omega)}=(\gamma_D f,\gamma_\nu g)_{L^2(\partial\Omega)}-(\gamma_\nu f,\gamma_D g)_{L^2(\partial\Omega)}
\end{equation}
is valid for all $f,g\in H^2(\Omega)$; cf. \cite{LM72}. We will also use the fact that 
\begin{equation}\label{gamss}
 \gamma_D f\in H^{k-1/2}(\partial\Omega)\, \text{ for all } \, f\in H^k(\Omega), \; k \in \bbN.
\end{equation}
The following lemma is a variant of \cite[Lemma~4.7]{BLL13-IEOT}; it will be useful for 
the $\sS_p$-estimates in this and the next section.

\begin{lemma}\label{usel}
Let $\Omega\subseteq \dR^n$ be as in Hypothesis~\ref{hypo5}, let 
$X\in \cL\big(L^2(\Omega),H^t(\partial\Omega)\big)$, and assume that 
$\ran(X)\subseteq H^s(\partial\Omega)$ for some $s>t\geq 0$. Then $X$ is compact and 
 \begin{equation*}
  X\in\sS_r\bigl(L^2(\Omega),H^t(\partial\Omega)\bigr)\, \text{ for all } \, r> (n-1)/(s-t).
 \end{equation*}
\end{lemma}

Assume that $\beta_0\in C^1(\partial\Omega)$ and $\beta_1\in C^1(\partial\Omega)$ are real-valued functions.
We consider the elliptic differential operators in $L^2(\Omega)$, 
\begin{equation}\label{anni}
 A_{\beta_0} f=\cL f,\quad \dom (A_{\beta_0}) 
 = \bigl\{f\in H^2(\Omega) \, \big| \,\beta_0 \gamma_D f=\gamma_\nu f\bigr\},
\end{equation}
and
\begin{equation}\label{robi}
 A_{\beta_1} f=\cL f,\quad \dom (A_{\beta_1}) 
 = \bigl\{f\in H^2(\Omega) \, \big| \,\beta_1 \gamma_D f=\gamma_\nu f\bigr\},
\end{equation}
which correspond to the densely defined, closed, semibounded quadratic forms
\begin{equation}\label{form0}
 \mathfrak a_{\beta_0}[f,g]=\sum_{j,k=1}^n \left(a_{jk} \frac{\partial f}{\partial x_k},\frac{\partial g}{\partial x_j}\right)_{L^2(\Omega)}+(af,g)_{L^2(\Omega)} - 
 (\beta_0\gamma_D f,\gamma_D g)_{L^2(\partial\Omega)},
\end{equation}
and
\begin{equation}\label{form1}
 \mathfrak a_{\beta_1}[f,g]=\sum_{j,k=1}^n \left(a_{jk} \frac{\partial f}{\partial x_k},\frac{\partial g}{\partial x_j}\right)_{L^2(\Omega)}+(af,g)_{L^2(\Omega)} - 
 (\beta_1\gamma_D f,\gamma_D g)_{L^2(\partial\Omega)}
\end{equation}
defined on $H^1(\Omega)\times H^1(\Omega)$.
Both operators $A_{\beta_0}$ and $A_{\beta_1}$ are self-adjoint in $L^2(\Omega)$ and semibounded from below. For $\beta\in\dR$ 
we shall also make use of the self-adjoint Robin realization 
\begin{equation}\label{annibb}
 A_\beta f=\cL f,\quad \dom (A_\beta) 
 = \bigl\{f\in H^2(\Omega) \, \big| \,\beta \gamma_D f=\gamma_\nu f\bigr\},
\end{equation}
which corresponds to the densely defined, closed, semibounded quadratic form
\begin{equation}\label{form0bb}
 \mathfrak a_\beta[f,g]=\sum_{j,k=1}^n \left(a_{jk} \frac{\partial f}{\partial x_k},\frac{\partial g}{\partial x_j}\right)_{L^2(\Omega)}+(af,g)_{L^2(\Omega)} - 
 (\beta\gamma_D f,\gamma_D g)_{L^2(\partial\Omega)}.
\end{equation}
on $H^1(\Omega)\times H^1(\Omega)$.

Next, we define the Neumann-to-Dirichlet map associated to $\cL$ as a densely defined operator in $L^2(\partial\Omega)$. First one notes that for $\beta_0=0$ in \eqref{anni}
(or $\beta=0$ in \eqref{annibb}) one obtains 
\begin{equation}\label{aneu}
A_N:=A_0=A_{\beta_0},
\end{equation}
where $A_N$ denotes the self-adjoint Neumann realization of $\cL$ in $L^2(\Omega)$. One 
recalls that for $\varphi\in H^{1/2}(\partial\Omega)$ and $z \in \rho(A_N)$, the boundary value problem 
\begin{equation}\label{bvpelli}
 \cL f_z = z f_z ,\quad \gamma_\nu f_z =\varphi,
\end{equation}
admits a unique solution $f_z \in H^2(\Omega)$; this follows, for instance, from \eqref{tracemapelli} and $z \in \rho(A_N)$. 
The corresponding solution operator is denoted by
\begin{equation}\label{ppelli}
 P_\nu(z):L^2(\partial\Omega) \rightarrow  L^2(\Omega),\quad \varphi\mapsto f_z ,
\end{equation}
and from the construction it is clear that 
\begin{equation*}
 \dom(P_\nu(z))=H^{1/2}(\partial\Omega)\, \text{ and } \, \ran(P_\nu(z))\subseteq H^2(\Omega).
\end{equation*}
For $z \in \rho(A_N)$ the {\it Neumann-to-Dirichlet map} associated to $\cL$ is 
defined as
\begin{equation}\label{ndmapll}
 \cN(z):L^2(\partial\Omega) \rightarrow L^2(\partial\Omega),\quad \varphi\mapsto 
 \gamma_D P_\nu(z)\varphi;
\end{equation}
it maps the (oblique) Neumann boundary values $\gamma_\nu f_z $ of solutions $f_z \in H^2(\Omega)$ of \eqref{bvpelli} onto the Dirichlet
boundary values $\gamma_D f_z $. It follows from the properties of the trace maps that
\begin{equation}
 \dom(\cN(z))=H^{1/2}(\partial\Omega)\, \text{ and } \, \ran(\cN(z))\subseteq H^{3/2}(\partial\Omega).
\end{equation}

In the next theorem a spectral shift function for the pair $\{A_{\beta_0},A_{\beta_1}\}$ is expressed in terms of the limits of the 
Neumann-to-Dirichlet map $\cN(z)$ and the functions $\beta_0$ and $\beta_1$ in the boundary conditions of the Robin realizations $A_{\beta_0}$ and $A_{\beta_1}$. 
We mention that the trace class condition for the difference of the $2k+1$-th powers of the resolvents was shown for $k=0$ in \cite{BLLLP10,G11}
and for $k \in \bbN$ in \cite{BLL13-3}.

\begin{theorem}\label{nrthmelli}
Assume Hypothesis~\ref{hypo5}, let $A_{\beta_0}$ and $A_{\beta_1}$ be the self-adjoint Robin realizations of $\cL$ in $L^2(\Omega)$ in
\eqref{anni} and \eqref{robi}, respectively, let $\beta\in\dR$ such that $\beta_p(x)<\beta$ 
for all $x\in\partial\Omega$ and $p=0,1$ and let $A_\beta$ be the self-adjoint Robin realizations of $\cL$ in \eqref{annibb}. 
Furthermore, let \begin{equation*}
 \cM_p(z)= (\beta-\beta_p)^{-1} \bigl(\beta_p\overline{\cN(z)}-I_{L^2(\partial\Omega)}\bigr)\bigl(\beta\overline{\cN(z)}-I_{L^2(\partial\Omega)}\bigr)^{-1},
 \quad z\in\dC\backslash\dR, \; j=1,2, 
\end{equation*}
where $\overline{\cN(z)}$ denotes the closure in $L^2(\partial\Omega)$ of the 
Neumann-to-Dirichlet map associated with $\cL$ in \eqref{ndmapll}.
Then the following assertions $(i)$ and $(ii)$ hold for $k \in \bbN_0$ such that $k\geq (n-3)/4$:
 \begin{itemize}
  \item [$(i)$] The difference of the $2k+1$th-powers of the resolvents of $A_{\beta_0}$ and $A_{\beta_1}$ is 
a trace class operator, that is,
\begin{equation*}
\big[(A_{\beta_1} - z I_{L^2(\Omega)})^{-(2k+1)} 
 - (A_{\beta_0} - z I_{L^2(\Omega)})^{-(2k+1)}\big] \in\sS_1\bigl(L^2(\Omega)\bigr) 
\end{equation*}
holds for all $z\in\rho(A_{\beta_0})\cap\rho(A_{\beta_1})$.
 \item [$(ii)$] For any orthonormal basis $(\varphi_j)_{j \in J}$ in 
 $L^2(\partial\Omega)$ the function 
  \begin{align*}
   \xi(\lambda) 
   =\sum_{j \in J} \lim_{\varepsilon\downarrow 0}\pi^{-1} \Bigl(\bigl(\Im\bigl( \log (\cM_1(\lambda+i\varepsilon)) - \log (\cM_0(\lambda + i \varepsilon))\bigr)\bigr)
   \varphi_j,\varphi_j\Bigr)_{L^2(\partial\Omega)}\\ 
   \text{for a.e. $\lambda\in\dR$,} 
  \end{align*}
is a spectral shift function for the pair $\{A_{\beta_0},A_{\beta_1}\}$ such that $\xi(\lambda)=0$ for 
$\lambda < \min(\sigma(A_\beta))$ and the trace formula
\begin{equation*}
 \tr_{L^2(\Omega)}\bigl( (A_{\beta_1} - z I_{L^2(\Omega)})^{-(2k+1)} 
 - (A_{\beta_0} - z I_{L^2(\Omega)})^{-(2k+1)}\bigr) 
 = - (2k+1) \int_\dR \frac{\,\xi(\lambda)\, d\lambda}{(\lambda - z)^{2k+2}}
\end{equation*}
is valid for all $z \in \rho(A_{\beta_0})\cap\rho(A_{\beta_1})$.
\end{itemize}
\end{theorem}
\begin{proof}
The proof of Theorem~\ref{nrthmelli} consists of three steps. In the first step we construct a suitable quasi boundary triple such that
the self-adjoint operators $A_\beta$
and $A_{\beta_1}$ correspond to the kernels of the boundary mappings $\Gamma_0$ and $\Gamma_1$, and in the second and third step we show that
the pair $\{A_{\beta},A_{\beta_1}\}$ and the $\gamma$-field and Weyl function satisfy the assumptions in Theorem~\ref{mainssf2}.
The same reasoning applies to the pair $\{A_{\beta},A_{\beta_0}\}$, and hence Theorem~\ref{mainssf2} can be applied to both pairs $\{A_{\beta},A_{\beta_1}\}$
and $\{A_{\beta},A_{\beta_0}\}$, which together with the considerations at the end of 
Section~\ref{ssfsec} yield the assertions in Theorem~\ref{nrthmelli}.

{\it Step 1.} The basic techniques in this step have been used in a similar framework, for instance, in \cite{BL07,BL12,BLL13-IEOT,BMN15}. 
We consider the closed symmetric operator $S=A_{\beta}\cap A_{\beta_1}$, which is given by
\begin{equation}\label{selli}
 Sf=\cL f,\quad\dom(S)=\bigl\{f\in H^2(\Omega) \, \big| \, \gamma_D f=\gamma_\nu f=0\bigr\}, 
\end{equation}
where we have used that $\beta-\beta_1(x)\not=0$ for all $x\in\partial\Omega$. In this step we check that the operator 
\begin{equation}\label{telli}
 Tf=\cL f,\quad\dom (T)=H^2(\Omega),
\end{equation}
satisfies $\overline T=S^*$ and that $\big\{L^2(\partial\Omega),\Gamma_0,\Gamma_1\big\}$, where
\begin{equation}\label{qbtelli1}
 \Gamma_0 f=\beta\gamma_D f - \gamma_\nu f,\quad f\in\dom (T),
\end{equation}
and
\begin{equation}\label{qbtelli2}
 \Gamma_1 f= (\beta-\beta_1)^{-1} \bigl(\beta_1 \gamma_D f - \gamma_\nu f\bigr) ,\quad f\in\dom (T),
\end{equation} 
is a quasi boundary triple for $T\subset S^*$ such that
\begin{equation}\label{aaelli}
 A_{\beta}=T\upharpoonright\ker(\Gamma_0)\, \text{ and } \, A_{\beta_1}=T\upharpoonright\ker(\Gamma_1),
\end{equation}
and for all $z \in \rho(A_{\beta})\cap\rho(A_N)$, where $A_N$ is the self-adjoint Neumann realization in \eqref{aneu}, 
the corresponding $\gamma$-field $\gamma$ and Weyl function $M$ in $L^2(\partial\Omega)$ 
are given by
\begin{equation}\label{welli}
\gamma(z)=P_\nu(z)\bigl(\beta\cN(z)-I_{L^2(\partial\Omega)}\bigr)^{-1},\quad\dom(\gamma(z))=H^{1/2}(\partial\Omega),
\end{equation}
and
\begin{equation}\label{mmchen}
\begin{split}
& M(z) = (\beta-\beta_1)^{-1} \bigl(\beta_1\cN(z)-I_{L^2(\partial\Omega)}\bigr)\bigl(\beta\cN(z)-I_{L^2(\partial\Omega)}\bigr)^{-1},\\
& \dom(M(z)) = H^{1/2}(\partial\Omega).
 \end{split}
\end{equation}

We will use Theorem~\ref{ratemal} for this purpose. For $f,g\in\dom(T)=H^2(\Omega)$ one 
obtains with the help of Green's identity \eqref{g2}, 
\begin{equation*}
\begin{split}
&(\Gamma_1 f,\Gamma_0 g)_{L^2(\partial\Omega)}-(\Gamma_0 f,\Gamma_1 g)_{L^2(\partial\Omega)}\\
& \quad =\bigl((\beta-\beta_1)^{-1}(\beta_1 \gamma_D f - \gamma_\nu f),\beta \gamma_D g - \gamma_\nu g\bigr)_{L^2(\partial\Omega)}\\
&\qquad - \bigl(\beta \gamma_D f - \gamma_\nu f,(\beta-\beta_1)^{-1}(\beta_1 \gamma_D g - \gamma_\nu g)\bigr)_{L^2(\partial\Omega)}\\
&\quad=-\bigl((\beta-\beta_1)^{-1} \beta_1 \gamma_D f , \gamma_\nu g \bigr)_{L^2(\partial\Omega)}
        - \bigl(\gamma_\nu f ,(\beta-\beta_1)^{-1}\beta\gamma_D g \bigr)_{L^2(\partial\Omega)} \\
&\qquad + \bigl((\beta-\beta_1)^{-1}\beta \gamma_D f , \gamma_\nu g\bigr)_{L^2(\partial\Omega)}
              +\bigl(\gamma_\nu f , (\beta-\beta_1)^{-1}\beta_1\gamma_D g\bigr)_{L^2(\partial\Omega)}\\
&\quad=\bigl(\gamma_D f,\gamma_\nu g\bigr)_{L^2(\partial\Omega)}-\bigl(\gamma_\nu f,\gamma_D g\bigr)_{L^2(\partial\Omega)}\\
&\quad=(\cL f,g)_{L^2(\Omega)}-(f,\cL g)_{L^2(\Omega)}\\
&\quad=(Tf,g)_{L^2(\Omega)}-(f,Tg)_{L^2(\Omega)}, 
\end{split}
\end{equation*}
and hence condition $(i)$ in Theorem~\ref{ratemal} holds. Since
\begin{equation}\label{dxxx}
\begin{pmatrix}\Gamma_0 f\\ \Gamma_1 f\end{pmatrix}=\begin{pmatrix} \beta & -I_{L^2(\partial\Omega)} \\ 
\beta_1(\beta-\beta_1)^{-1} & - (\beta-\beta_1)^{-1} \end{pmatrix} 
\begin{pmatrix}\gamma_D f\\ \gamma_\nu f\end{pmatrix},\quad f\in\dom(T),
\end{equation}
and the $2\times 2$ operator matrix in \eqref{dxxx} is an isomorphism in $L^2(\partial\Omega)\times L^2(\partial\Omega)$, it follows from
\eqref{tracemapelli} that 
$\ran(\Gamma_0,\Gamma_1)^\top$ is dense in $L^2(\partial\Omega)\times L^2(\partial\Omega)$.
It is easy to see that $\ker(\Gamma_0)\cap\ker(\Gamma_1)$ is dense in $L^2(\Omega)$.
Moreover, \eqref{aaelli} is clear from the definition
of $T$ and the boundary maps in \eqref{qbtelli1}--\eqref{qbtelli2}. Hence also conditions $(ii)$ 
and $(iii)$ in Theorem~\ref{ratemal} are satisfied, and from \eqref{selli}, \eqref{telli},
and \eqref{qbtelli1}--\eqref{qbtelli2} one obtains 
\begin{equation*}
 S=T\upharpoonright\bigl(\ker(\Gamma_0)\cap\ker(\Gamma_1)\bigr).
\end{equation*}
Thus Theorem~\ref{ratemal} yields $\overline T=S^*$ and that $\{L^2(\partial\Omega),\Gamma_0,\Gamma_1\}$ is a quasi boundary triple for $S^*$ such
that \eqref{aaelli} holds.

It remains to show the explicit form of the corresponding $\gamma$-field and Weyl function $M$ in \eqref{welli} and \eqref{mmchen}, respectively. 
First of all it follows from \eqref{tracemapelli} and the definition of $\Gamma_0$ in \eqref{qbtelli1} that
\begin{equation*}
 H^{1/2}(\partial\Omega)=\ran(\Gamma_0)=\dom(\gamma(z))=\dom(M(z)),\quad z \in \rho(A_{\beta}).
\end{equation*}
One notes that for $z \in \rho(A_N)$ and  
$f_z \in\ker(T - z I_{L^2(\Omega)})$ one has $\cN(z)\gamma_\nu f_z =\gamma_D f_z $ according to \eqref{ndmapll}, and hence
\begin{equation}\label{plk}
 \bigl(\beta\cN(z)-I_{L^2(\partial\Omega)}\bigr)\gamma_\nu f_z =
 \beta\gamma_D f_z -\gamma_\nu f_z =\Gamma_0 f_z, 
\end{equation}
and 
\begin{equation}\label{plkk}
 (\beta-\beta_1)^{-1}\bigl(\beta_1\cN(z)-I_{L^2(\partial\Omega)}\bigr)\gamma_\nu f_z =
 (\beta-\beta_1)^{-1}\bigl(\beta_1\gamma_D f_z -\gamma_\nu f_z \bigr)=\Gamma_1 f_z, 
\end{equation}
by \eqref{qbtelli1}--\eqref{qbtelli2}.
The relation \eqref{plk} also shows that $\ker(\beta\cN(z)-I_{L^2(\partial\Omega)})=\{0\}$ for $z \in \rho(A_{\beta})\cap\rho(A_N)$ and hence
\begin{equation}\label{plk3}
 \gamma_\nu f_z =\bigl(\beta\cN(z)-I_{L^2(\partial\Omega)}\bigr)^{-1}\Gamma_0 f_z .
\end{equation}
From this and \eqref{ppelli} it follows that the $\gamma$-field corresponding to $\{L^2(\partial\Omega),\Gamma_0,\Gamma_1\}$  has the form \eqref{welli}.
One also concludes from \eqref{plk3} and \eqref{plkk} that
\begin{equation*}
\begin{split}
 &(\beta-\beta_1)^{-1}\bigl(\beta_1\cN(z)-I_{L^2(\partial\Omega)}\bigr)\bigl(\beta\cN(z)-I_{L^2(\partial\Omega)}\bigr)^{-1}\Gamma_0 f_z \\
 &\quad =(\beta-\beta_1)^{-1}\bigl(\beta_1\cN(z)-I_{L^2(\partial\Omega)}\bigr)\gamma_\nu f_z \\
 &\quad =\Gamma_1 f_z 
\end{split}
\end{equation*}
holds for all $f_z \in\ker(T - z I_{L^2(\Omega)})$ and $z \in \rho(A_{\beta})\cap\rho(A_N)$. Thus the Weyl function corresponding
to the quasi boundary triple $\{L^2(\partial\Omega),\Gamma_0,\Gamma_1\}$  has the form \eqref{mmchen}. 

\vskip 0.2cm\noindent
{\it Step 2.} In this step we verify 
that the pair $\{A_{\beta},A_{\beta_1}\}$ satisfies the sign condition \eqref{sign333} and that the values of Weyl function and its inverse are bounded operators; see the assumptions of Theorem~\ref{mainssf2}. 

The assumption $\beta >\beta_1(x)$ shows that the semibounded quadratic forms $\mathfrak a_{\beta}$ and $\mathfrak a_{\beta_1}$ in \eqref{form0} and \eqref{form0bb} corresponding to 
$A_{\beta}$ and $A_{\beta_1}$ satisfy the inequality $\mathfrak a_{\beta}\leq \mathfrak a_{\beta_1}$. 
Hence $\min(\sigma(A_{\beta}))\leq \min(\sigma(A_{\beta_1}))$ and for
$\zeta < \min (\sigma(A_{\beta}))$ the forms $\mathfrak a_{\beta}-\zeta$ and  $\mathfrak a_{\beta_1}-\zeta$ are both nonnegative, satisfy 
the inequality $\mathfrak a_{\beta}-\zeta \leq \mathfrak a_{\beta_1}-\zeta$, and hence the resolvents of the
corresponding nonnegative self-adjoint operators $A_{\beta}-\zeta I_{L^2(\Omega)}$ and 
$A_{\beta_1}-\zeta I_{L^2(\Omega)}$ satisfy the inequality
\begin{equation*}
 (A_{\beta}-\zeta I_{L^2(\Omega)})^{-1}\geq (A_{\beta_1}-\zeta I_{L^2(\Omega)})^{-1}, 
 \quad \zeta < \min (\sigma(A_{\beta})) 
\end{equation*}
(see, e.g., \cite[Chapter VI, $\S$ 2.6]{K76} or \cite[Chapter 10, $\S$2-Theorem 6]{BS87}). Thus the  sign condition \eqref{sign333} in the assumptions of Theorem~\ref{mainssf2}
holds.

Next we prove that
\begin{equation}\label{ensure21}
M(z_1)=(\beta-\beta_1)^{-1}\bigl(\beta_1\cN(z_1)-I_{L^2(\partial\Omega)}\bigr)\bigl(\beta\cN(z_1)-I_{L^2(\partial\Omega)}\bigr)^{-1}
\end{equation}
and
\begin{equation}\label{ensure22}
M(z_2)^{-1}=\bigl(\beta\cN(z_2)-I_{L^2(\partial\Omega)}\bigr)\bigl(\beta_1\cN(z_2)-I_{L^2(\partial\Omega)}\bigr)^{-1}(\beta-\beta_1)
\end{equation}
are bounded operators for some $z_1,z_2\in\dC\backslash\dR$. According to \cite[Lemma 4.4]{BLL13-IEOT} the closure $\overline{\cN(z)}$, $z\in\dC\backslash\dR$,
of the Neumann-to-Dirichlet map in \eqref{ndmapll} in $L^2(\partial\Omega)$ is a compact operator, and hence it is clear that $\beta\cN(z)-I_{L^2(\partial\Omega)}$
and $\beta_1\cN(z)-I_{L^2(\partial\Omega)}$ are 
densely defined bounded operators in $L^2(\partial\Omega)$, and that for 
$z\in\dC\backslash\dR$ their closures are
\begin{equation}\label{kjj}
\big[\beta\overline{\cN(z)}-I_{L^2(\partial\Omega)}\big] \in \cL\bigl(L^2(\partial\Omega)\bigr) 
\, \text{ and } \, \big[\beta_1\overline{\cN(z)}-I_{L^2(\partial\Omega)}\big] 
\in \cL\bigl(L^2(\partial\Omega)\bigr).
\end{equation}
In order to see that 
\begin{equation*}
Q(z):=\bigl(\beta\cN(z)-I_{L^2(\partial\Omega)}\bigr)^{-1}\, \text{ and } \,  Q_1(z):=\bigl(\beta_1\cN(z)-I_{L^2(\partial\Omega)}\bigr)^{-1}
\end{equation*}
are bounded for $z\in\dC\backslash\dR$ 
we argue in a similar way as in the proof of \cite[Lemma~4.4]{BLL13-IEOT}: First, one notes that 
$\cN(z)\subseteq \cN({\ol z})^*$, $z\in\dC\backslash\dR$, 
holds by \eqref{g2}, and this yields that also
$Q(z)\subseteq Q({\ol z})^*$, $z\in\dC\backslash\dR$.
Hence the operator $Q(z)$ is closable in $L^2(\partial\Omega)$. Moreover, as $Q(z)$ is defined on $H^{1/2}(\partial\Omega)$ and maps
into $H^{1/2}(\partial\Omega)$, it follows that $Q(z)$ is a closed operator in $H^{1/2}(\partial\Omega)$, and hence 
\begin{equation}\label{qelli}
 Q(z):H^{1/2}(\partial\Omega)\rightarrow H^{1/2}(\partial\Omega),\quad  \varphi\mapsto Q(z)\varphi,
\end{equation}
is bounded.
Therefore, the dual operator
\begin{equation}\label{qselli}
 Q(z)^\prime:H^{-1/2}(\partial\Omega)\rightarrow H^{-1/2}(\partial\Omega),\quad  \psi\mapsto Q(z)^\prime\psi,
\end{equation}
where 
$(Q(z)^\prime\psi)(\varphi)=\psi(Q(z)\varphi)$, $\varphi\in H^{1/2}(\partial\Omega)$, is also bounded. 
One verifies that 
$Q({\ol z})^\prime$ is an extension of $Q(z)$ and hence by interpolation and 
\eqref{qelli}--\eqref{qselli}, the restriction 
\begin{equation*}
 Q({\ol z})^\prime \upharpoonright_{L^2(\partial\Omega)}:L^2(\partial\Omega)\rightarrow L^2(\partial\Omega),\quad  \phi\mapsto Q({\ol z})^\prime\phi, 
\end{equation*}
of $Q({\ol z})^\prime$ onto $L^2(\partial\Omega)$
is a bounded operator in $L^2(\partial\Omega)$ and an extension of $Q(z)$. Hence for all $z\in\dC\backslash\dR$ the operator 
$Q(z)$ is bounded in $L^2(\partial\Omega)$ and its closure is 
\begin{equation}\label{kjj2} 
\overline{Q(z)}= \bigl(\beta\overline{\cN(z)}-I_{L^2(\partial\Omega)}\bigr)^{-1}\in\cL\bigl(L^2(\partial\Omega)\bigr),\quad z\in\dC\backslash\dR.
\end{equation}
The same reasoning with $Q(z)$ replaced by $Q_1(z)$ shows that for all $z\in\dC\backslash\dR$ the operator 
$Q_1(z)$ is bounded in $L^2(\partial\Omega)$ and  
\begin{equation}\label{kjj22} 
\overline{Q_1(z)}= \bigl(\beta_1\overline{\cN(z)}-I_{L^2(\partial\Omega)}\bigr)^{-1}\in\cL\bigl(L^2(\partial\Omega)\bigr),\quad z\in\dC\backslash\dR.
\end{equation}
Next, it follows that $M(z_1)$ and $M(z_2)^{-1}$ 
in \eqref{ensure21}--\eqref{ensure22} are bounded in $L^2(\partial\Omega)$ for $z_1,z_2\in\dC\backslash\dR$ and the closure of $M(z)$ is given by 
\begin{equation}
 \overline{M(z)}= (\beta-\beta_1)^{-1}\bigl(\beta_1\overline{\cN(z)}-I_{L^2(\partial\Omega)}\bigr)\bigl(\beta\overline{\cN(z)}-I_{L^2(\partial\Omega)}\bigr)^{-1}
\end{equation}
by \eqref{kjj} and \eqref{kjj2}. One notes that $\overline{M(z)}=\cM_1(z)$ in the formulation of Theorem~\ref{nrthmelli}.

\vskip 0.2cm\noindent
{\it Step 3.} In this step we verify that the $\gamma$-field and Weyl function corresponding to the 
quasi boundary triple $\big\{L^2(\partial\Omega),\Gamma_0,\Gamma_1\big\}$ in Step 1 satisfy 
the $\sS_p$-conditions in the 
assumptions of Theorem~\ref{mainssf2} for dimensions $n \in \bbN$, $n \geq 2$, and $k\geq (n-3)/4$, 
that is, we verify 
for all $p,q \in \bbN_0$ and all $z \in \rho(A_{\beta})\cap\rho(A_{\beta_1})$ the conditions
\begin{equation}\label{ass1q}
 \overline{\gamma(z)}^{(p)}\bigl( M(z)^{-1} \gamma({\ol z})^*\bigr)^{(q)}\in\sS_1\bigl(L^2(\Omega)\bigr),\quad p+q=2k,
\end{equation}
\begin{equation}\label{ass2q}
 \bigl( M(z)^{-1} \gamma({\ol z})^*\bigr)^{(q)}\overline{\gamma(z)}^{(p)}\in\sS_1\bigl(L^2(\partial\Omega)\bigr),\quad p+q=2k,
\end{equation}
and 
 \begin{equation}\label{ass3q}
   \frac{d^j}{dz^j} \overline{M (z)}\in\sS_{(2k+1)/j}\bigl(L^2(\partial\Omega)\bigr),\quad j=1,\dots,2k+1.
 \end{equation}

In the following we shall often use the smoothing property
\begin{equation}\label{smoothi}
 (A_{\beta} - z I_{L^2(\Omega)})^{-1} f \in H^{k+2}(\Omega)\, \text{ for all } \, f\in H^k(\Omega), \; 
 k \in \bbN_0,
\end{equation}
of the resolvent of $A_{\beta}$, which follows, for instance, from \cite[Theorem 4.18]{McL00}.
One notes that \eqref{gstar} and the definition of the boundary map $\Gamma_1$ in \eqref{qbtelli2} yield
\begin{equation}\label{yes}
\begin{split}
 \gamma({\ol z})^*f&=\Gamma_1(A_{\beta} - z I_{L^2(\Omega)})^{-1}f\\
 &=(\beta-\beta_1)^{-1}(\beta_1 \gamma_D  - \gamma_\nu )
 (A_{\beta} - z I_{L^2(\Omega)})^{-1}f\\
 &=(\beta-\beta_1)^{-1}\bigl(\beta\gamma_D - \gamma_\nu + (\beta_1-\beta)\gamma_D\bigr)
 (A_{\beta} - z I_{L^2(\Omega)})^{-1}f\\
 &=(\beta-\beta_1)^{-1}(\beta\gamma_D - \gamma_\nu)(A_{\beta} - z I_{L^2(\Omega)})^{-1}f - \gamma_D(A_{\beta} - z I_{L^2(\Omega)})^{-1}f\\
 &=- \gamma_D(A_{\beta} - z I_{L^2(\Omega)})^{-1}f
\end{split}
\end{equation}
for all $z \in \rho(A_\beta)$ and $f\in L^2(\Omega)$. Here we have used in the last step that 
$$g=(A_{\beta} - z I_{L^2(\Omega)})^{-1}f\in \dom(A_{\beta})$$ 
satisfies the boundary 
condition $\beta\gamma_Dg - \gamma_\nu g=0$.
It follows from \eqref{gammad2} and \eqref{yes} that
\begin{equation*}
 \bigl(\gamma({\ol z})^*\bigr)^{(q)}=q! \, \gamma({\ol z})^*(A_{\beta} - z I_{L^2(\Omega)})^{-q} 
 = - q! \, \gamma_D(A_{\beta} - z I_{L^2(\Omega)})^{-(q+1)}, 
\end{equation*}
and hence,  
\begin{equation*}
 \ran \bigl((\gamma({\ol z})^*)^{(q)}\bigr)\subset H^{2q+3/2}(\partial\Omega)
 \end{equation*}
by \eqref{smoothi} and \eqref{gamss}. From Lemma~\ref{usel} with $s=2q+(3/2)$ and $t=0$ one concludes that  
\begin{equation}\label{ass2w}
 \bigl(\gamma({\ol z})^*\bigr)^{(q)} \in\sS_r\bigl(L^2(\Omega),L^2(\partial\Omega)\bigr), 
 \quad r>(n-1)/[2q+(3/2)],
\end{equation}
for all $z \in \rho(A_{\beta})$, $q \in \bbN_0$, and hence by \eqref{gammad1} also 
\begin{equation}\label{ass1w}
 \overline{\gamma(z)}^{(p)}\in\sS_r\bigl(L^2(\partial\Omega),L^2(\Omega)\bigr), 
 \quad r> (n-1)/[2p+(3/2)],
\end{equation}
for all $z \in \rho(A_{\beta})$, $p \in \bbN_0$. Furthermore, 
\begin{equation}\label{mqqq}
 \frac{d^j}{dz^j} \overline{M (z)}=j! \, \gamma({\ol z})^* 
 (A_{\beta} - z I_{L^2(\Omega)})^{-(j-1)}\overline{\gamma(z)}, \quad j \in \bbN, 
\end{equation}
by \eqref{gammad3} and with the help of \eqref{yes} it follows in the same way as in \eqref{ass2w} that 
\begin{equation*}
\gamma({\ol z})^*(A_{\beta} - z I_{L^2(\Omega)})^{-(j-1)} 
= - \gamma_D (A_{\beta} - z I_{L^2(\Omega)})^{-j} \in 
\sS_x\bigl(L^2(\Omega),L^2(\partial\Omega)\bigr)
 \end{equation*}
for $x> (n-1)/[2j-(1/2)]$. Moreover, $\overline{\gamma(z)}\in\sS_y(L^2(\partial\Omega),L^2(\Omega))$ for $y> 2 (n-1)/3$ by \eqref{ass1w} and hence 
it follows from \eqref{mqqq} and the well-known property $PQ\in\sS_w$ for $P\in\sS_x$, $Q\in\sS_y$, and $x^{-1}+y^{-1}=w^{-1}$, that 
 \begin{equation}\label{ass3w}
 \frac{d^j}{dz^j} \overline{M(z)}\in\sS_w\bigl(L^2(\partial\Omega\bigr), 
 \quad w >  (n-1)/(2j+1), \; z \in \rho(A_{\beta}), \; j \in \bbN.
\end{equation}
One observes that
\begin{equation*}
 \frac{d}{dz} \overline{M(z)}^{-1}= - \overline{M(z)}^{-1}\left(\frac{d}{dz} \overline{M(z)}\right)\overline{M(z)}^{-1}, \quad z \in \rho(A_{\beta})\cap\rho(A_{\beta_1}),
\end{equation*}
that $\overline{M(z)}^{-1}$ is bounded, and by \eqref{ass3w} that also 
\begin{equation}\label{ass4w}
 \frac{d^j}{dz^j} \overline{M (z)}^{-1}\in\sS_w\bigl(L^2(\partial\Omega)\bigr), 
 \quad w > (n-1)/(2j+1), \; z \in \rho(A_{\beta})\cap\rho(A_{\beta_1}), \; j \in \bbN;
\end{equation}
we leave the formal induction step to the reader. Therefore, 
\begin{equation*}
\begin{split}
\bigl(  M(z)^{-1} \gamma({\ol z})^* \bigr)^{(q)}&=\bigl( \overline{M(z)}^{-1} \gamma({\ol z})^*\bigr)^{(q)} \\
&=\sum_{\substack{p+m=q \\[0.2ex] p,m\ge0}} \begin{pmatrix} q \\ p \end{pmatrix} 
 \bigl(\overline{M(z)}^{-1}\bigr)^{(p)} \bigl(\gamma({\ol z})^*\bigr)^{(m)}\\
 &=\overline{M(z)}^{-1} \bigl(\gamma({\ol z})^*\bigr)^{(q)}+ \sum_{\substack{p+m=q \\[0.2ex] p > 0, m\geq 0}} \begin{pmatrix} q \\ p \end{pmatrix} 
 \bigl(\overline{M(z)}^{-1}\bigr)^{(p)} \bigl(\gamma({\ol z})^*\bigr)^{(m)}, 
 \end{split}
 \end{equation*}
and one has $\overline{M(z)}^{-1} (\gamma({\ol z})^*)^{(q)}\in\sS_r\big(L^2(\Omega),L^2(\partial\Omega)\big)$ for $r> (n-1)/[2q+(3/2)]$ by \eqref{ass2w} and each
summand (and hence also the finite sum) on the right-hand side is in $\sS_r\big(L^2(\Omega),L^2(\partial\Omega)\big)$ for 
$r> (n-1)/[2p+1+2m+(3/2)] = (n-1)/[2q+(5/2)]$, which follows from
\eqref{ass4w} and \eqref{ass1w}. Hence one has 
\begin{align}\label{hohoja} 
\begin{split} 
& \bigl( M(z)^{-1} \gamma({\ol z})^*\bigr)^{(q)}\in \sS_r\bigl(L^2(\Omega),L^2(\partial\Omega)\bigr),  \\   
& \, r > (n-1)/[2q+(3/2)], \; z \in \rho(A_{\beta})\cap\rho(A_{\beta_1}).
\end{split} 
\end{align}
From \eqref{ass1w} and \eqref{hohoja} one then concludes that 
\begin{align*}
& \overline{\gamma(z)}^{(p)}\bigl( M(z)^{-1} \gamma({\ol z})^*\bigr)^{(q)}\in\sS_r\bigl(L^2(\Omega)\bigr),  \\ 
& r > (n-1)/[2(p+q)+3] = (n-1)/(4k+3),
\end{align*}
and since $k\geq (n-3)/4$, one has $1> (n-1)/(4k+3)$, that is, the trace class condition~\eqref{ass1q} is satisfied. The same argument shows that \eqref{ass2q} is satisfied.
Finally, \eqref{ass3q} follows from \eqref{ass3w} and the fact that $k\geq (n-3)/4$ implies
\begin{equation*}
 \frac{2k+1}{j}\geq\frac{n-1}{2j}>\frac{n-1}{2j+1},\quad j=1,\dots,2k+1.
\end{equation*} 

Hence the assumptions in Theorem~\ref{mainssf2} are satisfied with $S$ in \eqref{selli}, the quasi boundary triple 
in \eqref{qbtelli1}--\eqref{qbtelli2} and the corresponding $\gamma$-field and Weyl function in \eqref{welli} and \eqref{mmchen}, respectively. Now Theorem~\ref{mainssf2} yields
assertion $(i)$ in Theorem~\ref{nrthmelli} with $A_{\beta_0}$ replaced by $A_\beta$ and
for any orthonormal basis $(\varphi_j)_{j \in J}$ in $L^2(\partial\Omega)$ the function 
  \begin{equation*}
   \xi_1(\lambda) 
   =\sum_{j \in J} \lim_{\varepsilon\downarrow 0}\pi^{-1} \bigl(\Im\bigl( \log (\cM_1(\lambda+i\varepsilon)) \bigr)
   \varphi_j,\varphi_j\bigr)_{L^2(\partial\Omega)}  \, \text{ for a.e.~$\lambda \in \dR$},
  \end{equation*}
is a spectral shift function for the pair $\{A_\beta,A_{\beta_1}\}$ such that $\xi_1(\lambda)=0$ for 
$\lambda < \min(\sigma(A_\beta))\leq \min(\sigma(A_{\beta_1}))$ and the trace formula
\begin{equation*}
 \tr_{L^2(\Omega)}\bigl( (A_{\beta_1} - z I_{L^2(\Omega)})^{-(2k+1)} 
 - (A_\beta - z I_{L^2(\Omega)})^{-(2k+1)}\bigr) 
 = - (2k+1) \int_\dR \frac{\,\xi_1(\lambda)\, d\lambda}{(\lambda - z)^{2k+2}}
\end{equation*}
is valid for all $z \in \rho(A_\beta)\cap\rho(A_{\beta_1})$.

The same construction as above with $\beta_1$ replaced by $\beta_0$ yields an analogous representation for a spectral shift function $\xi_0$ 
of the pair $\{A_\beta,A_{\beta_0}\}$.
Finally it follows from the considerations in the end of Section~\ref{ssfsec} (see \eqref{ssfab}) 
that 
\begin{equation*}
\begin{split}
   \xi(\lambda)&=\xi_1(\lambda)-\xi_0(\lambda)\\
   &=\sum_{j \in J} \lim_{\varepsilon\downarrow 0}\pi^{-1} \Bigl(\bigl(\Im\bigl( \log (\cM_1(\lambda+i\varepsilon)) - \log (\cM_0(\lambda + i \varepsilon))\bigr)\bigr)
   \varphi_j,\varphi_j\Bigr)_{L^2(\partial\Omega)} 
\end{split}
   \end{equation*}
   for a.e.~$\lambda \in \dR$
is a spectral shift function for the pair $\{A_{\beta_0},A_{\beta_1}\}$ such that $\xi(\lambda)=0$ for 
$\lambda < \min(\sigma(A_{\beta}))\leq \min(\sigma(A_{\beta_p}))$, $p=0,1$. This completes the proof of Theorem~\ref{nrthmelli}.
\end{proof}

In space dimensions $n=2$ and $n=3$ one can choose $k=0$ in Theorem~\ref{nrthmelli}, and hence the resolvent difference of $A_{\beta_1}$ and 
$A_{\beta_0}$ is a trace class operator. In this situation Corollary~\ref{mainthmcorchen} leads to the following slightly stronger statement.

\begin{corollary}\label{mainthmcorchen2}
Let the assumptions be as in Theorem~\ref{nrthmelli} and suppose that $n=2$ or $n=3$. 
Then the following assertions $(i)$--$(iii)$ hold:

\begin{itemize}
  \item [$(i)$]
The difference of the resolvents of $A_{\beta_1}$ and $A_{\beta_0}$ is 
a trace class operator, that is,
\begin{equation*}
 \big[(A_{\beta_1} - z I_{L^2(\Omega)})^{-1}-(A_{\beta_0} - z I_{L^2(\Omega)})^{-1}\big] \in \sS_1\bigl(L^2(\Omega)\bigr)
\end{equation*}
holds for all $z\in\rho(A_{\beta_1})\cap\rho(A_{\beta_0})$. 
  \item [$(ii)$] $\Im(\log (\cM_p(z)))\in\sS_1(L^2(\partial\Omega))$ for all $z\in\dC\backslash\dR$ and $p=0,1$, and the limit 
  $$\Im\bigl(\log(\cM_p(\lambda+i 0))\bigr):=\lim_{\varepsilon\downarrow 0}\Im\bigl(\log(\cM_p(\lambda+i\varepsilon))\bigr)$$ 
  exists for a.e.~$\lambda \in \dR$ and $p=0,1$ in $\sS_1(L^2(\partial\Omega))$. 
  \item [$(iii)$] The function
  \begin{equation*}
   \xi(\lambda)=\pi^{-1} \tr_{L^2(\partial\Omega)}\bigl(\Im\big(\log(\cM_1(\lambda + i0))-\log(\cM_0(\lambda + i0)) \big)\bigr) \, \text{ for a.e.~$\lambda \in \dR$}, 
  \end{equation*}
is a spectral shift function for the pair $\{A_{\beta_0},A_{\beta_1}\}$ such that $\xi(\lambda)=0$ for 
$\lambda < \min(\sigma(A_\beta))$ and the trace formula
\begin{equation*}
 \tr_{L^2(\Omega)}\bigl( (A_{\beta_1} - z I_{L^2(\Omega)})^{-1} 
 - (A_{\beta_0} - z I_{L^2(\Omega)})^{-1}\bigr) 
 = - (2k+1) \int_\dR \frac{\,\xi(\lambda)\, d\lambda}{(\lambda - z)^2}
\end{equation*}
is valid for all $z \in \rho(A_{\beta_0})\cap\rho(A_{\beta_1})$.
\end{itemize}
\end{corollary}

\section{Schr\"{o}dinger operators with compactly supported potentials}\label{ap11sec}

In this section we determine a spectral shift function for the self-adjoint operators $\{-\Delta,-\Delta+V\}$ in $L^2(\dR^n)$, $n \in \bbN$, $n\geq 2$, where 
it is assumed that $V\in L^\infty(\dR^n)$ is a compactly supported real-valued function. Thus we consider the self-adjoint operators  
\begin{equation}\label{abva}
 A=-\Delta,\quad \dom(A)=H^2(\dR^n),
\end{equation}
and
\begin{equation}\label{abvb}
 B=-\Delta+V,\quad \dom(B)=H^2(\dR^n),
\end{equation}
in $L^2(\bbR^n)$,
and we fix an open ball $\cB_+\subset\dR^n$ such that $\text{\rm supp}\, (V) \subset\cB_+$. The $n-1$ dimensional
sphere $\partial\cB_+$ is denoted by $\cS$. Besides the operators $A$ and $B$ 
we shall also make use of the self-adjoint Dirichlet realizations 
\begin{equation}\label{apbva}
A_+=-\Delta, \quad \dom(A_+)=H^2(\cB_+)\cap H^1_0(\cB_+),
\end{equation}
and
\begin{equation}\label{apbvb}
B_+=-\Delta+V,\quad \dom(B_+)=H^2(\cB_+)\cap H^1_0(\cB_+),
\end{equation}
of $-\Delta$ and $-\Delta+V$ in $L^2(\cB_+)$.
The spectrum of both operators $A_+$ and $B_+$ is discrete
and bounded from below. The corresponding eigenvalue counting functions are denoted by 
$N(\, \cdot \,,A_+)$ and $N(\, \cdot \,,B_+)$, respectively; one 
recalls that $N(\lambda,A_+)$ and $N(\lambda,B_+)$ stand for the total number of eigenvalues (with multiplicities counted) of $A_+$ and $B_+$ in the open interval 
$(-\infty, \lambda)$, $\lambda \in \bbR$.

The main ingredient in the proof of Theorem~~\ref{thmvv} below is a decoupling technique for the operators $A$ and $B$, where artificial 
Dirichlet boundary conditions on the sphere $\cS$ will be imposed. We shall use the extension of the $L^2(\cS)$ scalar product onto the dual pair
$H^{1/2}(\cS)\times H^{-1/2}(\cS)$ via
\begin{equation}\label{ext}
 \langle\varphi,\psi\rangle=\bigl(\imath\varphi,\widetilde{\imath^{-1}}\psi\bigr)_{L^2(\cS)},\quad \varphi\in H^{1/2}(\cS),\,\,\psi\in H^{-1/2}(\cS),
\end{equation}
where $\imath$ is a uniformly positive self-adjoint operator in $L^2(\cS)$ defined on the dense subspace $H^{1/2}(\cS)$ 
(and in the following $\iota$ is regarded as an isomorphism from $H^{1/2}(\cS)$ onto $L^2(\cS)$),
and 
$\widetilde{\imath^{-1}}$ is the extension of $\imath^{-1}$ to an isomorphism from $H^{-1/2}(\cS)$ onto $L^2(\cS)$. A typical and convenient choice for
$\imath$ is $(-\Delta_\cS+I_{L^2(\cS)})^{1/4}$, where $-\Delta_\cS$ is the Laplace-Beltrami operator on the sphere $\cS$; for other choices see also
\cite[Remark 5.3]{BMN15}.

Since $\langle\cdot,\cdot\rangle$
in \eqref{ext} is an extension of the $L^2(\cS)$ scalar product, Green's identity can also be written in the form
\begin{equation}\label{green11}
 (-\Delta f_+,g_+)_{L^2(\cB_+)}-(f_+,-\Delta g_+)_{L^2(\cB_+)}=\langle\gamma_D^+ f_+,\gamma_N^+ g_+ \rangle-\langle\gamma_N^+ f_+,\gamma_D^+ g_+\rangle
\end{equation}
for $f_+,g_+\in H^2(\cB_+)$. Here $\gamma_D^+$ and $\gamma_N^+$ denote the Dirichlet and Neumann trace operators in \eqref{tracemapelli} (with $\Omega$ and $\partial\Omega$ replaced by
$\cB_+$ and $\cS$, respectively).
Let $\cB_-:=\dR^n\backslash\overline\cB_+$ and let  $\gamma_D^-$ and $\gamma_N^-$ be the Dirichlet and Neumann trace operators on $\cB_-$;
the normal vector in the definition of $\gamma_N^-$ is pointing in the outward direction of $\cB_-$ and hence opposite to the normal of $\cB_+$. 
Besides \eqref{green11} we also have the corresponding Green's identity on $\cB_-$, that is, 
\begin{equation}\label{green12}
 (-\Delta f_-,g_-)_{L^2(\cB_-)}-(f_-,-\Delta g_-)_{L^2(\cB_-)}=\langle\gamma_D^- f_-,\gamma_N^- g_- \rangle-\langle\gamma_N^- f_-,\gamma_D^- g_-\rangle
\end{equation}
holds for all $f_-,g_-\in H^2(\cB_-)$.

Next we define Dirichlet-to-Neumann maps associated to 
$-\Delta$ and $-\Delta+V$ on $\cB_+$ and $-\Delta$ on $\cB_-$ as operators from $H^{1/2}(\cS)$ to $H^{-1/2}(\cS)$. First, we recall that for $z \not \in \sigma(A_+)$ and
$\varphi\in H^{1/2}(\cS)$ there exists a unique solution $f_z \in H^1(\cB_+)$ of the boundary value problem
\begin{equation*}
 -\Delta f_z = z f_z  ,\quad \gamma_D^+ f_z =\varphi.
\end{equation*}
The corresponding solution operator is 
\begin{equation*} 
P_+(z):H^{1/2}(\cS)\rightarrow H^1(\cB_+), \quad \varphi\mapsto f_z,
\end{equation*} 
and for $z \not \in \sigma(A_+)$, the {\it Dirichlet-to-Neumann map} $\cD_+(z)$ 
{\it associated to $-\Delta$ 
in $\cB_+$} is defined by
\begin{equation*}
 \cD_+(z):H^{1/2}(\cS)\rightarrow H^{-1/2}(\cS),\quad \varphi\mapsto \gamma_N^+ P_+(z)\varphi.
\end{equation*}
Similarly, for $\zeta \not \in \sigma(B_+)$ and
$\psi\in H^{1/2}(\cS)$, there exists a unique solution $g_\zeta \in H^1(\cB_+)$ of the boundary value problem
\begin{equation*}
 (-\Delta +V)g_\zeta = \zeta g_\zeta ,\quad \gamma_D^+ g_\zeta = \psi.
\end{equation*}
The corresponding solution operator is 
\begin{equation*} 
P_+^V(\zeta):H^{1/2}(\cS)\rightarrow H^1(\cB_+), \quad \psi\mapsto g_\zeta, 
\end{equation*} 
and for 
$\zeta \not \in \sigma(B_+)$ the {\it Dirichlet-to-Neumann map} $\cD_+^V(\zeta)$ 
{\it associated to $-\Delta+V$ 
in} $\cB_+$ is defined by
\begin{equation*}
 \cD_+^V(\zeta):H^{1/2}(\cS)\rightarrow H^{-1/2}(\cS),\quad \psi\mapsto \gamma_N^+ P_+^V(\zeta)\psi.
\end{equation*}
Furthermore, for $\zeta' \not \in  [0,\infty)$ and $\xi\in H^{1/2}(\cS)$ there exists a unique solution $h_{\zeta'}\in H^1(\cB_-)$ of the boundary value problem
\begin{equation*}
 -\Delta h_{\zeta'} = \zeta' h_{\zeta'} ,\quad \gamma_D^- h_{\zeta'}=\xi.
\end{equation*}
As above the corresponding solution operator is 
\begin{equation*} 
P_-(\zeta'):H^{1/2}(\cS)\rightarrow H^1(\cB_-), \quad \xi\mapsto h_{\zeta'}, 
\end{equation*} 
and for $\zeta' \not \in  [0,\infty)$, the {\it Dirichlet-to-Neumann map} $\cD_-(\zeta')$ {\it associated 
to $-\Delta$ in} $\cB_-$ is defined by
\begin{equation*}
 \cD_-(\zeta'):H^{1/2}(\cS)\rightarrow H^{-1/2}(\cS),\quad \xi\mapsto \gamma_N^- P_-(\zeta')\xi.
\end{equation*}

One recalls that the Dirichlet-to-Neumann maps $\cD_+(z)$, $\cD_+^V(\zeta)$, and $\cD_-(\zeta')$ above are bounded operators from $H^{1/2}(\cS)$ to $H^{-1/2}(\cS)$.
Moreover, for $z\in\dC\backslash\dR$, each of the Dirichlet-to-Neumann maps is boundedly invertible and the same is true for the sums
\begin{equation*}
 \cD_+(z)+ \cD_-(z):H^{1/2}(\cS)\rightarrow H^{-1/2}(\cS),\quad z\in\dC\backslash\dR,
 \end{equation*}
 and
 \begin{equation*}
 \cD_+^V(z)+ \cD_-(z):H^{1/2}(\cS)\rightarrow H^{-1/2}(\cS),\quad z\in\dC\backslash\dR 
\end{equation*}
(this can be shown in the same way as in Step 2 of the proof of Theorem~\ref{dddthm}).
Hence, the operators
\begin{equation}\label{nvn1}
 \sN(z)=\imath \bigl(\cD_+(z)+ \cD_-(z)\bigr)^{-1}\,\widetilde\imath:L^2(\cS)\rightarrow L^2(\cS), 
 \quad z\in\dC\backslash\dR,
\end{equation}
and
\begin{equation}\label{nvn2}
 \sN_V(z)=\imath \bigl(\cD_+^V(z)+ \cD_-(z)\bigr)^{-1}\,\widetilde\imath:L^2(\cS)\rightarrow L^2(\cS), 
  \quad z\in\dC\backslash\dR, 
\end{equation}
are everywhere defined and bounded in $L^2(\cS)$. 

In the next theorem we obtain a representation for a spectral shift function for $\{A,B\}$ in \eqref{abva}--\eqref{abvb} via a decoupling technique 
and Theorem~\ref{mainssf2}.
The considerations in the beginning of Step 1 of the proof of Theorem~\ref{thmvv} are similar as in \cite[Section 5.2]{BMN15} and hence some details
are omitted.

\begin{theorem}\label{thmvv}
Let $n \in \bbN$, $n\geq 2$, and $k \in \bbN$, $k> (n-2)/4$, and suppose that $V\in L^\infty(\dR^n)$ is real-valued with support in the open ball $\cB_+$. 
In addition, let $\sN(z)$ and $\sN_V(z)$ be 
as in \eqref{nvn1}--\eqref{nvn2}, 
and denote the eigenvalue counting functions of the Dirichlet operators $A_+$ and $B_+$
in $L^2(\cB_+)$ by $N(\, \cdot \,,A_+)$ and $N(\, \cdot \,,B_+)$, respectively. 
Then the following assertions $(i)$ and $(ii)$ hold:
 \begin{itemize}
  \item [$(i)$] The difference of the $2k+1$th-powers of the resolvents of $A$ and $B$ is 
a trace class operator, that is,
\begin{equation*}
\big[(B  - z I_{L^2(\bbR^n)})^{-(2k+1)}-(A - z I_{L^2(\bbR^n)})^{-(2k+1)}\big] \in \sS_1\bigl(L^2(\dR^n)\bigr)
\end{equation*}
holds for all $z \in \rho(B)=\rho(A)\cap\rho(B)$. 
 \item [$(ii)$] For any orthonormal basis $(\varphi_j)_{j \in J}$ in $L^2(\cS)$ the function 
  \begin{align*}
   \xi(\lambda) 
   & =\sum_{j \in J} \lim_{\varepsilon\downarrow 0}\pi^{-1} \Bigl(\bigl(\Im\bigl( \log (\sN(\lambda+i\varepsilon)) - \log (\sN_V(\lambda + i \varepsilon))\bigr)\bigr)
   \varphi_j,\varphi_j\Bigr)_{L^2(\cS)} \\
   & \quad +
   N(\lambda,B_+) - N(\lambda,A_+)\, \text{ for a.e. $\lambda\in\dR$,}
  \end{align*} 
is a spectral shift function for the pair $\{A,B\}$ such that $\xi(\lambda)=0$ for 
$\lambda < \min(\sigma(B))\leq 0$ and the trace formula
\begin{equation*}
 \tr_{L^2(\bbR^n)}\bigl( (B - z I_{L^2(\bbR^n)})^{-(2k+1)}-(A - z I_{L^2(\bbR^n)})^{-(2k+1)}\bigr) 
 = - (2k+1) \int_\dR \frac{\xi(\lambda)\, d\lambda}{(\lambda - z)^{2k+2}}
\end{equation*}
is valid for all $z \in \rho(B)=\rho(A)\cap\rho(B)$.
\end{itemize}
\end{theorem}
\begin{proof}
Besides the self-adjoint operators $A=-\Delta$ and $B=-\Delta+V$ in \eqref{abva} and \eqref{abvb}, and the Dirichlet realizations $A_+=-\Delta$ and $B_+=-\Delta+V$
in $L^2(\cB_+)$ in \eqref{apbva} and \eqref{apbvb} we shall also make use of the Dirichlet realization $A_-$ of $-\Delta$ in $L^2(\cB_-)$ given by
\begin{equation}
 A_-=-\Delta,\quad \dom(A_-)=H^2(\cB_-)\cap H^1_0(\cB_-),
\end{equation}
as well as the orthogonal sums in $L^2(\dR^n)=L^2(\cB_+)\oplus L^2(\cB_-)$, 
\begin{align}\label{tgg} 
\begin{split} 
& A_D:=\begin{pmatrix} A_+ & 0 \\ 0 & A_-\end{pmatrix}\, \text{ and } \, B_D:=\begin{pmatrix} B_+ & 0 \\ 0 & A_-\end{pmatrix},   \\
 &  \dom(A_D)=\dom(B_D)=\bigl(H^2(\cB_+)\cap H^1_0(\cB_+)\bigr)\times\bigl(H^2(\cB_-)\cap H^1_0(\cB_-)\bigr).
\end{split} 
\end{align} 
For any orthonormal basis $(\varphi_j)_{j \in J}$ in $L^2(\cS)$
we shall first prove the representation 
\begin{equation}\label{xa}
   \xi_A(\lambda)=\sum_{j \in J} \lim_{\varepsilon\downarrow 0} \pi^{-1} 
   \bigl(\Im\bigl(\log (\sN(\lambda + i\varepsilon))\bigr)\varphi_j,\varphi_j\bigr)_{L^2(\cS)}
  \end{equation}
for a spectral shift function $\xi_A$ of the pair $\{A,A_D\}$ and the representation 
\begin{equation}\label{xb}
   \xi_B(\lambda)=\sum_{j \in J} \lim_{\varepsilon\downarrow 0} \pi^{-1} 
   \bigl(\Im\bigl(\log (\sN_V(\lambda + i\varepsilon))\bigr)\varphi_j,\varphi_j\bigr)_{L^2(\cS)}
  \end{equation}
for
a spectral shift function $\xi_B$ of the pair $\{B,B_D\}$. 

\vskip 0.2cm\noindent
{\it Step 1.} In this step we consider the operators $B$ and $B_D$ as self-adjoint extensions of the closed symmetric $S=B\cap B_D$, which is given by
\begin{equation}\label{sv}
S = - \Delta+V,   \quad  
\dom(S) = \bigl\{f\in H^2(\dR^n) \, \big| \, \gamma_D^+f_+=0=\gamma_D^-f_-\bigr\}. 
 \end{equation}
Furthermore, consider the operator
\begin{equation*}
T=-\Delta+V,   \quad  
\dom(T) = \left\{f=\begin{pmatrix} f_+\\ f_-\end{pmatrix}\in H^2(\cB_+)\times H^2(\cB_-) 
 \, \bigg| \, \gamma_D^+f_+=\gamma_D^-f_-\right\}, 
\end{equation*}
and set $\gamma_D f:= \gamma_D^+f_+=\gamma_D^-f_-$ for $f\in\dom(T)$.
It is easy to see with the help of Theorem~\ref{ratemal}, \eqref{green11}--\eqref{green12} and \eqref{tracemapelli} that $\{L^2(\cS),\Gamma_0,\Gamma_1\}$, where
\begin{equation}\label{qbtv}
 \Gamma_0 f=\widetilde{\imath^{-1}}(\gamma_N^+f_++\gamma_N^-f_-) \, 
 \text{ and } \, \Gamma_1 f=\imath\gamma_D f,\quad f\in\dom(T),
\end{equation}
is a quasi boundary triple for $T\subset S^*$ such that
\begin{equation*}
 B=T\upharpoonright\ker(\Gamma_0)\, \text{ and } \, B_D=T\upharpoonright\ker(\Gamma_1)
\end{equation*}
hold (cf.\ the proof of \cite[Theorem~5.1]{BMN15}). The corresponding Weyl function is 
\begin{equation}\label{mmv}
 M(z)\varphi =\imath\bigl(\cD_+^V(z)+ \cD_-(z)\bigr)^{-1}\widetilde\imath\varphi 
 = \sN_V(z)\varphi
\end{equation}
for all $z \in \rho(B)\cap\rho(B_D)$ and $\varphi\in\ran(\Gamma_0)$. Furthermore, the proof of \cite[Theorem~5.1]{BMN15} shows that
$M(z)$ and $M(z)^{-1}$ are bounded for all  $z \in \rho(B)\cap\rho(B_D)$ and one has $\overline{M(z)}=\sN_V(z)$.

One observes that $B$ corresponds to the densely defined, closed quadratic form 
$$\mathfrak b[f,g]=(\nabla f,\nabla g)_{(L^2(\dR^n))^n}+(Vf,g)_{L^2(\dR^n)},\quad \dom(\mathfrak b)=H^1(\dR^n),$$
and that $B_D$ 
corresponds to the densely defined closed quadratic form 
$$\mathfrak b_D[f,g]=(\nabla f,\nabla g)_{(L^2(\dR^n))^n}+(Vf,g)_{L^2(\dR^n)},\quad \dom(\mathfrak b_D)=H^1_0(\cB_+)\times H^1_0(\cB_-).$$
Since $H^1(\dR^n)\subset (H^1_0(\cB_+)\times H^1_0(\cB_-))$ this implies $\mathfrak b\leq\mathfrak b_D$ and yields the sign condition 
$(B - \zeta I_{L^2(\bbR^n)})^{-1}\geq (B_D - \zeta I_{L^2(\bbR^n)})^{-1}$
in the assumptions of Theorem~\ref{mainssf2} for all $\zeta < \min(\sigma(B))\leq \min(\sigma(B_D))$; see the beginning of Step 2 in the proof of Theorem~\ref{nrthmelli}. 

Next, we verify the $\sS_p$-conditions 
\begin{equation}\label{ddcq1}
\, \overline{\gamma(z)}^{(p)}\bigl( M(z)^{-1} \gamma({\ol z})^*\bigr)^{(q)}\in\sS_1\bigl(L^2(\dR^n)\bigr),\quad p+q=2k,
\end{equation}
\begin{equation}\label{ddcq2}
 \bigl( M(z)^{-1} \gamma({\ol z})^*\bigr)^{(q)}\overline{\gamma(z)}^{(p)}\in\sS_1\bigl(L^2(\cS)\bigr),\quad p+q=2k,
\end{equation}
and 
 \begin{equation}\label{ddcq3}
   \frac{d^j}{dz^j} \overline{M (z)}\in\sS_{(2k+1)/j}\bigl(L^2(\cS)\bigr),\quad j=1,\dots,2k+1,
 \end{equation}
for all $z \in \rho(B)\cap\rho(B_D)$ in the assumptions of Theorem~\ref{mainssf2}. 
For this we use the smoothing property
\begin{equation}\label{smsm}
 (B - z I_{L^2(\bbR^n)})^{-1}f \in H^{\ell+2}_\cO(\dR^n)\, \text{ for all } \, f\in H^{\ell}_\cO(\dR^n)  
 \, \text{ and $\ell \in \bbN_0$,}
\end{equation}
where $\cO$ is an open neighborhood of the sphere $\cS$ in $\dR^n$ such that 
$\text{\rm supp}\, (V) \cap\cO=\emptyset$ and 
$H^{\ell}_\cO(\dR^n) = \big\{f\in L^2(\dR^n) \, \big| \, f\upharpoonright_\cO\in H^{\ell}(\cO)\big\}$ 
(cf.\ \cite[Lemma 4.1\,$(i)$]{BLL13-IEOT}). 

It follows from \eqref{gstar} and the definition of $\Gamma_1$ that
\begin{equation}\label{yes2}
 \gamma({\ol z})^* f=\Gamma_1(B - z I_{L^2(\bbR^n)})^{-1}f 
 = \imath \, \gamma_D (B - z I_{L^2(\bbR^n)})^{-1}f
\end{equation}
and hence \eqref{gammad2} yields
\begin{equation}\label{schubidu}
 \bigl(\gamma({\ol z})^*\bigr)^{(q)} 
 = q! \, \gamma({\ol z})^*(B - z I_{L^2(\bbR^n)})^{-q} 
 = q! \, \imath \, \gamma_D(B - z I_{L^2(\bbR^n)})^{-(q+1)}.
\end{equation}
Since
\begin{equation*}
 \ran \bigl(\gamma_D(B - z I_{L^2(\bbR^n)})^{-(q+1)}\bigr)\subset H^{2q+3/2}(\cS)
 \end{equation*}
by \eqref{smsm} and \eqref{gamss}, and the operator $\gamma_D(B - z I_{L^2(\bbR^n)})^{-(q+1)}$ is bounded from $L^2(\dR^n)$ into $H^{1/2}(\cS)$ it follows from 
Lemma~\ref{usel} with $s=2q+(3/2)$ and $t=1/2$ that
\begin{equation*}
 \gamma_D(B - z I_{L^2(\bbR^n)})^{-(q+1)}\in\sS_r\bigl(L^2(\dR^n),H^{1/2}(\cS)\bigr).
 \quad r> (n-1)/(2q+1),
\end{equation*}
As $\imath:H^{1/2}(\cS)\rightarrow L^2(\cS)$ is an isomorphism one concludes from \eqref{schubidu} that 
\begin{equation} \label{11}
\bigl(\gamma({\ol z})^*\bigr)^{(q)} \in\sS_r\bigl(L^2(\dR^n),L^2(\cS)\bigr), 
\quad r> (n-1)/(2q+1),
\end{equation}
for all $z \in \rho(B)$ and $q \in \bbN_0$.
From this it is also clear that  
\begin{equation}\label{22}
 \overline{\gamma(z)}^{(p)}\in\sS_r\bigl(L^2(\cS),L^2(\dR^n)\bigr), 
 \quad r> (n-1)/(2p+1),
\end{equation}
for all $z \in \rho(B)$ and $p \in \bbN_0$. Furthermore, 
\begin{equation}\label{mqqq22}
 \frac{d^j}{dz^j} \overline{M(z)}=j! \, \gamma({\ol z})^*(B - z I_{L^2(\bbR^n)})^{-(j-1)}\overline{\gamma(z)}, \quad j \in \bbN, 
\end{equation}
by \eqref{gammad3}, and using \eqref{yes2} one obtains with the arguments above that
\begin{align*}
\gamma({\ol z})^*(B - z I_{L^2(\bbR^n)})^{-(j-1)} 
= \imath \gamma_D (B - z I_{L^2(\bbR^n)})^{-j} 
          \in \sS_r\bigl(L^2(\dR^n),L^2(\cS)\bigr),&  \\
    r>  (n-1)/(2j-1),& 
 \end{align*}
 for all $z \in \rho(B)$ and $j \in \bbN$. Together with \eqref{22} for $p=0$ one finds that \eqref{mqqq22} satisfies
\begin{equation}\label{33}
 \frac{d^j}{dz^j} \overline{M (z)} \in \sS_r\bigl(L^2(\cS)\bigr), 
 \quad r>  (n-1)/(2j),
\end{equation}
for all $z \in \rho(B)$ and $j \in \bbN$. The same arguments as in Step 3 of the proof Theorem~\ref{nrthmelli} show that 
\begin{equation}\label{44}
 \frac{d^j}{dz^j} \overline{M (z)}^{-1}\in\sS_r\bigl(L^2(\cS)\bigr), 
 \quad r> (n-1)/(2j),
\end{equation}
for all $z \in \rho(B)\cap\rho(B_D)$ and $j \in \bbN$. It follows from \eqref{11} and \eqref{44} that each summand in the right-hand side in
\begin{equation*}
\bigl(  M(z)^{-1} \gamma({\ol z})^* \bigr)^{(q)}= \sum_{\substack{p+m=q \\[0.2ex] p,m\geq 0}} \begin{pmatrix} q \\ p \end{pmatrix} 
 \bigl(\overline{M(z)}^{-1}\bigr)^{(p)} \bigl(\gamma({\ol z})^*\bigr)^{(m)} 
 \end{equation*}
belongs to  $\sS_r\big(L^2(\cS),L^2(\dR^n)\big)$ for $r> (n-1)/(2q+1)$, and hence one infers together with \eqref{22} that
\begin{align*}
& \overline{\gamma(z)}^{(p)}\bigl( M(z)^{-1} \gamma({\ol z})^*\bigr)^{(q)} 
\in\sS_r\bigl(L^2(\bbR^n)\bigr), \\
& r> (n-1)/[2(p+q)+2]= (n-1)/(4k+2), 
\end{align*}
and since $k> (n-2)/4$ by assumption, one has $1> (n-1)/(4k+2)$, implying the trace class condition~\eqref{ddcq1}. The same argument shows that \eqref{ddcq2} is satisfied.
Finally, \eqref{ddcq3} follows from \eqref{33} and the fact that $k> (n-2)/4$ implies
\begin{equation*}
 \frac{2k+1}{j}>\frac{n}{2j}>\frac{n-1}{2j}, \quad j=1,\dots,2k+1.
\end{equation*}

Hence, the assumptions in Theorem~\ref{mainssf2} are satisfied with $S$ in \eqref{sv}, the quasi boundary triple in \eqref{qbtv}, and the corresponding Weyl function in \eqref{mmv}. Thus, 
Theorem~\ref{mainssf2} yields
assertion $(i)$ with $A$ replaced by $B_D$ and for any orthonormal basis 
$(\varphi_j)_{j \in J}$ in $L^2(\cS)$ the function 
  \begin{equation*}
   \xi_B(\lambda) 
   =\sum_{j \in J} \lim_{\varepsilon\downarrow 0}\pi^{-1} \bigl(\Im\bigl( \log (\sN_V(\lambda+i\varepsilon)) \bigr)
   \varphi_j,\varphi_j\bigr)_{L^2(\cS)}  \, \text{ for a.e.~$\lambda \in \dR$}
  \end{equation*}
  in \eqref{xb}
is a spectral shift function for the pair $\{B,B_D\}$ and the trace formula
\begin{equation}\label{trb}
 \tr_{L^2(\bbR^n)}\bigl( (B_D - z I_{L^2(\bbR^n)})^{-(2k+1)}-(B - z I_{L^2(\bbR^n)})^{-(2k+1)}\bigr)
 = - (2k+1) \int_\dR \frac{\xi_B(\lambda)\, d\lambda}{(\lambda - z)^{2k+2}}
\end{equation}
is valid for all $z \in \rho(B)\cap\rho(B_D)$.

\vskip 0.2cm\noindent
{\it Step 2.} Now we complete the proof of Theorem~\ref{thmvv}. First, we note that the same arguments as in Step 1 with $V=0$ show that 
assertion $(i)$ in Theorem~\ref{thmvv} holds with $B$ replaced by $A_D$ and $\xi_A$ in \eqref{xa} is a spectral shift function for the pair $\{A,A_D\}$ 
such that 
\begin{equation}\label{tra}
 \tr_{L^2(\bbR^n)}\bigl( (A_D - z I_{L^2(\bbR^n)})^{-(2k+1)}-(A - z I_{L^2(\bbR^n)})^{-(2k+1)}\bigr) 
 = - (2k+1) \int_\dR \frac{\xi_A(\lambda)\, d\lambda}{(\lambda - z)^{2k+2}}
\end{equation}
holds for all $z \in \dC\backslash [0,\infty)$. The assumption $k> (n-2)/4$ 
implies $2k+1> n/2$ and hence 
\begin{equation}\label{apbpres1}
 (A_+ - z I_{L^2(\cB_+)})^{-(2k+1)}\in\sS_1\bigl(L^2(\cB_+)\bigr),\qquad z\in\rho(A_+),
\end{equation} 
and
\begin{equation}\label{apbpres2}
 (B_+ - \zeta I_{L^2(\cB_+)})^{-(2k+1)}\in\sS_1\bigl(L^2(\cB_+)\bigr),\qquad \zeta\in\rho(B_+),
\end{equation}
by standard Weyl asymptotics.
Furthermore, since the spectra of $A_+$ and $B_+$ are discrete and semibounded from below,  
it is well-known that 
\begin{equation}\label{xp}
 \xi_+ (\lambda) = N(\lambda,B_+) - N(\lambda,A_+), \quad \lambda \in \dR,
\end{equation}
is a spectral shift function for the pair $\{A_+,B_+\}$ (see, e.g., \cite[(3.28)]{BY92}). From \eqref{tgg} it is clear that $\xi_+$ is also a spectral shift
function for the pair $\{A_D,B_D\}$. Since
\begin{equation*} 
\big[(B_D - z I_{L^2(\bbR^n)})^{-(2k+1)}-(A_D - z I_{L^2(\bbR^n)})^{-(2k+1)}\big] \in \sS_1\bigl(L^2(\dR^n)\bigr)
\end{equation*}
by \eqref{apbpres1}--\eqref{apbpres2} and \eqref{tgg} one concludes that  
\begin{align}\label{trp}
\begin{split} 
& \tr_{L^2(\bbR^n)}\bigl( (B_D - z I_{L^2(\bbR^n)})^{-(2k+1)}-(A_D - z I_{L^2(\bbR^n)})^{-(2k+1)}\bigr)    \\
 & \quad = - (2k+1) \int_\dR \frac{\xi_+(t)\, d\lambda}{(\lambda - z)^{2k+2}}, 
 \quad z \in \rho(A_D)\cap\rho(B_D). 
 \end{split}
\end{align}
Hence, 
\begin{equation*}
 \xi(\lambda)=\xi_A(\lambda)-\xi_B(\lambda)+\xi_+(t)\, \text{ for a.e.~$\lambda \in \dR$}
\end{equation*}
is a spectral shift function for the pair $\{A,B\}$, and taking into account the specific form of $\xi_A$, $\xi_B$, and $\xi_+$, in \eqref{xa}, \eqref{xb}, and \eqref{xp}
and the trace formulas \eqref{trb}, \eqref{tra}, and \eqref{trp}, the assertions in Theorem~\ref{thmvv} follow.
\end{proof}

\begin{remark}
We note that the spectral shift function $\xi$ in Theorem~\ref{thmvv} is continuous for $\lambda > 0$ since $V\in L^\infty(\dR^n)$ 
is compactly supported (see, e.g., \cite[Theorem~9.1.20]{Y10}). 
On the other hand the spectral shift function $\xi_+$ of $\{A_+,B_+\}$ is a step function and hence the difference of the spectral shift functions
$\xi_A$ and $\xi_B$ of the pairs $\{A,A_D\}$ and $\{B,B_D\}$ cancel the discontinuities of $\xi_+$ 
for $\lambda > 0$.
\end{remark}

\section{Schr\"{o}dinger operators with $\delta$-potentials supported on hypersurfaces}\label{ap2sec}

The aim of this section is to determine a spectral shift function for the pair $\{H,H_{\delta,\alpha}\}$, where $H=-\Delta$ is
the usual self-adjoint Laplacian in $L^2(\bbR^n)$, and $H_{\delta,\alpha}=-\Delta-\alpha \delta_\cC$ is a self-adjoint Schr\"{o}dinger operator with $\delta$-potential 
of strength $\alpha$ supported on a  compact hypersurface $\cC$ in $\dR^n$ which splits $\dR^n$ in a bounded interior domain and an unbounded exterior domain. 
Throughout this section we shall assume that the following hypothesis holds.

\begin{hypothesis}\label{hypo6}
Let $n \in \bbN$, $n \geq 2$, and $\Omega_{\rm i}$ be a nonempty, open, bounded interior domain 
in $\dR^n$ with a smooth boundary $\partial\Omega_{\rm i}$ and let 
$\Omega_{\rm e}=\dR^n\backslash\overline{\Omega_{\rm i}}$ be the corresponding 
exterior domain. The common boundary of the interior domain $\Omega_{\rm i}$ and exterior domain $\Omega_{\rm e}$ will be denoted by $\cC=\partial\Omega_{\rm e}=\partial\Omega_{\rm i}$. Furthermore, 
let $\alpha\in C^1(\cC)$ be a real-valued function on the boundary $\cC$.
\end{hypothesis}

We consider the self-adjoint operators 
\begin{equation}\label{h}
 H f=-\Delta f,\quad \dom (H)=H^2(\dR^n),
\end{equation}
and 
\begin{equation}\label{h-delta}
\begin{split}
& H_{\delta,\alpha} f = - \Delta f,\\
& \dom (H_{\delta,\alpha})=\left\{f=\begin{pmatrix}f_{\rm i}\\[1mm] f_{\rm e}\end{pmatrix}\in H^2(\Omega_{\rm i})\times H^2(\Omega_{\rm e}) \, \bigg| \, 
 \begin{matrix}\gamma_D^{\rm i}f_{\rm i}=\gamma_D^{\rm e}f_{\rm e},\quad\quad\!\\[1mm]
\alpha\gamma_D^{\rm i} f_{\rm i}=\gamma_N^{\rm i} f_{\rm i}+\gamma_N^{\rm e}f_{\rm e}\end{matrix}\right\},
\end{split}
 \end{equation}
in  $L^2(\bbR^n)$.
Here $f_{\rm i}$ and $f_{\rm e}$ denote the restrictions of a function $f$ on $\dR^n$ onto $\Omega_{\rm i}$ and $\Omega_{\rm e}$, and 
$\gamma_D^{\rm i}$, $\gamma_D^{\rm e}$ and $\gamma_N^{\rm i}$, $\gamma_N^{\rm e}$ are the Dirichlet and Neumann trace operators on $H^2(\Omega_{\rm i})$ and 
$H^2(\Omega_{\rm e})$, respectively.  We note that 
$H_{\delta,\alpha}$ in \eqref{h-delta} coincides with the self-adjoint operator associated
to the quadratic form
\begin{equation*}
 \mathfrak h_{\delta,\alpha}[f,g]=(\nabla f,\nabla g)-\int_\Sigma \alpha(x) f(x) \overline{g(x)}\,d\sigma(x),\quad f,g\in H^1(\dR^n), 
\end{equation*}
see \cite[Proposition 3.7]{BLL13-AHP} and \cite{BEKS94} for more details. For $c\in\dR$ we shall also make use of the self-adjoint operator
\begin{equation}\label{h-deltac}
\begin{split}
& H_{\delta,c} f = - \Delta f,\\
& \dom (H_{\delta,c})  = \left\{f=\begin{pmatrix}f_{\rm i}\\[1mm] f_{\rm e}\end{pmatrix}\in H^2(\Omega_{\rm i})\times H^2(\Omega_{\rm e}) \, \bigg| \, 
 \begin{matrix}\gamma_D^{\rm i}f_{\rm i}=\gamma_D^{\rm e}f_{\rm e},\quad\quad\!\\[1mm]
c\gamma_D^{\rm i} f_{\rm i}=\gamma_N^{\rm i} f_{\rm i}+\gamma_N^{\rm e}f_{\rm e}\end{matrix}\right\}. 
\end{split}
\end{equation}

We define interior and exterior
Dirichlet-to-Neumann maps $\cD_{\rm i}(z)$ and $\cD_{\rm e}(\zeta)$ as operators in $L^2(\cC)$ for all $z,\zeta \in \dC\backslash [0,\infty)$
in a similar way as in Section~\ref{ap11sec}. One notes 
that for $\varphi,\psi\in H^1(\cC)$ and $z,\zeta \in \dC\backslash [0,\infty)$, the boundary value problems 
\begin{equation}\label{i-bvp}
 -\Delta f_{{\rm i},z} = z f_{{\rm i},z},\quad \gamma_D^{\rm i} f_{{\rm i},z}=\varphi,
\end{equation}
and 
\begin{equation}\label{e-bvp}
 -\Delta f_{{\rm e},\zeta} = \zeta f_{{\rm e},\zeta},\quad \gamma_D^{\rm e} f_{{\rm e},\zeta} = \psi,
\end{equation}
admit unique solutions $f_{{\rm i},z}\in H^{3/2}(\Omega_{\rm i})$ and $f_{{\rm e},\zeta}\in H^{3/2}(\Omega_{\rm e})$, respectively. 
The corresponding solution operators are denoted by
\begin{equation*}
 P_{\rm i}(z):L^2(\cC) \rightarrow L^2(\Omega_{\rm i}),\quad \varphi\mapsto f_{{\rm i},z},
\end{equation*}
and
\begin{equation*}
 P_{\rm e}(\zeta):L^2(\cC) \rightarrow L^2(\Omega_{\rm e}),\quad \psi\mapsto f_{{\rm e},\zeta}.
\end{equation*}
The {\it interior Dirichlet-to-Neumann map}
\begin{equation}\label{di}
 \cD_{\rm i}(z):L^2(\cC)\rightarrow L^2(\cC),\quad \varphi\mapsto \gamma_N^{\rm i} P_{\rm i}(z)\varphi,
\end{equation}
is defined on $\dom(\cD_{\rm i}(z))=H^1(\cC)$ and
maps Dirichlet boundary values $\gamma_D^{\rm i} f_{{\rm i},z}$ of the solutions 
$f_{{\rm i},z}\in H^{3/2}(\Omega_{\rm i})$ of \eqref{i-bvp} onto the corresponding Neumann
boundary values $\gamma_N^{\rm i} f_{{\rm i},z}$, and the {\it exterior Dirichlet-to-Neumann map}
\begin{equation}\label{de}
 \cD_{\rm e}(\zeta):L^2(\cC)\rightarrow L^2(\cC),\quad 
 \psi\mapsto \gamma_N^{\rm e} P_{\rm e}(\zeta)\psi,
\end{equation}
is defined on $\dom(\cD_{\rm e}(\zeta))=H^1(\cC)$ and
maps Dirichlet boundary values $\gamma_D^{\rm e} f_{{\rm e},\zeta}$ of the solutions 
$f_{{\rm e},\zeta}\in H^{3/2}(\Omega_{\rm e})$ of \eqref{e-bvp} onto the corresponding Neumann
boundary values $\gamma_N^{\rm e} f_{{\rm e},\zeta}$. The interior and exterior Dirichlet-to-Neumann maps are both closed unbounded operators in $L^2(\cC)$.
We note that $\cD_{\rm i}(z)$ and $\cD_{\rm e}(\zeta)$ coincide with the restrictions of the Dirichlet-to-Neumann maps $\cD_+(z)$ and $\cD_-(\zeta)$
in Section~\ref{ap11sec} onto $H^1(\cC)$ in the case $\Omega_{\rm i}=\cB_+$, $\Omega_{\rm e}=\cB_-$, and $\cC=\cS$.

In the next theorem a spectral shift function for the pair $\{H,H_{\delta,\alpha}\}$ is expressed in terms of the limits of the sum of the 
interior and exterior Dirichlet-to-Neumann map $\cD_{\rm i}(z)$ and $\cD_{\rm e}(z)$ and the function $\alpha$. It will turn out that the
operators $\cD_{\rm i}(z) +\cD_{\rm e}(z)$ are boundedly invertible for all $z\in\dC\backslash [0,\infty)$ and for our purposes
it is convenient to work with the function 
\begin{equation}\label{ee}
 z \mapsto\cE(z)=\bigl(\cD_{\rm i}(z) +\cD_{\rm e}(z)\bigr)^{-1},\quad z\in\dC\backslash[0,\infty),
\end{equation}
which coincides with the acoustic single-layer potential for the Helmholtz equation, that is,
\begin{equation*}
 (\cE(z)\varphi)(x)=\int_\cC G(z,x,y)\varphi(y)d\sigma(y),\quad x\in\cC,\,\,\varphi\in C^\infty(\cC),
\end{equation*}
where $G(z,\,\cdot\,,\,\cdot\,)$, $z\in\dC\backslash\dR$, represents the integral kernel of the resolvent of $H$ (cf.\ \cite[Chapter~6]{McL00} and \cite[Remark 3.3]{BLL13-AHP}). 
The function $\cE$ in \eqref{ee} plays a similar role as the closure of the Neumann-to-Dirichlet map $\cN$ in Theorem~\ref{nrthmelli}. 
We mention that the trace class property of the difference of the
$2k+1$th powers of the resolvents in the next theorem is known from \cite{BLL13-AHP} (see also \cite{BLL-Exner}).

\begin{theorem}\label{dddthm}
Assume Hypothesis~\ref{hypo6}, let $H$ and $H_{\delta,\alpha}$ be the self-adjoint operators in
\eqref{h} and \eqref{h-delta}, respectively, 
let $\cE(z)$ be defined as in \eqref{ee}, let $\alpha\in C^1(\cC)$ be a real-valued function, fix 
$c>0$ such that $\alpha(x)<c$ for all $x\in\cC$, and let $H_{\delta,c}$ be the self-adjoint operator in \eqref{h-deltac}. 
Then the following assertions $(i)$ and $(ii)$ hold for $k \in \bbN_0$ such that $k\geq (n-3)/4$:
 \begin{itemize}
  \item [$(i)$] The difference of the $2k+1$th-powers of the resolvents of $H$ and $H_{\delta,\alpha}$ is 
a trace class operator, that is,
\begin{equation*}
\big[(H_{\delta,\alpha} - z I_{L^2(\bbR^n)})^{-(2k+1)} 
- (H - z I_{L^2(\bbR^n)})^{-(2k+1)}\big] \in \sS_1\bigl(L^2(\dR^n)\bigr)
\end{equation*}
holds for all $z\in\rho(H_{\delta,\alpha})=\rho(H)\cap\rho(H_{\delta,\alpha})$. 
 \item [$(ii)$] For any orthonormal basis $(\varphi_j)_{j \in J}$ in $L^2(\cC)$ the function 
  \begin{align*}
   \xi(\lambda) 
   =\sum_{j \in J} \lim_{\varepsilon\downarrow 0}\pi^{-1} \Bigl(\bigl(\Im\bigl( \log (\cM_\alpha(\lambda+i\varepsilon)) - \log (\cM_0(\lambda + i \varepsilon))\bigr)\bigr)
   \varphi_j,\varphi_j\Bigr)_{L^2(\cC)} \\
   \text{for a.e. $\lambda\in\dR$,}
  \end{align*}
is a spectral shift function for the pair $\{H,H_{\delta,\alpha}\}$ such that $\xi(\lambda)=0$ for 
$\lambda < \min(\sigma(H_{\delta,c}))$ and the trace formula
\begin{align*}
& \tr_{L^2(\bbR^n)}\bigl( (H_{\delta,\alpha} - z I_{L^2(\bbR^n)})^{-(2k+1)} 
 - (H - z I_{L^2(\bbR^n)})^{-(2k+1)}\bigr)   \\ 
& \quad = - (2k+1) \int_\dR \frac{\xi(\lambda)\, d\lambda}{(\lambda - z)^{2k+2}} 
 \end{align*}
is valid for all $z \in \rho(H_{\delta,\alpha})=\rho(H)\cap\rho(H_{\delta,\alpha})$.
\end{itemize}
\end{theorem}

\begin{proof}
The structure and underlying idea of the proof of Theorem~\ref{dddthm} is the same as in the proof of Theorem~\ref{nrthmelli}. In the first two steps
a suitable quasi boundary triple and its Weyl function are constructed; in these steps all details are provided. In the third step it is shown that 
the assumptions in Theorem~\ref{mainssf2} are satisfied. Since some of the main arguments are the same as in Step 3 of the proof of Theorem~\ref{nrthmelli}
not all details are repeated.

{\it Step 1.}
Since $c-\alpha(x)\not=0$ for all $x\in\cC$ by assumption, the closed symmetric operator $S=H_{\delta,c}\cap H_{\delta,\alpha}$ is given by
\begin{equation}\label{ssrr} 
Sf=-\Delta f, \quad \dom (S) = \bigl\{f\in H^2(\dR^n) \, \big| \, 
\gamma_D^{\rm i}f_{\rm i}=\gamma_D^{\rm e}f_{\rm e}=0\bigr\}. 
 \end{equation}
In this step we show that the operator
\begin{equation}\label{ttrr}
T = - \Delta,    \quad  
\dom (T) = \left\{f=\begin{pmatrix}f_{\rm i}\\f_{\rm e}\end{pmatrix}\in H^2(\Omega_{\rm i})\times H^2(\Omega_{\rm e}) \, \bigg| \, \gamma_D^{\rm i}f_{\rm i}=\gamma_D^{\rm e}f_{\rm e}\right\}, 
 \end{equation}
satisfies $\overline T=S^*$ and that $\{L^2(\cC),\Gamma_0,\Gamma_1\}$, where
\begin{equation}\label{qbtw1}
 \Gamma_0 f=c\gamma_D^{\rm i} f_{\rm i}-(\gamma_N^{\rm i}f_{\rm i}+\gamma_N^{\rm e}f_{\rm e}),\quad \dom(\Gamma_0)=\dom(T),
\end{equation}
and
\begin{equation}\label{qbtw2}
 \Gamma_1 f= (c-\alpha)^{-1}\bigl(\alpha\gamma_D^{\rm i} f_{\rm i}-(\gamma_N^{\rm i}f_{\rm i}+\gamma_N^{\rm e}f_{\rm e})\bigr),\quad 
 \dom(\Gamma_1)=\dom(T),
\end{equation} 
is a quasi boundary triple for $T\subset S^*$ such that
\begin{equation}\label{hhd}
 H_{\delta,c}=T\upharpoonright\ker(\Gamma_0)\, \text{ and } \, 
 H_{\delta,\alpha}=T\upharpoonright\ker(\Gamma_1).
\end{equation}

For the proof of this fact we make use of Theorem~\ref{ratemal} and verify next that 
assumptions $(i)$--$(iii)$  in Theorem~\ref{ratemal} are satisfied with the above choice 
of $S$, $T$ and boundary maps $\Gamma_0$ and $\Gamma_1$. For $f,g\in\dom (T)$ one computes 
\begin{equation*}
 \begin{split}
&(\Gamma_1 f,\Gamma_0 g)_{L^2(\cC)}-(\Gamma_0 f,\Gamma_1 g)_{L^2(\cC)}\\
&\quad =\Bigl((c-\alpha)^{-1}\bigl(\alpha\gamma_D^{\rm i} f_{\rm i}-(\gamma_N^{\rm i}f_{\rm i}+\gamma_N^{\rm e}f_{\rm e})\bigr), 
                c\gamma_D^{\rm i} g_{\rm i}-(\gamma_N^{\rm i}g_{\rm i}+\gamma_N^{\rm e}g_{\rm e})\Bigr)_{L^2(\cC)}\\
&\qquad - \Bigl(c\gamma_D^{\rm i} f_{\rm i}-(\gamma_N^{\rm i}f_{\rm i}+\gamma_N^{\rm e}f_{\rm e}),
                     (c-\alpha)^{-1}\bigl(\alpha\gamma_D^{\rm i} g_{\rm i}-(\gamma_N^{\rm i}g_{\rm i}+\gamma_N^{\rm e}g_{\rm e})\bigr)\Bigr)_{L^2(\cC)}\\
&\quad =-\bigl((c-\alpha)^{-1} \alpha\gamma_D^{\rm i} f_{\rm i}, \gamma_N^{\rm i}g_{\rm i}+\gamma_N^{\rm e}g_{\rm e}\bigr)_{L^2(\cC)}
         -\bigl(\gamma_N^{\rm i}f_{\rm i}+\gamma_N^{\rm e}f_{\rm e}, (c-\alpha)^{-1}c\gamma_D^{\rm i} g_{\rm i}\bigr)_{L^2(\cC)}\\
&\qquad +\bigl((c-\alpha)^{-1} c\gamma_D^{\rm i} f_{\rm i}, \gamma_N^{\rm i}g_{\rm i}+\gamma_N^{\rm e}g_{\rm e}\bigr)_{L^2(\cC)}
         +\bigl(\gamma_N^{\rm i}f_{\rm i}+\gamma_N^{\rm e}f_{\rm e}, (c-\alpha)^{-1}\alpha\gamma_D^{\rm i} g_{\rm i}\bigr)_{L^2(\cC)}\\                    
&\quad =\bigl(\gamma_D^{\rm i} f_{\rm i},\gamma_N^{\rm i}g_{\rm i}+\gamma_N^{\rm e}g_{\rm e}\bigr)_{L^2(\cC)} -\bigl(\gamma_N^{\rm i}f_{\rm i}+\gamma_N^{\rm e}f_{\rm e},\gamma_D^{\rm i} g_{\rm i}\bigr)_{L^2(\cC)},
 \end{split}
\end{equation*}
and on the other hand, Green's identity and $\gamma_D^{\rm i}f_{\rm i}=\gamma_D^{\rm e}f_{\rm e}$ and $\gamma_D^{\rm i}g_{\rm i}=\gamma_D^{\rm e}g_{\rm e}$ yield 
\begin{equation*}
 \begin{split}
&(Tf,g)_{L^2(\dR^n)}-(f,Tg)_{L^2(\dR^n)}\\
&\,\,=(-\Delta f_{\rm i},g_{\rm i})_{L^2(\Omega_{\rm i})}-(f_{\rm i},-\Delta g_{\rm i})_{L^2(\Omega_{\rm i})}+(-\Delta f_{\rm e},g_{\rm e})_{L^2(\Omega_{\rm e})}
-(f_{\rm e},-\Delta g_{\rm e})_{L^2(\Omega_{\rm e})}\\
&\,\,=(\gamma_D^{\rm i}f_{\rm i},\gamma_N^{\rm i}g_{\rm i})_{L^2(\cC)}-(\gamma_N^{\rm i}f_{\rm i},\gamma_D^{\rm i}g_{\rm i})_{L^2(\cC)}
 +(\gamma_D^{\rm e}f_{\rm e},\gamma_N^{\rm e}g_{\rm e})_{L^2(\cC)}-(\gamma_N^{\rm e}f_{\rm e},\gamma_D^{\rm e}g_{\rm e})_{L^2(\cC)}\\
 &\,\,=\bigl(\gamma_D^{\rm i} f_{\rm i},\gamma_N^{\rm i}g_{\rm i}+\gamma_N^{\rm e}g_{\rm e}\bigr)_{L^2(\cC)} 
  -\bigl(\gamma_N^{\rm i}f_{\rm i}+\gamma_N^{\rm e}f_{\rm e},\gamma_D^{\rm i} g_{\rm i}\bigr)_{L^2(\cC)},
\end{split}
\end{equation*}
and hence condition $(i)$ in Theorem~\ref{ratemal} holds. In order to show that $\ran (\Gamma_0,\Gamma_1)^\top$ is dense in $L^2(\cC)$ we recall that 
\begin{equation*}
 \begin{pmatrix}
  \gamma_D^{\rm i}\\[1mm] \gamma_N^{\rm i}
 \end{pmatrix}:H^2(\Omega_{\rm i})\rightarrow H^{3/2}(\cC)\times H^{1/2}(\cC)
\end{equation*}
and
\begin{equation*}
 \begin{pmatrix}
  \gamma_D^{\rm e}\\[1mm] \gamma_N^{\rm e}
 \end{pmatrix}:H^2(\Omega_{\rm e})\rightarrow H^{3/2}(\cC)\times H^{1/2}(\cC)
\end{equation*}
are surjective mappings. It follows that also the mapping
\begin{equation}\label{tmap}
 \begin{pmatrix}
  \gamma_D^{\rm i} \\[1mm] \gamma_N^{\rm i}+\gamma_N^{\rm e}
 \end{pmatrix}:\dom(T) \rightarrow H^{3/2}(\cC)\times H^{1/2}(\cC)
\end{equation}
is surjective, and since the $2\times 2$-block operator matrix
\begin{equation*}
 \Theta:=\begin{pmatrix} c I_{L^2(\cC)} & -I_{L^2(\cC)} \\[1mm]  \alpha (c-\alpha)^{-1} I_{L^2(\cC)}
 &  - (c-\alpha)^{-1} I_{L^2(\cC)} \end{pmatrix} 
\end{equation*}
is an isomorphism in $L^2(\cC)\times L^2(\cC)$, it follows that the range of the mapping, 
\begin{equation*}
\begin{pmatrix}\Gamma_0 \\ \Gamma_1\end{pmatrix}= \Theta  \begin{pmatrix}
  \gamma_D^{\rm i} \\[1mm] \gamma_N^{\rm i}+\gamma_N^{\rm e}
 \end{pmatrix}  :\dom(T)\rightarrow  L^2(\cC)\times L^2(\cC), 
\end{equation*}
is dense. Furthermore, as $C_0^\infty(\Omega_{\rm i})\times C_0^\infty(\Omega_{\rm e})$ is contained in $\ker(\Gamma_0)\cap\ker(\Gamma_1)$, it is clear that 
$\ker(\Gamma_0)\cap\ker(\Gamma_1)$ is dense in $L^2(\dR^n)$. Hence one concludes that 
condition $(ii)$ in Theorem~\ref{ratemal} is satisfied. Condition $(iii)$ in Theorem~\ref{ratemal} is satisfied since \eqref{hhd} holds by construction
and $H_{\delta,c}$ in \eqref{h-deltac} is self-adjoint. Thus, Theorem~\ref{ratemal}
implies that the closed symmetric operator 
\begin{equation}\label{tg}
 T\upharpoonright\bigl(\ker(\Gamma_0)\cap\ker(\Gamma_1)\bigr)=H_{\delta,c}\cap H_{\delta,\alpha}=S
\end{equation}
is densely defined, its adjoint coincides with $\overline T$, and 
$\{L^2(\cC),\Gamma_0,\Gamma_1\}$ is a quasi boundary triple for $T\subset S^*$ such that \eqref{hhd} holds.

\vskip 0.2cm\noindent
{\it Step 2.} In this step we prove that for $z \in \rho(H_{\delta,c})\cap\rho(H)$ 
the Weyl function corresponding to the quasi boundary triple 
$\{L^2(\cC),\Gamma_0,\Gamma_1\}$ is given by 
\begin{equation}\label{mddf}
\begin{split}
& M(z)= (c-\alpha)^{-1}\bigl(\alpha\cE_{1/2}(z)- I_{L^2(\cC)}\bigr)\bigl(c \cE_{1/2}(z) 
- I_{L^2(\cC)}\bigr)^{-1},\\
& \dom(M(z))=H^{1/2}(\cC),
\end{split}
\end{equation}
where $\cE_{1/2}(z)$ denotes the restriction of the operator $\cE(z)$ in \eqref{ee} onto $H^{1/2}(\cC)$, and we verify 
that $M(z_1)$ and $M(z_2)^{-1}$ are bounded for some $z_1,z_2\in\dC\backslash\dR$. 

It will first be shown that the operator $\cE(z)$ and its restriction $\cE_{1/2}(z)$ are well-defined 
for all $z \in \rho(H)=\dC\backslash [0,\infty)$. 
For this fix $z\in\dC\backslash [0,\infty)$, and let
\begin{equation}\label{flt}
 f_z =\begin{pmatrix} f_{{\rm i},z} \\ f_{{\rm e},z}\end{pmatrix} \in H^{3/2}(\Omega_{\rm i})\times H^{3/2}(\Omega_{\rm e})
\end{equation}
such that $\gamma_D^{\rm i}f_{{\rm i},z}=\gamma_D^{\rm e}f_{{\rm e},z}$, and
\begin{equation*}
-\Delta f_{{\rm i},z} = z f_{{\rm i},z}\, \text{ and } \, 
-\Delta f_{{\rm e},z} = z f_{{\rm e},z}.  
\end{equation*}
From the definition of $\cD_{\rm i}(z)$ and $\cD_{\rm e}(z)$ in \eqref{di} and \eqref{de} one concludes 
that 
\begin{equation}\label{useit2}
\begin{split}
 \bigl(\cD_{\rm i}(z)+\cD_{\rm e}(z)\bigr)\gamma_D^{\rm i}f_{{\rm i},z}&=\cD_{\rm i}(z)\gamma_D^{\rm i}f_{{\rm i},z}
 +\cD_{\rm e}(z)\gamma_D^{\rm e}f_{{\rm e},z}\\
 &=\gamma_N^{\rm i}f_{{\rm i},z}+\gamma_N^{\rm e}f_{{\rm e},z}.
 \end{split}
 \end{equation}
This also proves that $\cD_{\rm i}(z)+\cD_{\rm e}(z)$ is invertible for $z\in\dC\backslash [0,\infty)$.
In fact, otherwise there would exist a 
function $f_z =(f_{{\rm i},z}, f_{{\rm e},z})^\top\not=0$ as in \eqref{flt} which would satisfy both conditions
\begin{equation}\label{plkoplko}
 \gamma_D^{\rm i}f_{{\rm i},z}=\gamma_D^{\rm e}f_{{\rm e},z} \, 
 \text{ and } \,\gamma_N^{\rm i}f_{{\rm i},z}+\gamma_N^{\rm e}f_{{\rm e},z}=0,
\end{equation}
and hence for all $h\in\dom(H)=H^2(\dR^n)$,  
Green's identity together with the conditions \eqref{plkoplko} would imply 
\begin{equation}\label{fvb}
\begin{split}
& (H h,f_z )_{L^2(\dR^n)}-(h, z f_z )_{L^2(\dR^n)}\\
&\quad =(-\Delta h_{\rm i},f_{{\rm i},z})_{L^2(\Omega_{\rm i})}-(h_{\rm i},-\Delta f_{{\rm i},z})_{L^2(\Omega_{\rm i})}\\
&\qquad
            + (-\Delta h_{\rm e},f_{{\rm e},z})_{L^2(\Omega_{\rm e})}-(h_{\rm e},-\Delta f_{{\rm e},z})_{L^2(\Omega_{\rm e})}\\
&\quad =(\gamma_D^{\rm i}h_{\rm i},\gamma_N^{\rm i}f_{{\rm i},z})_{L^2(\cC)} 
- (\gamma_N^{\rm i}h_{\rm i},\gamma_D^{\rm i}f_{{\rm i},z})_{L^2(\cC)}\\       
&\qquad  +(\gamma_D^{\rm e}h_{\rm e},\gamma_N^{\rm e}f_{{\rm e},z})_{L^2(\cC)} 
- (\gamma_N^{\rm e}h_{\rm e},\gamma_D^{\rm e}f_{{\rm e},z})_{L^2(\cC)}  \\
&\quad =0,
\end{split}
\end{equation}
that is, $f_z \in\dom(H)$ 
and $H f_z = z f_z $; a contradiction since $z \in \rho(H)$. Hence,  
$$\ker\bigl(\cD_{\rm i}(z)+\cD_{\rm e}(z)\bigr)=\{0\},\quad z\in\dC\backslash [0,\infty),$$
and if we denote the restrictions of $\cD_{\rm i}(z)$ and $\cD_{\rm e}(z)$ onto $H^{3/2}(\cC)$ by $\cD_{{\rm i},3/2}(z)$
and $\cD_{{\rm e},3/2}(z)$, respectively, then also $\ker(\cD_{{\rm i},3/2}(z)+\cD_{{\rm e},3/2}(z))=\{0\}$ for $z\in\dC\backslash [0,\infty)$. Thus, we have shown that
$\cE(z)$ and its restriction $\cE_{1/2}(z)$ 
are well-defined for all $z \in \rho(H)=\dC\backslash [0,\infty)$. 

Furthermore, if the function $f_z $ in \eqref{flt} belongs to 
$H^2(\Omega_{\rm i})\times H^2(\Omega_{\rm e})$, that is, $f_z \in\ker(T - z I_{L^2(\bbR^n)})$, then
$\gamma_D^{\rm i}f_{{\rm i},z}=\gamma_D^{\rm e}f_{{\rm e},z}\in H^{3/2}(\cC)$ and hence besides \eqref{useit2} one also has 
\begin{equation}\label{useit3}
 \bigl(\cD_{{\rm i},3/2}(z)+\cD_{{\rm e},3/2}(z)\bigr)\gamma_D^{\rm i}f_{{\rm i},z}=\gamma_N^{\rm i}f_{{\rm i},z}+\gamma_N^{\rm e}f_{{\rm e},z}\in H^{1/2}(\cC).
 \end{equation} 
One concludes from \eqref{useit3} that
\begin{equation}\label{useit4}
\cE_{1/2}(z)
\bigl(\gamma_N^{\rm i}f_{{\rm i},z}+\gamma_N^{\rm e}f_{{\rm e},z}\bigr)=\gamma_D^{\rm i}f_{{\rm i},z},
\end{equation}
and from \eqref{qbtw1} one then obtains 
\begin{equation}\label{ggb}
\begin{split}
\bigl(c\cE_{1/2}(z)-I_{L^2(\cC)}\bigr)\bigl(\gamma_N^{\rm i}f_{{\rm i},z}+\gamma_N^{\rm e}f_{{\rm e},z}\bigr)
&=c\gamma_D^{\rm i}f_{{\rm i},z}-\bigl(\gamma_N^{\rm i}f_{{\rm i},z} 
+ \gamma_N^{\rm e}f_{{\rm e},z}\bigr)\\
&=\Gamma_0f_z,  
\end{split}
\end{equation}
and
\begin{equation}\label{ggb2}
\bigl(\alpha\cE_{1/2}(z)-I_{L^2(\cC)}\bigr)\bigl(\gamma_N^{\rm i}f_{{\rm i},z}+\gamma_N^{\rm e}f_{{\rm e},z}\bigr)
=\alpha\gamma_D^{\rm i}f_{{\rm i},z}-\bigl(\gamma_N^{\rm i}f_{{\rm i},z}+\gamma_N^{\rm e}f_{{\rm e},z}\bigr).
\end{equation}
For $z \in \rho(H_{\delta,c})\cap\rho(H)$ one verifies $\ker(c\cE_{1/2}(z)-I_{L^2(\cC)})=\{0\}$ with the help of \eqref{ggb}. Then 
\eqref{qbtw1} and \eqref{tmap} yield
\begin{equation*}
\ran\bigl(c\cE_{1/2}(z)-I_{L^2(\cC)}\bigr)=\ran(\Gamma_0)=H^{1/2}(\cC).
\end{equation*}
Thus, it follows from \eqref{ggb}, \eqref{ggb2}, and \eqref{qbtw2} that
\begin{equation*}
 \begin{split}
  & (c-\alpha)^{-1}\bigl(\alpha\cE_{1/2}(z)- I_{L^2(\cC)}\bigr)\bigl(c \cE_{1/2}(z) - I_{L^2(\cC)}\bigr)^{-1}\Gamma_0f_z \\
  &\quad = (c-\alpha)^{-1}\bigl(\alpha\cE_{1/2}(z)- I_{L^2(\cC)}\bigr)\bigl(\gamma_N^{\rm i}f_{{\rm i},z} 
  + \gamma_N^{\rm e}f_{{\rm e},z}\bigr)\\
  &\quad = (c-\alpha)^{-1}\Bigl(\alpha\gamma_D^{\rm i}f_{{\rm i},z} 
  - \bigl(\gamma_N^{\rm i}f_{{\rm i},z}+\gamma_N^{\rm e}f_{{\rm e},z}\bigr)\Bigr)\\
  &\quad =\Gamma_1 f_z 
 \end{split}
\end{equation*}
holds for all $z \in \rho(H_{\delta,c})\cap\rho(H)$. This proves that the Weyl function corresponding to the quasi boundary triple \eqref{qbtw1}--\eqref{qbtw2}
is given by \eqref{mddf}.

Next it will be shown that $M(z)$ and $M(z)^{-1}$ are bounded for $z\in\dC\backslash\dR$. For this it suffices to check that the operators
\begin{equation}\label{ees}
 \alpha\cE_{1/2}(z)- I_{L^2(\cC)} \, \text{ and } \, c\cE_{1/2}(z)- I_{L^2(\cC)}
\end{equation}
are bounded and have bounded inverses. The argument is the same for both operators in \eqref{ees} and hence we discuss $\alpha\cE_{1/2}(z)- I_{L^2(\cC)}$ only. One recalls that
\begin{equation*}
 \cD_{\rm i}(z) +\cD_{\rm e}(z),\quad z\in\dC\backslash \dR,
\end{equation*}
maps onto $L^2(\cC)$, is boundedly invertible, and its inverse $\cE(z)$ in \eqref{ee} is a compact operator in $L^2(\cC)$ with $\ran(\cE(z))=H^1(\cC)$ (cf.\  
\cite[Proposition 3.2\,$(iii)$]{BLL13-AHP}). Hence also the restriction $\cE_{1/2}(z)$ of $\cE(z)$ onto $H^{1/2}(\cC)$ is bounded. It follows that
$\alpha\cE_{1/2}(z)- I_{L^2(\cC)}$ is bounded, and its closure is given by
\begin{equation}\label{ovm}
\overline{\alpha \cE_{1/2}(z)- I_{L^2(\cC)}}=\alpha\cE(z)- I_{L^2(\cC)}\in\cL\bigl(L^2(\cC)\bigr), 
\quad z\in\dC\backslash\dR.
\end{equation}

In order to show that the inverse $(\alpha\cE_{1/2}(z)- I_{L^2(\cC)})^{-1}$ exists and is 
bounded for $z\in\dC\backslash \dR$ we first check that 
\begin{equation}\label{kern}
\ker\bigl(\alpha\cE(z)- I_{L^2(\cC)}\bigr)=\{0\},\quad z\in\dC\backslash\dR.
\end{equation}
In fact, assume that $z\in\dC\backslash\dR$ and $\varphi\in L^2(\cC)$
are such that $\alpha\cE(z)\varphi=\varphi$. 
It follows from  $\dom(\cE(z))=\ran(\cD_{\rm i}(z) +\cD_{\rm e}(z))=L^2(\cC)$ that there exists $\psi\in H^1(\cC)$ such that
\begin{equation}\label{zuzu}
 \varphi=\bigl(\cD_{\rm i}(z) +\cD_{\rm e}(z)\bigr)\psi, 
\end{equation}
and from \eqref{i-bvp}--\eqref{e-bvp} one concludes that there exists a unique 
\begin{equation}\label{flt22}
 f_z =\begin{pmatrix} f_{{\rm i},z} \\ f_{{\rm e},z}\end{pmatrix} \in H^{3/2}(\Omega_{\rm i})\times H^{3/2}(\Omega_{\rm e}) 
\end{equation}
such that
\begin{equation}\label{df1}
\gamma_D^{\rm i}f_{{\rm i},z}=\gamma_D^{\rm e}f_{{\rm e},z}=\psi,
\end{equation}
and
\begin{equation}\label{gls}
-\Delta f_{{\rm i},z} = z f_{{\rm i},z}\, \text{ and } \, 
-\Delta f_{{\rm e},z} = z f_{{\rm e},z}.  
\end{equation}
Since $\varphi=\alpha \cE(z)\varphi=\alpha\psi$ by \eqref{zuzu}, one obtains from \eqref{useit2}, \eqref{df1}, and \eqref{zuzu} that
\begin{equation}\label{bcc}
\begin{split}
 \gamma_N^{\rm i}f_{{\rm i},z}+\gamma_N^{\rm e}f_{{\rm e},z}
 &=\bigl(\cD_{\rm i}(z)+\cD_{\rm e}(z)\bigr)\gamma_D^{\rm i}f_{{\rm i},z}\\
 &=\bigl(\cD_{\rm i}(z)+\cD_{\rm e}(z)\bigr)\psi\\
 &=\varphi\\
 &=\alpha\psi\\
 &=\alpha\gamma_D^{\rm i}f_{{\rm i},z}.
\end{split}
\end{equation}
For $h=(h_{\rm i},h_{\rm e})^\top\in\dom(H_{\delta,\alpha})$ one has  
\begin{equation}\label{df2}
 \gamma_D^{\rm i}h_{\rm i}=\gamma_D^{\rm e}h_{\rm e}\, \text{ and } \, 
 \gamma_N^{\rm i}h_{\rm i}+\gamma_N^{\rm e}h_{\rm e} 
 =\alpha\gamma_D^{\rm i}h_{\rm i}, 
\end{equation}
and in a similar way as in \eqref{fvb}, Green's identity together with \eqref{df1}, \eqref{bcc}, and \eqref{df2} imply
\begin{equation*}
\begin{split}
& (H_{\delta,\alpha}h,f_z )_{L^2(\dR^n)}-(h,z f_z )_{L^2(\dR^n)}\\
&\quad =(-\Delta h_{\rm i},f_{{\rm i},z})_{L^2(\Omega_{\rm i})}-(h_{\rm i},-\Delta f_{{\rm i},z})_{L^2(\Omega_{\rm i})}\\
&\qquad
            + (-\Delta h_{\rm e},f_{{\rm e},z})_{L^2(\Omega_{\rm e})}-(h_{\rm e},-\Delta f_{{\rm e},z})_{L^2(\Omega_{\rm e})}\\
&\quad=(\gamma_D^{\rm i}h_{\rm i},\gamma_N^{\rm i}f_{{\rm i},z})_{L^2(\cC)}-(\gamma_N^{\rm i}h_{\rm i},\gamma_D^{\rm i}f_{{\rm i},z})_{L^2(\cC)}\\       
&\qquad +(\gamma_D^{\rm e}h_{\rm e},\gamma_N^{\rm e}f_{{\rm e},z})_{L^2(\cC)}-(\gamma_N^{\rm e}h_{\rm e},\gamma_D^{\rm e}f_{{\rm e},z})_{L^2(\cC)}  \\
&\quad =\bigl(\gamma_D^{\rm i}h_{\rm i},\gamma_N^{\rm i}f_{{\rm i},z}+\gamma_N^{\rm e}f_{{\rm e},z}\bigr)_{L^2(\cC)}
 -\bigl(\gamma_N^{\rm i}h_{\rm i}+\gamma_N^{\rm e}h_{\rm e},\gamma_D^{\rm i}f_{{\rm i},z}\bigr)_{L^2(\cC)}\\
&\quad =\bigl(\gamma_D^{\rm i}h_{\rm i},\alpha\gamma_D^{\rm i}f_{{\rm i},z}\bigr)_{L^2(\cC)}
 -\bigl(\alpha\gamma_D^{\rm i}h_{\rm i},\gamma_D^{\rm i}f_{{\rm i},z}\bigr)_{L^2(\cC)}\\ 
&\quad=0.
\end{split}
\end{equation*}
As $H_{\delta,\alpha}$ is self-adjoint one concludes that $f_z \in\dom (H_{\delta,\alpha})$ and $f_z \in\ker(H_{\delta,\alpha} - z I_{L^2(\bbR^n)})$. 
Since $z\in\dC\backslash\dR$, 
this yields $f_z =0$ and therefore, $\psi= \gamma_D^{\rm i}f_{{\rm i},z}=0$ and hence $\varphi=0$ by \eqref{zuzu}, implying \eqref{kern}.

Since $\cE(z)$ is a compact operator in $L^2(\cC)$ (see \cite[Proposition 3.2\,$(iii)$]{BLL13-AHP}) also $\alpha\cE(z)$ is compact 
and together with \eqref{kern} one concludes that 
$$(\alpha\cE(z)-I_{L^2(\cC)})^{-1}\in\cL\big(L^2(\cC)\big).$$ 
Hence also the restriction
\begin{equation}
 \bigl(\alpha\cE_{1/2}(z)-I_{L^2(\cC)}\bigr)^{-1}
\end{equation}
is a bounded operator in $L^2(\cC)$. Summing up, we have shown that the operators in \eqref{ees} are bounded and have bounded inverses for all $z\in\dC\backslash\dR$,
and hence the values $M(z)$ of the Weyl function in \eqref{mddf} are bounded and have bounded inverses for all $z\in\dC\backslash\dR$. 

\vskip 0.2cm\noindent
{\it Step 3.} Now we check that the operators $\{H_{\delta,c},H_{\delta,\alpha}\}$ and  Weyl function corresponding to the 
quasi boundary triple $\{L^2(\cC),\Gamma_0,\Gamma_1\}$ in Step 1 satisfy the assumptions of Theorem~\ref{mainssf2} for $n \in \bbN$, $n \geq 2$, and all $k\geq (n-3)/4$.

In fact, the sign condition \eqref{sign333} follows from the assumption $\alpha(x)<c$ and the fact that the closed quadratic forms 
$\mathfrak h_{\delta,\alpha}$ and $\mathfrak h_{\delta,c}$ associated to $H_{\delta,\alpha}$ and $H_{\delta,c}$ satisfy the inequality
$\mathfrak h_{\delta,c}\leq \mathfrak h_{\delta,\alpha}$ (cf.\ the beginning of Step 2 in the proof of Theorem~\ref{nrthmelli}). 
In order to verify the $\sS_p$-conditions 
\begin{align}
& \overline{\gamma(z)}^{(p)}\bigl( M(z)^{-1} \gamma({\ol z})^*\bigr)^{(q)}\in\sS_1\bigl(L^2(\dR^n)\bigr),\quad p+q=2k,  \label{ddc1} \\
& \bigl( M(z)^{-1} \gamma({\ol z})^*\bigr)^{(q)}\overline{\gamma(z)}^{(p)}\in\sS_1\bigl(L^2(\cC)\bigr),\quad p+q=2k,  \label{ddc2}
\end{align}
and 
 \begin{equation}\label{ddc3}
   \frac{d^j}{dz^j} \overline{M (z)}\in\sS_{(2k+1)/j}\bigl(L^2(\cC)\bigr),\quad j=1,\dots,2k+1,
 \end{equation}
for all $z \in \rho(H_{\delta,c})\cap\rho(H_{\delta,\alpha})$ in the assumptions of 
Theorem~\ref{mainssf2}, one first recalls the smoothing property
\begin{equation}\label{smoothi2}
 (H_{\delta,c} - z I_{L^2(\bbR^n)})^{-1} f \in H^{k+2}(\Omega_{\rm i})\times H^{k+2}(\Omega_{\rm e}), 
 \quad f\in H^k(\Omega_{\rm i})\times H^k(\Omega_{\rm e}), \; k \in \bbN_0, 
\end{equation}
of the resolvent of $H_{\delta,c}$, which follows, for instance, from \cite[Theorem~4.20]{McL00}. Next one observes that \eqref{gstar}, \eqref{qbtw2}, 
and the definition of $H_{\delta,c}$ imply in the same way as in \eqref{yes} that
\begin{equation}
\begin{split}
\gamma({\ol z})^*&=\Gamma_1(H_{\delta,c} - z I_{L^2(\bbR^n)})^{-1}f\\
&=(c-\alpha)^{-1}\bigl(\alpha\gamma_D^{\rm i} -(\gamma_N^{\rm i}+\gamma_N^{\rm e})\bigr)(H_{\delta,c} - z I_{L^2(\bbR^n)})^{-1}f\\
&=(c-\alpha)^{-1}\bigl(c\gamma_D^{\rm i} -(\gamma_N^{\rm i}+\gamma_N^{\rm e}) + (\alpha- c)\gamma_D^{\rm i}\bigr)(H_{\delta,c} - z I_{L^2(\bbR^n)})^{-1}f\\
&=-\gamma_D^{\rm i}(H_{\delta,c} - z I_{L^2(\bbR^n)})^{-1}f, 
\end{split}
\end{equation}
and hence \eqref{gammad2}, \eqref{smoothi2}, and Lemma~\ref{usel} yield
\begin{align*}
 \bigl(\gamma({\ol z})^*\bigr)^{(q)} 
 =-q! \, \gamma_D^{\rm i}(H_{\delta,c} - z I_{L^2(\bbR^n)})^{-(q+1)}\in\sS_r\bigl(L^2(\dR^n),L^2(\cC)\bigr),&   \\
r> (n-1)/[2q+(3/2)],&
\end{align*}
for all $z \in \rho(H_{\delta,c})$ and $q \in \bbN_0$ 
(cf.\ \cite[Lemma 3.1]{BLL-Exner} for the case $c=0$). One also has
\begin{equation*}
 \overline{\gamma(z)}^{(p)}\in\sS_r\bigl(L^2(\cC),L^2(\dR^n)\bigr),\quad r> (n-1)/[2p+(3/2)],
\end{equation*}
for all $z \in \rho(H_{\delta,c})$ and $q \in \bbN_0$. In addition, one verifies with the same arguments as in Step 3 of the proof of Theorem~\ref{nrthmelli} that also
\begin{equation*}
   \frac{d^j}{dz^j} \overline{M(z)}
   = j! \, \gamma({\ol z})^*(H_{\delta,c} - z I_{L^2(\bbR^n)})^{-(j-1)} 
   \overline{\gamma(z)}\in\sS_r\bigl(L^2(\cC)\bigr),\quad r> (n-1)/(2j+1),
 \end{equation*}
for all $z \in \rho(H_{\delta,c})$ and $j \in \bbN$ (cf.\ \cite[Lemma 3.1]{BLL-Exner} for the case $c=0$). Summing up, we see that the derivatives of the 
$\gamma$-field and Weyl function of the quasi boundary triple 
$\{L^2(\cC),\Gamma_0,\Gamma_1\}$ have the same $\sS_p$-properties as the $\gamma$-field and
Weyl function in Step 3 of the proof of Theorem~\ref{nrthmelli} (cf.\ \eqref{ass2w}, \eqref{ass1w}, and \eqref{ass3w}). Now the same arguments as in the proof 
of Theorem~\ref{nrthmelli} show that the conditions \eqref{ddc1}--\eqref{ddc3} are satisfied for all $z \in \rho(H_{\delta,c})\cap\rho(H_{\delta,\alpha})$, 
$p,q \in \bbN_0$, $p+q=2k$ and $k\geq (n-3)/4$.

Hence the assumptions in Theorem~\ref{mainssf2} are satisfied with $S$ in \eqref{ssrr}, the quasi boundary triple 
in \eqref{qbtw1}--\eqref{qbtw2}, and the corresponding $\gamma$-field and Weyl function. It follows that the 
difference of the $2k+1$-th powers of the resolvents of $H_{\delta,c}$ and $H_{\delta,\alpha}$ is a trace class operator and that
for any orthonormal basis $(\varphi_j)_{j \in J}$ in $L^2(\cC)$ the function 
  \begin{equation*}
   \xi_\alpha(\lambda) 
   =\sum_{j \in J} \lim_{\varepsilon\downarrow 0}\pi^{-1} \bigl(\Im\bigl( \log (\cM_\alpha(\lambda+i\varepsilon))\bigr)
   \varphi_j,\varphi_j\bigr)_{L^2(\cC)} \, \text{ for a.e.~$\lambda \in \dR$},
  \end{equation*} 
is a spectral shift function for the pair $\{H_{\delta,c},H_{\delta,\alpha}\}$ such that $\xi_\alpha(\lambda)=0$ for 
$\lambda < \min(\sigma(H_{\delta,c}))\leq\min(\sigma(H_{\delta,\alpha}))$. The above considerations remain valid in the special case $\alpha=0$ which corresponds
to the pair $\{H_{\delta,c},H\}$. Now the assertions in Theorem~\ref{dddthm} follow in the same way as in the proof of Theorem~\ref{nrthmelli} from
the remarks in the end of Section~\ref{ssfsec}.
\end{proof}

In space dimensions $n=2$ and $n=3$ one can choose $k=0$ in Theorem~\ref{dddthm} and obtains a result of the same type as in Corollary~\ref{mainthmcorchen2}.

\begin{corollary}\label{mainthmcorchen3}
Let the assumptions be as in Theorem~\ref{dddthm} and suppose that $n=2$ or $n=3$. 
Then the following assertions $(i)$--$(iii)$ hold: 
\begin{itemize}
  \item [$(i)$]
The difference of the resolvents of $H$ and $H_{\delta,\alpha}$ is 
a trace class operator, that is,
\begin{equation*}
 \big[(H_{\delta,\alpha} - z I_{L^2(\dR^n)})^{-1}-(H - z I_{L^2(\dR^n)})^{-1}\big] \in \sS_1\bigl(L^2(\dR^n)\bigr)
\end{equation*}
holds for all $z\in\rho(H_{\delta,\alpha})=\rho(H)\cap\rho(H_{\delta,\alpha})$. 
  \item [$(ii)$] $\Im(\log (\cM_\alpha(z)))\in\sS_1(L^2(\cC))$ and $\Im(\log (\cM_0(z)))\in\sS_1(L^2(\cC))$ for all $z\in\dC\backslash\dR$, and the limits 
  $$\Im\bigl(\log(\cM_\alpha(\lambda+i 0))\bigr):=\lim_{\varepsilon\downarrow 0}\Im\bigl(\log(\cM_\alpha(\lambda+i\varepsilon))\bigr)$$ 
  and
  $$\Im\bigl(\log(\cM_0(\lambda+i 0))\bigr):=\lim_{\varepsilon\downarrow 0}\Im\bigl(\log(\cM_0(\lambda+i\varepsilon))\bigr)$$ 
  exist for a.e.~$\lambda \in \dR$ in $\sS_1(L^2(\cC))$. 
  \item [$(iii)$] The function
  \begin{equation}
   \xi(\lambda)=\pi^{-1} \tr_{L^2(\cC)}\bigl(\Im\big(\log(\cM_\alpha(\lambda + i0))-\log(\cM_0(\lambda + i0)) \big)\bigr) \, \text{ for a.e.~$\lambda \in \dR$}, 
  \end{equation}
is a spectral shift function for the pair $\{H,H_{\delta,\alpha}\}$ such that $\xi(\lambda)=0$ for 
$\lambda < \min(\sigma(H_{\delta,c}))$ and the trace formula
\begin{equation*}
 \tr_{L^2(\dR^n)}\bigl( (H_{\delta,\alpha} - z I_{L^2(\dR^n)})^{-1} 
 - (H - z I_{L^2(\dR^n)})^{-1}\bigr) 
 = - (2k+1) \int_\dR \frac{\,\xi(\lambda)\, d\lambda}{(\lambda - z)^2}
\end{equation*}
is valid for all $z\in\rho(H_{\delta,\alpha})=\rho(H)\cap\rho(H_{\delta,\alpha})$.
\end{itemize}
\end{corollary}

In the special case $\alpha<0$, Theorem~\ref{dddthm} simplifies slightly since in that case the sign condition \eqref{sign333} in Theorem~\ref{mainssf2}
is satisfied by the pair $\{H,H_{\delta,\alpha}\}$. Hence it is not necessary to introduce the operator $H_{\delta,c}$ in \eqref{h-deltac} as a comparison
operator in the proof of Theorem~\ref{dddthm}. Instead, one considers the operators $S$ and $T$ in \eqref{ssrr} and \eqref{ttrr}, and defines the 
boundary maps by
\begin{equation*}
 \Gamma_0 f=-\gamma_N^{\rm i}f_{\rm i}-\gamma_N^{\rm e}f_{\rm e},\quad \dom(\Gamma_0)=\dom(T),
\end{equation*}
and
\begin{equation*}
 \Gamma_1 f=-\gamma_D^{\rm i} f_{\rm i}+\frac{1}{\alpha}(\gamma_N^{\rm i}f_{\rm i}+\gamma_N^{\rm e}f_{\rm e})\bigr),\quad 
 \dom(\Gamma_1)=\dom(T).
\end{equation*} 
In this case the corresponding Weyl function is given by
\begin{equation*}
 M(z)=\cE_{1/2}(z) - \alpha^{-1} I_{L^2(\cC)},\quad z\in\dC\backslash\dR,
\end{equation*}
and hence the next statement follows in the same way as Theorem~\ref{dddthm} from our abstract result Theorem~\ref{mainssf2}. 
We leave it to te reader to formulate a variant of Corollary~\ref{mainthmcorchen3} for the special cases $n=2$ and $n=3$.

\begin{theorem}\label{dddthm2}
Assume Hypothesis~\ref{hypo6}, let $H$ and $H_{\delta,\alpha}$ be the self-adjoint operators in
\eqref{h} and \eqref{h-delta}, respectively, 
let $\cE(z)$ be defined as in \eqref{ee}, and let $\alpha\in C^1(\cC)$ be a real-valued function such that $\alpha(x)<0$ for all $x\in\cC$. 
Then the following assertions $(i)$ and $(ii)$ hold for $k \in \bbN_0$ such that $k\geq (n-3)/4$:
 \begin{itemize}
  \item [$(i)$] The difference of the $2k+1$th-powers of the resolvents of $H$ and $H_{\delta,\alpha}$ is 
a trace class operator, that is,
\begin{equation*}
\big[(H_{\delta,\alpha} - z I_{L^2(\bbR^n)})^{-(2k+1)} 
- (H - z I_{L^2(\bbR^n)})^{-(2k+1)}\big] \in \sS_1\bigl(L^2(\dR^n)\bigr)
\end{equation*}
holds for all $z\in\rho(H_{\delta,\alpha})=\rho(H)\cap\rho(H_{\delta,\alpha})$. 
 \item [$(ii)$] For any orthonormal basis $(\varphi_j)_{j \in J}$ in $L^2(\cC)$ the function 
  \begin{equation*}
   \xi(\lambda) 
   =\sum_{j \in J} \lim_{\varepsilon\downarrow 0}\pi^{-1} \bigl(\Im\bigl( \log (\cE(t+i\varepsilon) - \alpha^{-1} I_{L^2(\cC)}) \bigr)
   \varphi_j,\varphi_j\bigr)_{L^2(\cC)} 
  \end{equation*}
  for a.e. $\lambda\in\dR$ 
is a spectral shift function for the pair $\{H,H_{\delta,\alpha}\}$ such that $\xi(\lambda)=0$ for 
$\lambda < 0$ and the trace formula
\begin{align*}
& \tr_{L^2(\bbR^n)}\bigl( (H_{\delta,\alpha} - z I_{L^2(\bbR^n)})^{-(2k+1)} 
 - (H - z I_{L^2(\bbR^n)})^{-(2k+1)}\bigr)    \\ 
 & \quad = - (2k+1) \int_\dR \frac{\xi(\lambda)\, d\lambda}{(\lambda - z)^{2k+2}} 
 \end{align*}
is valid for all $z\in\dC\backslash [0,\infty)$.
\end{itemize}
\end{theorem}

\vskip 0.8cm
\noindent {\bf Acknowledgments.}  
J. Behrndt is most grateful for the stimulating research stay and the hospitality at the 
Graduate School of Mathematical Sciences of the University of Tokyo from April to July 2016, where parts of this paper were written. 
The authors also wish to thank Hagen Neidhardt for fruitful discussions and helpful remarks.
This work is supported by International Relations and Mobility Programs of the TU Graz and the Austrian Science Fund (FWF), project P 25162-N26.


\end{document}